\documentclass[fontsize=11pt,twoside=false,draft=false]{scrartcl}
\usepackage{hyperref}
\hypersetup{plainpages=false,
            linkcolor=black,
            citecolor=black,
            urlcolor=black,
            colorlinks=true}
            
\usepackage[utf8]{inputenc}

\usepackage[automark]{scrpage2}

\usepackage[english]{babel}

\usepackage[intlimits]{amsmath}
\usepackage{amsfonts}
\usepackage{amstext}
\usepackage{amsopn}
\usepackage{url}

\usepackage{tikz}
\usetikzlibrary{calc}
\usetikzlibrary{patterns}
\tikzset{
  nw/.style={pattern=north west lines, line width=0.1pt, pattern color=#1},
  nw/.default=black
}
\tikzset{
  ne/.style={pattern=north east lines,line width=0.1pt, pattern color=#1},
  ne/.default=black
}

\usepackage{graphicx}
\usepackage{caption}
\usepackage{subfig}
\usepackage{wrapfig}
\usepackage{enumitem}

\tikzstyle{every picture}+=[font=\footnotesize]
\usepackage{amsmath,amssymb,mathtools,commath}
\usepackage{cases}
\usepackage[amsmath,hyperref,thmmarks]{ntheorem}
\numberwithin{equation}{section}
\numberwithin{figure}{section}
\numberwithin{table}{section}

\makeatletter
\theoremheaderfont{\bfseries\upshape}
\theoremseparator{.}
\newtheorem{theorem}{Theorem}[section]
\newtheorem{lemma}[theorem]{Lemma}

\newtheorem{example}[theorem]{Example}
\newtheorem{corollary}[theorem]{Corollary}
\newtheorem{lemmaAppendix}{Lemma}

\theoremsymbol{\ensuremath{\Box}}
\newtheorem{corollaryAppendix}{Corollary}
\theoremsymbol{}
\theorembodyfont{\upshape}
\newtheorem{definition}[theorem]{Definition}

\theoremseparator{}

\theoremstyle{break}

\theoremstyle{plain}
\theoremheaderfont{\itshape}
\newtheorem{remark}[theorem]{Remark}

\theoremstyle{plain}
\theoremsymbol{\ensuremath{\Box}}
\theoremseparator{.}
\newtheorem*{proof}{Proof}
\theoremseparator{}
\newtheorem*{proofof}{Proof of}

\usepackage[ruled,vlined]{algorithm2e}
\usepackage{algorithmic}

\usepackage{dirtree}
\usepackage{zref}

\usepackage{booktabs}

\let\oldbullet\bullet
  \newlength{\raisebulletlen}
  \setbox1=\hbox{$\bullet$}\setbox2=\hbox{\tiny$\bullet$}
  \renewcommand\bullet{\raisebox{\raisebulletlen}{\,\tiny$\oldbullet$}\,}

\tikzset{slopetriangle/.style={
  bottom color=black!20,
  middle color=black!5,
  top color=white,
  draw=black
}}

\def\Xint#1{\mathchoice
{\XXint\displaystyle\textstyle{#1}}%
{\XXint\textstyle\scriptstyle{#1}}%
{\XXint\scriptstyle\scriptscriptstyle{#1}}%
{\XXint\scriptscriptstyle\scriptscriptstyle{#1}}%
\!\int}
\def\XXint#1#2#3{{\setbox0=\hbox{$#1{#2#3}{\int}$}
\vcenter{\hbox{$#2#3$}}\kern-.5\wd0}}
\newcommand{\intmean}{\Xint-}

\newcommand{\SPnew}[1]{{#1}}

\usepackage{a4wide}
\usepackage[capitalise, english]{cleveref}
\usepackage{pgfplots}
\usepackage{placeins}
\usepackage{enumitem}
\allowdisplaybreaks

\title{How to prove the discrete reliability for nonconforming finite element methods}
\author{
Carsten Carstensen\footnote{Department of Mathematics, Humboldt-Universit\"at zu Berlin, Unter den Linden 6, 10099 Berlin, Germany. \href{mailto:cc@math.hu-berlin.de}{cc@math.hu-berlin.de} and \href{mailto:puttkams@math.hu-berlin.de}{puttkams@math.hu-berlin.de}}
\and
 Sophie Puttkammer$^*$
}
\date{}

\begin{document}

\maketitle

\begin{abstract}
Optimal convergence rates of adaptive finite element methods are well understood in terms of
the axioms of adaptivity. One key ingredient is the discrete reliability of a residual-based a posteriori error estimator,  
which controls the error of two discrete finite element solutions based on two nested triangulations.
In the error analysis of nonconforming finite element methods, like the Crouzeix-Raviart or Morley  finite element schemes, the difference
of the piecewise derivatives of discontinuous approximations to the distributional gradients of global 
Sobolev functions plays a dominant role and is the object of this paper. 
The nonconforming interpolation operator, which comes natural with the
definition of the aforementioned nonconforming finite element in the sense of Ciarlet, allows for stability and approximation properties that enable direct proofs of the reliability for the residual that monitors the equilibrium condition. 
The novel approach of this paper is the suggestion of a right-inverse of this interpolation operator in conforming
piecewise polynomials to design a nonconforming approximation of a given coarse-grid approximation on a 
refined triangulation. The results of this paper allow for simple proofs of the discrete reliability in any space dimension and multiply connected
domains on general shape-regular triangulations beyond newest-vertex bisection of simplices. 
Particular attention is on optimal constants in some standard discrete estimates listed in the appendices.
\end{abstract}

\section{Introduction }
\subsection{Motivation}
The nonconforming finite element schemes are a subtile but important part of the finite element practice not exclusively in computational fluid dynamics \cite{BS08, Bra13,BBF13}, but also with benefits for guaranteed lower bounds of eigenvalues \cite{CGal14,CGed14}, lower bounds for energies e.g. in the obstacle problem \cite{CK17a}, 
 or guaranteed convergence for a convex energy density despite the presence of the Lavrentiev phenomenon \cite{OP11}.
 Prominent examples are Crouzeix-Raviart \cite{CR73} and Morley \cite{Mor68} finite elements illustrated in \cref{fig:FEs}.a and d.
 
The discrete reliability is one key-property in the overall analysis of optimal convergence rates in adaptive mesh-refining algorithms and one axiom in \cite{CFPP14,CR16}. Its proof is a challenge in the nonconforming setting since even given an admissible refinement $\widehat{\mathcal{T}}$ of an regular triangulation $\mathcal{T}$ the associated finite element spaces are non-nested $V(\widehat{\mathcal{T}})\not\subset V(\mathcal{T})$.
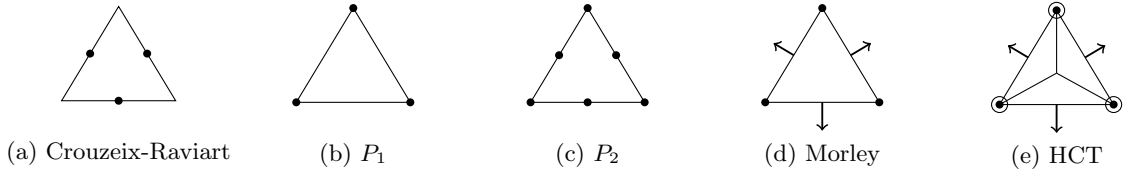
\begin{figure}
\begin{center}
	\begin{minipage}{0.19\textwidth}
		\begin{center}
			\begin{tikzpicture}[x=7.5mm,y=12.5mm]
			\phantom{\draw[->,line width=0.7pt] (1,0)--(1,-0.3);}
  					\draw (0,0)--(1,1)--(2,0)--cycle;
  					\foreach \a/\b in {1/0,0.5/0.5,1.5/0.5}
   					 { \fill (\a,\b) circle (1.5pt); }
			\end{tikzpicture}\\ \footnotesize{{(a) Crouzeix-Raviart}}
		\end{center}
	\end{minipage}
	\begin{minipage}{0.19\textwidth}
		\begin{center}
			\begin{tikzpicture}[x=7.5mm,y=12.5mm]
			\phantom{\draw[->,line width=0.7pt] (1,0)--(1,-0.3);}
 				 \draw (0,0)--(1,1)--(2,0)--cycle;
 				 \foreach \a/\b in {0/0,1/1,2/0}
    			{ \fill (\a,\b) circle (1.5pt); }
			\end{tikzpicture}\\\footnotesize{(b) $P_1$}
		\end{center}
	\end{minipage}
	\begin{minipage}{0.19\textwidth}
		\begin{center}
			\begin{tikzpicture}[x=7.5mm,y=12.5mm]
			\phantom{\draw[->,line width=0.7pt] (1,0)--(1,-0.3);}
 				 \draw (0,0)--(1,1)--(2,0)--cycle;
  				\foreach \a/\b in {1/0,0.5/0.5,1.5/0.5,0/0,1/1,2/0}
   				 { \fill (\a,\b) circle (1.5pt); }
   				 
			\end{tikzpicture}\\\footnotesize{(c) $P_2$}
		\end{center}
	\end{minipage}
	\begin{minipage}{0.19\textwidth}
		\begin{center}
			\begin{tikzpicture}[x=7.5mm,y=12.5mm]
  				\draw (0,0)--(1,1)--(2,0)--cycle;
  				\foreach \a/\b in {0/0,1/1,2/0}
  				  { \fill (\a,\b) circle (1.5pt); }
  				\draw[->, line width=0.7pt] (1,0)--(1,-0.3);
  				\draw[->,line width=0.7pt, rotate around={90:(0.5,0.5)}] (0.5,0.5)--(0.712,0.712);
 				 \draw[->,line width=0.7pt, rotate around={90:(1.5,0.5)}] (1.5,0.5)--(1.712,0.288);
			\end{tikzpicture}\\ \footnotesize{(d) Morley}
		\end{center}
	\end{minipage}
	\begin{minipage}{0.19\textwidth}
		\begin{center}
			\begin{tikzpicture}[x=7.5mm,y=12.5mm]
  					\draw (0,0)--(1,1)--(2,0)--cycle;
  					\foreach \a/\b in {0/0,1/1,2/0}
    				{ \fill (\a,\b) circle (1.5pt); 
   					 \draw (\a,\b) circle (3pt);
   					 \draw (\a,\b)--(1,1/3);}
				  \draw[->,line width=0.7pt] (1,0)--(1,-0.3);
 				 \draw[->,line width=0.7pt, rotate around={90:(0.5,0.5)}] (0.5,0.5)--(0.712,0.712);
 				 \draw[->,line width=0.7pt, rotate around={90:(1.5,0.5)}] (1.5,0.5)--(1.712,0.288);
			\end{tikzpicture}\\ \footnotesize{(e) HCT}
		\end{center}
	\end{minipage}
\end{center}
	\caption{ Mnemonic diagrams of the finite elements in $2$D.}\label{fig:FEs}
\end{figure} 

\subsection{Methodology}
The authors see three different arguments \ref{circum:(i)}--\ref{circum:(iii)} to circumvent the non-nestedness of the nonconforming schemes in the literature, 
\begin{enumerate}[label=(\roman*)]
\item \label{circum:(i)} appropriate mesh-refining, 
\item \label{circum:(ii)} discrete Helmholtz decomposition, 
\item \label{circum:(iii)} conforming companions.
\end{enumerate}
For Crouzeix-Raviart finite elements see Theorem 2.1 in \cite{Rab10} for \ref{circum:(ii)}.
The restriction to simply-connected domains and dimension $n=2$ from \ref{circum:(ii)} is circumvented in \cite{CGS13} for Crouzeix-Raviart using intermediate triangulations \ref{circum:(i)} and an associated discrete quasi-interpolation. 
For the Morley finite element analysis  see Lemma 5.5 in \cite{HSX12} for \ref{circum:(i)} and Theorem 4.1 in \cite{CGH14} for \ref{circum:(ii)}. 
This paper presents \ref{circum:(iii)} and its application for more general and refined results to prove discrete reliability. 
\SPnew{
This general domain independent principle shall serve as a guideline for the many non-conforming methods in the rich literature. Often a discrete Helmholtz decomposition is not available, however the construction of a conforming companion although allows to compute guaranteed upper error bounds. Therefore, it seems intuitive to use this operator for the proof of discrete reliability as outlined in this paper.   
}
\subsection{\SPnew{Model Problems}}
For better intuition the reader may have the following model problems in mind.  
Given a polyhedral Lipschitz domain $\Omega\subset \mathbb{R}^n$ and a right-hand side $f\in L^2(\Omega)$, for a second-order problem consider the Poisson Model Problem, find $u\in H^1(\Omega)$ with 
\begin{align*}
	\Delta u =f\text{ in }\Omega \quad\text{ and }\quad u=0\text{ along }\partial\Omega,  
\end{align*} 
where the weak formulation seeks $u\in H^1_0(\Omega)$ such that 
\begin{align*}
	\int_\Omega \nabla u \cdot \nabla v\,\textup{d}x =\int_\Omega fv\,\textup{d}x\quad
					\text{ for all }v\in H^1_0(\Omega). 
\end{align*} 
The discrete version of this energy scalar product reads 
\begin{align}
	a_h(u_h,v_h):= \int_\Omega \nabla_{\textup{NC}} u_h \cdot \nabla_{\textup{NC}}v_h\,\textup{d}x\quad\text{ for all } u_h,v_h \in H^1(\Omega)+V(\mathcal{T})+V(\widehat{\mathcal{T}}), \label{eq:SKP_PMP}
\end{align}
where a possible choice for the non-conforming finite element space $V(\mathcal{T})$ is the Crouzeix-Raviart space $CR^1_0(\mathcal{T})$.
A simple fourth-order elliptic problem is the biharmonic equation, which seeks $u\in H^2(\Omega)$ with  
\begin{align*}
	\Delta^2 u =f\text{ in }\Omega \quad\text{ and }\quad u=\frac{\partial u}{\partial\nu}=0\text{ along }\partial\Omega. 
\end{align*} 
The corresponding weak formulation seeks $u\in H^2_0(\Omega)$ such that 
\begin{align*}
	\int_\Omega D^2 u : D^2 v\,\textup{d}x =\int_\Omega fv\,\textup{d}x\quad
					\text{ for all }v\in H^2_0(\Omega). 
\end{align*} 
The discrete version of the associated energy scalar product reads 
\begin{align}
	a_h(u_h,v_h):= \int_\Omega D^2_{\textup{NC}} u_h : D^2_{\textup{NC}}v_h\,\textup{d}x\quad\text{ for all } u_h,v_h \in H^2(\Omega)+V(\mathcal{T})+V(\widehat{\mathcal{T}})\label{eq:SKP_Biha}
\end{align}
and a possible choice for the non-conforming finite element space $V(\mathcal{T})$ in the biharmonic setting is the Morley finite element space $M(\mathcal{T})$.
In both cases the discrete problem seeks $u_h \in V(\mathcal{T})$ such that 
\begin{align*}
	a_h(u_h,v_h)=\int_\Omega fv_h\,\textup{d}x\quad\text{ for all }v_h\in V(\mathcal{T}). 
\end{align*} 
\subsection{Results}
Given a regular triangulation $\mathcal{T}$ and its admissible refinement $\widehat{\mathcal{T}}$ with the finite element spaces $V(\mathcal{T})$ (resp. $V(\widehat{\mathcal{T}})$) and 
the discrete solutions $u_h$ (resp. $ {\widehat{u}}_h$), the abstract section shows the discrete reliability  
\begin{align*}
	\Vert {\widehat{u}}_h- u_h \Vert_h^2 \le \Lambda_{drel}^2 \sum_{T\in \mathcal{R}} \eta^2 (T). \label{eq:dRel}\tag{dRel}
\end{align*}
Here and throughout this paper, $\eta(T)$ is an error estimator contribution, the discrete norm $\Vert\bullet \Vert_h $ is induced by a scalar product $a_h$ on $V(\mathcal{T})+V (\widehat{\mathcal{T}})$, and  $\mathcal{R}:=\{K\in\mathcal{T}:\,\exists\, T\in \mathcal{T}\setminus\widehat{\mathcal{T}} \text{ with } \textup{dist}(K,T)=0\}$ is the set $\mathcal{T}\setminus\widehat{\mathcal{T}}$ of coarse but not fine simplices  plus one layer of coarse simplices around. 
The point is that the universal constant $ \Lambda_{drel}\lesssim 1$ solely depends on the shape-regularity of the triangulation  $ \mathcal{T}$, but neither on levels nor on mesh-sizes. 
Four abstract conditions \eqref{eq:C1_h}--\eqref{eq:C4_uhstar} in \cref{sec:C1toC4} below imply the existence of an approximation $\widehat{u}_h^*\in V(\widehat{\mathcal{T}})$ such that 
\begin{align}
\frac{2}{1+\sqrt{2}}\Vert \widehat{u}_h- u_h \Vert_h	
		\le  \Vert\widehat{u}_h^* -u_h\Vert_h+ \Lambda_1\Vert h_\mathcal{T}^m f  \Vert_{L^2(  \mathcal{T}\setminus\widehat{\mathcal{T}}) }.\label{eq:error<uhstarError}
\end{align}
The additional conditions \eqref{eq:C5_stab}--\eqref{eq:C7_ker} in \cref{sec:C5toC7} below result in
\begin{align}
	\Vert {\widehat{u}_h}^*-u_h\Vert_h^2\le \Lambda_2^2\sum_{T\in\mathcal{R}} h_T\sum_{F\in\mathcal{F}(T)} \Vert[D^m u_h]_F\times\nu_F\Vert^2_{L^2(F)}. \label{eq:uhstar-uhEstimate}
\end{align}
Throughout this paper, the piecewise constant function $h_{\mathcal{T}}|_T=h_T=\textup{diam}(T)$ is the diameter of the simplex $T\in\mathcal{T}$; $\mathcal{F}(T)$ is the set of sides  (edges for $n=2$ or faces for $n=3$) of $T$ with the tangential jumps $[v]_F\times \nu_F$  along sides $F$, and $\Vert\bullet\Vert_{L^2(\mathcal{T}\setminus\widehat{\mathcal{T}})}^2:=\sum_{T\in \mathcal{T}\setminus\widehat{\mathcal{T}}}\Vert\bullet\Vert_{L^2(T)}^2$ is the sum of the $L^2$-norms on the coarse but not fine simplices. \cref{sec:notation}  summarises the necessary notation. 
The combination of \eqref{eq:error<uhstarError}--\eqref{eq:uhstar-uhEstimate} proves \eqref{eq:dRel} with the estimator 
\begin{align}
\eta^2(T):=h_T^{2m}\Vert f\Vert^2_{L^2(T)}+h_T\sum_{F\in\mathcal{F}(T)} \Vert[D^m u_h]_F\times\nu_F\Vert^2_{L^2(F)}\label{eq:def_etaT}
\end{align}
for any simplex $T\in\mathcal{T}$ and $\Lambda_{drel}\le (1+2^{-1/2})\max\{\Lambda_1,\Lambda_2\}$. 
The second task of this paper is to sharpen this result; a modification of the companion operator behind $\widehat{u}_h^*$ allows the proof of \eqref{eq:uhstar-uhEstimate}  and thereby \eqref{eq:dRel} with  $\mathcal{T}\setminus\widehat{\mathcal{T}}$ replacing $\mathcal{R}$.
\subsection{Outline}
The remaining parts of this paper are organized as follows. \cref{sec:notation} simply recalls  the standard notation and characterizes a finite patch configuration condition for the admissible triangulations \ref{triang:A2} guaranteed for adaptive mesh refining by newest-vertex bisection.
The purpose of  \cref{sec:discussion} is an overview over the residual-based error analysis written in an abstract format to be accessible for non-experts and to describe  the state of the art and the design of the conforming companion in a language with minimal technicalities. The presented abstract conditions \eqref{eq:C1_h}--\eqref{eq:C7_ker} imply \eqref{eq:error<uhstarError}--\eqref{eq:uhstar-uhEstimate} and so \eqref{eq:dRel}.
Section \ref{sec:CrouzeixRaviart} (resp. \cref{sec:Morley}) on applications starts with the definition of the Crouzeix-Raviart (resp. Morley) finite element scheme and gives the proof of  \eqref{eq:C1_h}--\eqref{eq:C7_ker} to answer the question: How do we prove the discrete reliability for nonconforming finite element schemes? 
In \cref{sec:CrouzeixRaviart},  $u_h, \, V(\mathcal{T}),\, I_h$ etc. from the general analysis are replaced by $u_{CR},\, CR^1_0(\mathcal{T}),\, I_\textup{NC}$, and in \cref{sec:Morley} by $u_{M},\, M(\mathcal{T}),\, I_M$ etc. 
 \cref{sec:RefinedAnalysis} introduces a modified companion operator for both examples and proves that indeed $\mathcal{T}\setminus\widehat{\mathcal{T}}$ replacing $\mathcal{R}$ is sufficient in \eqref{eq:uhstar-uhEstimate}.  
 The appendices highlight a few discrete inequalities with sharp explicit constants utilized throughout the paper to compute $\Lambda_{drel}$.
\medskip\\
Standard notation on Lebesgue and Sobolev spaces applies throughout this paper; $H^m(T )$ abbreviates $H^m(\textup{int}(T ))$ for a compact set $T$ with non-void interior $\textup{int}(T )$.
Furthermore, $a \lesssim b$ abbreviates $a\le Cb$ with a generic constant $C$ independent of the meshsize $h_{\mathcal{T}}$, while $a \approx b$ stands for $a\lesssim  b\lesssim   a$. 

\section{Notation }\label{sec:notation}
\textbf{Regular triangulation.}
Given a regular triangulation $\mathcal{T}$ of a bounded polyhedral Lipschitz domain $\Omega\subset\mathbb{R}^n$ into simplices in the sense of Ciarlet \cite{BS08,Bra13,BBF13}, 
let $\mathcal{F}$ (resp. $\mathcal{F}(\Omega)$ or $\mathcal{F}(\partial\Omega)$) denote the set of all (resp. interior or boundary) sides and let $\mathcal{N}$ (resp. $\mathcal{N}(\Omega)$ or $\mathcal{N}(\partial\Omega)$) denote the set of all (resp. interior or boundary) vertices in $\mathcal{T}$.
For any simplex $T\in\mathcal{T}$, the set of its vertices reads $\mathcal{N}(T)$ and the set of its sides reads $\mathcal{F}(T)$.  
The intersection $T_1\cap T_2$ of two distinct, non-disjoint simplices $T_1$ and $T_2$ in $\mathcal{T}$ is the shared sub-simplex $\textup{conv}\{\mathcal{N}(T_1)\cap\mathcal{N}(T_2)\}=\partial T_1\cap\partial T_2$ of their shared vertices. 
 
Given a side $F\in\mathcal{F}$, the side-patch $\omega_F:=\textup{int}\big(\bigcup_{T\in\mathcal{T}(F)}T\big)$ is the interior of the union $\bigcup\mathcal{T}(F)$ of the set $\mathcal{T}(F):=\{T\in\mathcal{T}:\, F\in \mathcal{F}(T)\}$ of all simplices with side $F$.
Given a vertex $z\in\mathcal{N}$, the nodal patch $\omega_z:=\textup{int}\big(\bigcup_{T\in\mathcal{T}(z)}T\big)$ is the interior of the union $\bigcup\mathcal{T}(z)$ of the set $\mathcal{T}(z):=\{T\in\mathcal{T}:\, z\in\mathcal{N}(T)\}$ of all simplices with vertex $z$. 
For any simplex $T\in\mathcal{T}$ the set $\mathcal{T}(\Omega(T)):=\{K\in\mathcal{T}:\, \textup{dist}(T,K)=0\}$ of simplices in $\mathcal{T}$ near $T$ has cardinality $|\mathcal{T}(\Omega(T))|$ and covers the closure of $\Omega(T):=\textup{int}\big(\bigcup_{T\in\mathcal{T}(\Omega(T))}T\big)=\textup{int}\big(\bigcup_{z\in\mathcal{N}(T)}\bigcup\mathcal{T}(z)\big)$.

\vspace{.2cm}\textbf{Admissible triangulation.}
Throughout this paper, $\mathcal{T}$ is computed by successive admissible mesh-refinements of a regular initial triangulation $\mathcal{T}_0$. The set  $\mathbb{T}$ of admissible triangulations of all those triangulations is always shape-regular in the following sense. 
\begin{enumerate}[wide,label=\textbf{(A\arabic*)}]
\vspace{-0.2cm}\item\label{triang:A1}
There exists $M_1<\infty$ such that any $T\in\mathcal{T}\in\mathbb{T}$ is included in a closed ball $\overline{B}(M_T,R_T)$ and includes a closed ball $\overline{B}(m_T,r_T)$ of radii $R_T$ and $r_T$, 
$\overline{B}(m_T,r_T)\subset T \subset \overline{B}(M_T,R_T)$, with $R_T\le M_1r_T$. 
This implies finite overlap of patches and their extensions in that $|\mathcal{T}(z)|\le M_2<\infty$ for any $\mathcal{T}\in\mathbb{T}$ and $z\in\mathcal{N}$ and $M_3:=\sup_{T\in\mathcal{T}\in\mathbb{T}}|\mathcal{T}(\Omega(T))|\le (n+1) M_2<\infty$. The constants $M_1,\,M_2,\,M_3$ are universal in $\mathbb{T}$. 

Adaptive mesh-refinement typically leads to triangulations with a finite number of configurations up to scaling in the following sense.
\item\label{triang:A2}
There exists a finite number of reference patches $\mathcal{C}_1,\,\dots,\,\mathcal{C}_J$ of the vertex $0$ such that for all $\mathcal{T}\in\mathbb{T}$ and any vertex $z\in\mathcal{N}$ the patch
\begin{align*}
\mathcal{T}(z)=z+h\mathcal{C}_j
\end{align*}
is equal to a scaled copy of $\mathcal{C}_j$ for some $h>0$ and some $j\in\{1,\dots,J\}$ and $h\mathcal{C}_j=\{hK:\,K\in \mathcal{C}_j\}$ with $hK=\{hx:\,x\in K\}$.  
\end{enumerate}
The most prominent mesh-refining strategy with \ref{triang:A2} is the newest vertex bisection (NVB) based on an initial triangulation $\mathcal{T}_0$
(plus some initialization of tagged simplices as in \cite{Stev08}). It is obvious that \ref{triang:A2} implies \ref{triang:A1}.

\vspace{.2cm}\textbf{Jumps.}
Given any side $F\in\mathcal{F}$, assign its unit normal $\nu_F$ with a fixed orientation, while $\nu_T$ denotes the unit outward normal along the simplex boundary  $\partial T$ of $T\in\mathcal{T}$. Suppose $\nu_F=\nu_T|_F$ on each boundary side $F\in\mathcal{F}(\partial\Omega)\cap \mathcal{F}(T)$. 
Once the orientation of the unit normal $\nu_F$ is fixed for an interior side $F=\partial T_+\cap\partial T_-\in\mathcal{F}(\Omega)$ shared by the simplices $T_+,\, T_-\in \mathcal{T}(F)$, let $T_+$ denote the adjoint simplex with $\nu_{T_+}|_F=\nu_F$  and let $T_-$ denote the simplex with $\nu_{T_-}|_F=-\nu_F$. With this sign convention, the jump $[v]_F$ of a piecewise Lipschitz continuous function $v$ across $F$ is defined by
\begin{align*}
[v]_F(x):=\begin{cases}
v|_{T_+}(x)-v|_{T_-}(x)\qquad &\text{for }x\in F=\partial T_+\cap \partial T_-\in\mathcal{F}(\Omega),\\
v(x)\qquad &\text{for }x\in F\in\mathcal{F}(\partial\Omega).
\end{cases}
\end{align*}

\vspace{.2cm}\textbf{General notation in $\mathbb{R}^{m\times k}$.}
For $a,b\in \mathbb{R}^{m\times k}$, let $a\cdot b=a^\top b\in\mathbb{R}^{k\times k}$ and $a\otimes b=ab^\top\in\mathbb{R}^{m\times m}$. 
Let $e_k\in\mathbb{R}^m$ denote the canonical $k$-unit vector for $k=1,\dots,m$ with $e_k(j)=\delta_{jk}$ for $1\le j,k\le m$ and Kronecker delta $\delta_{jk}$. 
If $K=\textup{conv}\{P_1,P_2,\dots,P_J\}\subset\mathbb{R}^m$, let $\textup{mid}(K):=J^{-1}\sum_{j=1}^{J}P_j\in\mathbb{R}^m$ denote its centroid, e.g., the midpoint of an simplex, face or edge; set $h_K:=\textup{diam}(K)$.

The notation $|\bullet|$ depends on the context and denotes the euclidean length, the cardinality of a finite set, the $n$- or $(n-1)$-dimensional Lebesgue measure of a subset of  $\mathbb{R}^n$, e.g., $|T|$ is the volume of a simplex $T\in\mathcal{T}$ and $|F|$  denotes the area of a face $F\in\mathcal{F}$ in $3$D or the length of an edge in $2$D.

\vspace{.2cm}\textbf{Piecewise polynomials.}
The vector space of piecewise polynomials of at most degree $k$ is denoted by $P_k(\mathcal{T})$,  the subset in $H^1(\Omega)$ by $S^k(\mathcal{T}):=P_k(\mathcal{T})\cap C(\bar{\Omega})\subset H^1(\Omega)$, and the subset in $H^1_0(\Omega)$ including homogeneous boundary conditions by  $S^k_0(\mathcal{T}):=S^k(\mathcal{T})\cap C_0({\Omega})\subset H^1_0(\Omega)$.
Given a function $v\in L^2(\omega)$, define the integral mean $\intmean_\omega v\,\textup{d}x:= 1/|\omega|\,\int_\omega v\,\textup{d}x$. 
The orthogonal projection $\Pi_0:L^2(\Omega)\to P_0(\mathcal{T})$ is defined for all $f\in L^2(\Omega)$ by its average $\Pi_0(f)|_T:=\intmean_T f\,\textup{d}x$ in $T\in\mathcal{T}$. 

\section{Abstract discussion of discrete reliability}\label{sec:discussion}
\subsection{Goal}
It is the scope of this section to give an abstract and easy-to-read introduction to the principles of a proof of the discrete reliability \eqref{eq:dRel} for nonconforming finite element methods. 
One key difficulty in the a~posteriori error analysis of those methods results  from the fact that even if the triangulation $\widehat{\mathcal{T}}$ is an admissible refinement of a regular triangulation $\mathcal{T}$, the related finite element spaces $V (\widehat{\mathcal{T}})$ and $V(\mathcal{T})$ are non-nested in that
$V(\mathcal{T})\not\subset V (\widehat{\mathcal{T}})$ in general.  
In comparison with nested conforming discretizations, this causes an additional a posteriori error term in \eqref{eq:error<uhstarError} involving an approximation $\widehat{u}_h^*\in V (\widehat{\mathcal{T}})$ of the discrete solution $u_h\in V(\mathcal{T})$. The abstract description in this section introduces some general properties that cover the Crouzeix-Raviart and the Morley finite element method. 
One key ingredient  in the methodology \ref{circum:(iii)} for the definition of $\widehat{u}_h^*$ is the design of a conforming companion guided by \eqref{eq:C6_wh}--\eqref{eq:C7_ker} and the consequence \eqref{eq:rho_mu}.
In the abstract setting of this section, \eqref{eq:uhstar-uhEstimate} and therefore \eqref{eq:dRel} is proven for the set $\mathcal{R}$, which contains $\mathcal{T}\setminus\widehat{\mathcal{T}}$ plus one layer of simplices. 
A novel design of the companion operator in \cref{sec:RefinedAnalysis} allows the replacement of $\mathcal{R}$ by $\mathcal{T}\setminus\widehat{\mathcal{T}}$.

\subsection{Model problem}
To illustrate the proof of the discrete reliability \eqref{eq:dRel}, suppose that $(V(\mathcal{T}),a_h)$ is a finite-dimensional Hilbert space based on a regular triangulation $\mathcal{T}$
of  $\Omega\subset\mathbb{R}^n $, where $V(\mathcal{T})\subset P_k(\mathcal{T})$ is a vector space of piecewise polynomials  of degree at most $k$
and $a_h(\bullet, \bullet):=(D^m_\textup{NC}\bullet,D^m_\textup{NC}\bullet)$ is a scalar product that involves all piecewise derivatives $D^m_{\textup{NC}}$ of order $m$.
In the case $m=1$, $D_{\textup{NC}}^1:=D_{\textup{NC}}=\nabla_{\textup{NC}}$ denotes the piecewise action of gradient $\nabla$, while $D^2_{\textup{NC}}$ stands for the piecewise action of the Hessian $D^2$ for $m=2$. 
The underlying triangulation is neither explicit in the notation
of the scalar product $a_h$ nor in its induced norm $ \Vert\bullet\Vert_h $ with $$\Vert\bullet\Vert^2_h := \sum_{T\in\mathcal{T}}  \Vert D^m_{\textup{NC}}\bullet \Vert_{L^2(T)}^2 ,$$
so both are defined for a nonconforming finite element space $V (\widehat{\mathcal{T}})$ with respect to any admissible refinement $\widehat{\mathcal{T}}\in\mathbb{T}(\mathcal{T})$ of $\mathcal{T}$. 
The conditions \eqref{eq:C1_h}--\eqref{eq:C3_vanish} below imply a partial a~posteriori error control exemplified in \cref{thm:ProofOferror<uhstarError} for a linear model problem
with $a_h$ and the right-hand side $f\in L^2(\Omega)$ with the associated functional $F(v):=\int _\Omega f v\, \textup{d}x$ for $v\in L^2(\Omega)$. Let the discrete solution  $u_h\in V(\mathcal{T})$ solve
\begin{align*}
	a_h(u_h, v_h)=F(v_h)\quad\text{for all } v_h\in V(\mathcal{T}).
\end{align*}
On the fine level, let ${\widehat{u}_h}\in V (\widehat{\mathcal{T}})$ denote the discrete solution to $a_h({\widehat{u}_h},{\widehat{v}_h})=F({\widehat{v}_h})$ for all ${\widehat{v}_h}\in V(\widehat{\mathcal{T}})$.
The local  error estimator  $\eta(T)$ from \eqref{eq:def_etaT} leads for $\mathcal{M}\subseteq\mathcal{T}$  to
\begin{align*}
	\eta(\mathcal{T},\mathcal{M}):=\sqrt{\sum_{K\in\mathcal{M}}\eta^2(K)}.
\end{align*} 
In  \eqref{eq:def_etaT}, $[D^m u_h]_F\times\nu_F$ stands for the tangential components of the jump of the derivative $D^m u_h$ in $3$D and simplifies to $[\partial u_h/\partial s]_F$ in $2$D for $m=1$.
The error estimator $\eta$ is reliable and efficient for a large class of examples \cite{CHO07b}.

\subsection{Conditions \eqref{eq:C1_h}--\eqref{eq:C4_uhstar}}\label{sec:C1toC4}
Suppose that the nonconforming finite element space 
$V(\mathcal{T})\not\subset H^m_0(\Omega)$ allows for an interpolation operator 
$I_h : H^m_0(\Omega)+V (\widehat{\mathcal{T}})  \to V(\mathcal{T})$ with an
approximation property
\begin{align}
	\Vert {\widehat{v}_h} - I_h {\widehat{v}_h} \Vert_{L^2(T)} \le \Lambda_1\, h_T^m\, \Vert D^m(  {\widehat{v}_h}- I_h {\widehat{v}_h}) \Vert_{L^2(T)}  \quad\text{for all } T\in\mathcal{T}\label{eq:C1_h}\tag{C1}
\end{align}
and an orthogonality  
\begin{align}
	a_h( w_h , {\widehat{v}_h}-   I_h {\widehat{v}_h})= 0 \quad\text{for all }w_h\in V(\mathcal{T})\text{ and all } {\widehat{v}_h}\in V (\widehat{\mathcal{T}}).\label{eq:C2_orthogonality}\tag{C2}
\end{align}
Suppose  the interpolation operator $I_h$ acts as the identity on non-refined simplices, in the sense that 
\begin{align}
(1-I_h) {\widehat{v}_h}|_T =0  \quad\text{in } T\in\mathcal{T}\cap\widehat{\mathcal{T}}\text{ for all } {\widehat{v}_h}\in V (\widehat{\mathcal{T}}).\label{eq:C3_vanish}\tag{C3}
\end{align}
The point in what follows is that the non-nestedness $V(\mathcal{T})\not\subset V(\widehat{\mathcal{T}})$ causes that $u_h\not\in V(\widehat{\mathcal{T}})$ (in general) is not an admissible test function on the finer level.
Some transfer function  ${\widehat{u}_h}^* \in V(\widehat{\mathcal{T}})$ has to approximate 
$u_h$ in the norm of $L^2(\Omega)$ as well as in the norm $\Vert\bullet\Vert_h$ and results in estimator contributions for some simplices in $\mathcal{R}\subseteq\mathcal{T}$ below. 
The main argument for the later reduction to $\mathcal{R}$ is the property ${\widehat{u}_h}^*= u_h$ in $T\in \mathcal{T}\cap\widehat{\mathcal{T}}$. The introduction quotes a few references based on 
\ref{circum:(i)} appropriate mesh-refining  and \ref{circum:(ii)} discrete Helmholtz decomposition  to achieve this.
Given $u_h$ on the coarse level, this  paper suggests \ref{circum:(iii)} the design of  $ {\widehat{u}_h}^* \in V(\widehat{\mathcal{T}})$
on the fine level with 
\begin{align}
	I_h{\widehat{u}_h}^*=u_h.\label{eq:C4_uhstar}\tag{C4}
\end{align}

\subsection{Proof of \eqref{eq:error<uhstarError}}
\begin{theorem}\label{thm:ProofOferror<uhstarError}
	The conditions \eqref{eq:C1_h}--\eqref{eq:C4_uhstar} imply \eqref{eq:error<uhstarError} from the introduction. 
\end{theorem}
\begin{proof}
The linearity of the discrete scalar product and \eqref{eq:C2_orthogonality} imply
\begin{align*}
	\Vert  {\widehat{u}_h}- u_h\Vert_h^2
	&=  a_h(  {\widehat{u}_h},  {\widehat{u}_h}- u_h)-a_h(u_h,  I_h {\widehat{u}_h}- u_h).
\end{align*}
Given any $ {\widehat{u}_h}^* \in V(\widehat{\mathcal{T}})$, the discrete equations on the coarse level
with test-function $I_h {\widehat{u}_h}- u_h\in V(\mathcal{T})$ and on the fine level with test-function ${\widehat{u}_h}- {\widehat{u}_h}^*\in V(\widehat{\mathcal{T}})$
lead to
\begin{align*}
	\Vert {\widehat{u}_h}- u_h \Vert_h^2
		=&a_h(  {\widehat{u}_h},  {\widehat{u}_h}^*- u_h) +F\big((1-I_h)({\widehat{u}_h} - {\widehat{u}_h}^*)+ u_h-I_h{\widehat{u}_h}^*\big).
\end{align*}
Since ${\widehat{u}_h}^* \in V(\widehat{\mathcal{T}})$ satisfies \eqref{eq:C4_uhstar}, \eqref{eq:C2_orthogonality} implies $a_h(u_h,{\widehat{u}_h}^*-u_h)=0$. Therefore, the Cauchy-Schwarz inequality and \eqref{eq:C1_h}--\eqref{eq:C3_vanish} result in 
\begin{align}
	\Vert  {\widehat{u}_h}- u_h \Vert_h^2
		=& a_h(  {\widehat{u}_h}- u_h,  {\widehat{u}_h}^*- u_h)+ F((1 -I_h) ({\widehat{u}_h}  - {\widehat{u}_h}^*))\notag\\
		\le& \Vert{\widehat{u}_h}- {u_h}\Vert_h  \, \Vert{\widehat{u}_h}^*- u_h\Vert_h+ \Lambda_1\,\Vert h_\mathcal{T}^m f \Vert_{L^2( \mathcal{T}\setminus\widehat{\mathcal{T}}) } 
			\Vert(1 -I_h)({\widehat{u}_h}-{\widehat{u}_h}^*)\Vert_h \label{eq:hatuh-uh}
\end{align}
with the abbreviation (for any $s\in\mathbb{R}$) 
\begin{align*}
	\Vert h_\mathcal{T}^s \bullet \Vert_{L^2(  \mathcal{T}\setminus\widehat{\mathcal{T}}) } 
		:=\bigg(\sum_{T\in  \mathcal{T}\setminus\widehat{\mathcal{T}} }  h_T^{2s}\, \Vert\bullet\Vert_{L^2(T)}^2 \bigg)^{1/2}.
\end{align*}
The orthogonality \eqref{eq:C2_orthogonality} shows that 
\begin{align*}
	\Vert(1 -I_h)({\widehat{u}_h}-{\widehat{u}_h}^*)\Vert_h^2
		= a_h(  (1 -I_h)({\widehat{u}_h}-{\widehat{u}_h}^*), {\widehat{u}_h} -{\widehat{u}_h}^*)
		\le \Vert  (1 -I_h)({\widehat{u}_h}-{\widehat{u}_h}^*)\Vert_h\Vert {\widehat{u}_h} -{\widehat{u}_h}^*\Vert_h. 
\end{align*}
This and the triangle inequality verify 
\begin{align}
	\Vert(1 -I_h)({\widehat{u}_h}-{\widehat{u}_h}^*)\Vert_h
		\le& \Vert{\widehat{u}_h} -u_h\Vert_h+\Vert{\widehat{u}_h}^* -u_h\Vert_h.\label{eq:1-IhtoSubstitate}
\end{align}
The combination of  \eqref{eq:hatuh-uh}--\eqref{eq:1-IhtoSubstitate} and some elementary calculations conclude the proof of \eqref{eq:error<uhstarError}. 
\end{proof}
\subsection{Conditions \eqref{eq:C5_stab}--\eqref{eq:C7_ker}}\label{sec:C5toC7}
This section discusses the term $\Vert {\widehat{u}_h}^*-u_h\Vert_h$ and introduces additional conditions \eqref{eq:C5_stab}--\eqref{eq:C7_ker} sufficient for \eqref{eq:uhstar-uhEstimate}.
The explicit design of  ${\widehat{u}_h}^*$ in this paper \ref{circum:(iii)} involves a conforming companion $J_2 u_h\in V_C(\mathcal{T})\subset H^m_0(\Omega)$ followed by nonconforming interpolation $\widehat{I}_h:  V_C(\mathcal{T})\to V (\widehat{\mathcal{T}})$, namely
\begin{align*}
	{\widehat{u}_h}^*:= \widehat{I}_hJ_2 u_h.
\end{align*}
The conforming space $V_C(\mathcal{T})$ depends on the problem at hand; it is the conforming $V_C(\mathcal{T}):=S^n_0(\mathcal{T})$ for $m=1$ and the  Hsieh-Clough-Tocher finite element $V_C(\mathcal{T}):=HCT(\mathcal{T})\subset H^2_0(\Omega)$ for $m=2=n$. \SPnew{More details for the two examples follow in \cref{sec:CrouzeixRaviart} in \eqref{def:J1}--\eqref{def:J2} and in \cref{sec:Morley} in \cref{lem:MorleyCompanion}.}
Once $J_2u_h\in H^m_0(\Omega)$ is given, the stability 
 of the nonconforming interpolation $\widehat{I_h}:H^m_0(\Omega)\to V(\widehat{\mathcal{T}})$ leads to an  universal constant 
 $\Lambda_5\lesssim 1$ such that, for all $T\in\mathcal{T}$,
\begin{align}
	\Vert D^m_\textup{NC}(\widehat{I_h}v-w_h)\Vert_{L^2(T)}\le \Lambda_5\Vert D^m(v-w_h)\Vert_{L^2(T)} \ \text{ for all }v\in H^m_0(\Omega)\text{ and }\,w_h\in V(\mathcal{T}).\tag{C5}
\label{eq:C5_stab}
\end{align}
The combination of \eqref{eq:C3_vanish}--\eqref{eq:C5_stab} with $v=J_2u_h$, ${\widehat{u}_h}^*= \widehat{I}_hJ_2 u_h$, and $\Vert\bullet\Vert_h=\Vert D^m_\textup{NC}\bullet\Vert_{L^2(\Omega)}$ proves
\begin{align}
	\Vert {\widehat{u}_h}^* -u_h\Vert_h\le \Lambda_5 \Vert D^m(J_2 u_h-u_h)\Vert_{L^2(\mathcal{T}\setminus\widehat{\mathcal{T}}) }.\label{eq:UseOfC5}
\end{align} 
The subsequent discussion concerns the local analysis of the upper bound $\Vert D^m(J_2u_h-u_h)\Vert_{L^2(T)}$ for $T\in \mathcal{T}\setminus\widehat{\mathcal{T}}$ and that means the design of $J_2$.
The abstract description of the local design of $J_2:V(\mathcal{T})\to V_C(\mathcal{T})$ in \eqref{eq:C6_wh} below assumes that  $(J_2v_h)|_T$ depends on $v_h\in V(\mathcal{T})$ restricted to a neighbourhood $\Omega(T)$ of $T\in\mathcal{T}$. 
 In a formal notation, for all 
$T\in\mathcal{T}$ with $V(\mathcal{T})|_{\Omega(T)}:=\{v_h|_{\Omega(T)}:\, v_h\in V(\mathcal{T})\}\subset P_k(\mathcal{T}(\Omega(T)))$
 and $\big(V(\mathcal{T})\cap H^m_0(\Omega)\big)|_{\Omega(T)}:=\{v|_{\Omega(T)}:\, v\in V(\mathcal{T})\cap H^m_0(\Omega)\}$, assume the existence of an operator $J_{2,T}: V(\mathcal{T})|_{\Omega(T)}\to V_C(\mathcal{T})$ with 
\begin{align*}
(J_2v_h)|_T=J_{2,T}(v_h|_{\Omega(T)})\quad \text{ for all } v_h\in V(\mathcal{T}).
\end{align*}
This local contribution $J_{2,T}$ is exact for all conforming arguments in the sense that
\begin{align}
w_h|_T=J_{2,T}(w_h)\quad \text{ for all }w_h\in \big(V(\mathcal{T})\cap H^m_0(\Omega)\big)|_{\Omega(T)}\text{ and all } T\in\mathcal{T}.\label{eq:C6_wh}\tag{C6}
\end{align}
The jump estimator contributions near some simplex $T\in\mathcal{T}$ are associated with the set $\mathcal{F}(\Omega(T))$ of sides, which is defined as the set of all 
$F=\partial K_1\cap\partial K_2$ for distinct neighbouring simplices $K_1,K_2\in\mathcal{T}(\Omega(T))$ plus all boundary sides $F\subseteq \partial\Omega$ with $F\in\mathcal{F}(K)$ for some $K\in\mathcal{T}(\Omega(T))$. (Notice that any side $F$ on the boundary $\partial(\Omega(T))$ of $\Omega(T)$ is only included if it belongs to $\partial \Omega$; if $\textup{dist}(\Omega(T),\partial\Omega)>0$ then only interior sides in $\Omega(T)$ are considered in $\mathcal{F}(\Omega(T))$.)
Define the two seminorms $\mu_T,\varrho_T:V(\mathcal{T})|_{\Omega(T)} \to [0,\infty)$ for $w_h\in V(\mathcal{T})|_{\Omega(T)}=\{v_h|_{\Omega(T)}:\,v_h\in V(\mathcal{T})\}$ by
\begin{align*}
	\mu_T(w_h)&:=\Big(\sum_{F\in\mathcal{F}(\Omega(T))}h_F\Vert[D^mw_h]_F\times \nu_F\Vert^2_{L^2(F)}\Big)^{1/2}\text{ and }\\
	\varrho_T(w_h)&:=\Vert D^m(w_h-J_{2,T}w_h)\Vert_{L^2(T)}.
\end{align*}
The condition \eqref{eq:C6_wh} implies that $\big(V(\mathcal{T})\cap H^m_0(\Omega)\big)|_{\Omega(T)}$ belongs to the null space 
\begin{align*}
	\textup{Ker}\varrho_T=\{w_h\in V(\mathcal{T})|_{\Omega(T)}:\, \varrho_T(w_h)=0\}
\end{align*}
of $\varrho_T$. The latter space is supposed to include the null space $\textup{Ker}\mu_T$ of $\mu_T$ in that 
\begin{align}
	\forall w_h\in V(\mathcal{T})|_{\Omega(T)}\quad \bigg( \mu_T(w_h)=0\quad\Rightarrow\quad w_h\in \big(V(\mathcal{T})\cap H^m_0(\Omega)\big)|_{\Omega(T)}\bigg).\label{eq:C7_ker}\tag{C7}
\end{align}
In conclusion, $\textup{Ker}\mu_T\subset\textup{Ker}\varrho_T$. 
The vector space $V(\mathcal{T})|_{\Omega(T)}$ has dimension at most $\dim P_k(T)=\binom{k+n}{n}$ times the cardinality $|\mathcal{T}(\Omega(T))|\le M_3$ of simplices near $T$. 
Hence, an inverse estimate argument similar to that in the proof of the equivalence of norms on a finite-dimensional vector space $V(\mathcal{T})|_{\Omega(T)}$ leads to 
\begin{align}
	\varrho|_T(w_h)\le C(T)\mu_T(w_h)\quad\text{ for all }w_h\in V(\mathcal{T})|_{\Omega(T)}\label{eq:rho_mu}
\end{align}
for some constant $C(T)$ that depends on the local companion operator $J_{2,T},$ the triangulation $\mathcal{T}(\Omega(T))$, the sides $ \mathcal{F}(\Omega(T))$, and the maximal polynomial degree $k$. 
Under the assumption \ref{triang:A2} on $\mathbb{T}$, the constants $C(T)$ in \eqref{eq:rho_mu} are uniformly bounded. 
\begin{lemma}\label{lem:C2_universal}
	The assumptions \textup{\ref{triang:A2}} and \eqref{eq:C6_wh}--\eqref{eq:C7_ker} imply 
	\begin{align}
		C(\mathbb{T}):=\sup_{T\in\mathcal{T}\in\mathbb{T}}C(T)<\infty. \label{eq:defC2}
	\end{align} 
\end{lemma}
\begin{proof}
	The aforementioned soft analysis arguments lead to \eqref{eq:rho_mu} with a constant $C(T)$ that depends on  the maximal polynomial degree  $k$ and on the configuration 
	$\mathcal{T}(\Omega(T))$. 
	The assumption \ref{triang:A2} states that any nodal patch $ \mathcal{T}(z)$ is equal to $ z+ h_z \mathcal{C}_{j(z)}$ for some $j(z)\in\{1,\dots, J\}$ and some $h_z>0$.  
	Since $\mathcal{T}(\Omega(T))$ is the union of the $n+1$ nodal patches  for the vertices  $z\in\mathcal{N}(T)$ of $T$, it follows 
	\begin{align}
		\mathcal{T} (\Omega( T))=\bigcup_{z\in\mathcal{N}(T)}    ( z+ h_{z} \mathcal{C}_{j(z)} ).\label{eq:mathcalT(Omega(T))}
	\end{align}
	A scaling argument of the piecewise polynomials shows that the constant $C(T)$ does not depend on a uniform scaling of all those factors $\{h_{z} : z\in\mathcal{N}(T)\}$, 
	so without loss of generality let $h_T=1$. Then the other scaling factors are determined by the shape-regularity of $\mathcal{T} (\Omega(T))$ and their overlap
	$T$; in other words, there exists only a finite number of (scaled) configurations $\mathcal{T} (\Omega(T))$ with  $h_T=1$ despite the fact that there are infinite 
	triangulations $\mathcal{T}$ in $\mathbb{T}$. 
	Each of those configurations leads to some positive constant $C(T)$ and the maximum of those finite number of values is $C(\mathbb{T})$, which is positive and
	exclusively  depends on $\mathbb{T}$ and on the maximal polynomial degree $k$. This concludes the proof.
\end{proof}
\cref{lem:C2_universal} shows that  the general assumptions \eqref{eq:C6_wh}--\eqref{eq:C7_ker} and \ref{triang:A2} are one example for sufficient conditions for \eqref{eq:rho_mu}--\eqref{eq:defC2}. For nonconforming  Crouzeix-Raviart and Morley finite element methods, the subsequent sections present some upper bounds of $C(\mathbb{T})$ for $n=2$ 
and show that $C(\mathbb{T})>0$ depends solely on the minimal angle $\omega_0$ in $\mathbb{T}$ from \ref{triang:A1}.   
\subsection{Proof of \eqref{eq:uhstar-uhEstimate}} 
\begin{theorem}\label{thm:uhstar-uhEstimate}
	The assumptions \textup{\ref{triang:A2}} 
	and \eqref{eq:C3_vanish}--\eqref{eq:C7_ker} imply \eqref{eq:uhstar-uhEstimate} with 
	$\Lambda_2=\Lambda_5C(\mathbb{T})M_3^{1/2}$.
\end{theorem}
\begin{proof}
	Recall that a combination of \eqref{eq:C3_vanish}--\eqref{eq:C4_uhstar} shows  ${\widehat{u}_h}^*|_T=I_h{\widehat{u}_h}^*|_T={u_h}|_T$ for $T\in\mathcal{T}\cap\widehat{\mathcal{T}}$
	and then \eqref{eq:C5_stab} implies \eqref{eq:UseOfC5}.
	The definitions of $\mu_T$ and $\varrho_T$ lead in \cref{lem:C2_universal} to \eqref{eq:rho_mu}--\eqref{eq:defC2}, 
	\begin{align}
		\Vert D^m(w_h-J_{2,T}w_h)\Vert_{L^2(T)}\le C(\mathbb{T})\mu_T(w_h)\qquad \text{for all } T\in\mathcal{T}\text{ and all }w_h\in V(\mathcal{T})|_{\Omega(T)}.\label{eq:CmathbbTgutaufgeschrieben}
	\end{align}
	Given any  $v_h\in V(\mathcal{T})$,  the piecewise definition of $J_2$ through the local contributions $J_{2,T}$ for $T\in\mathcal{T}$ and 
	\eqref{eq:CmathbbTgutaufgeschrieben} for  $w_h=v_h|_{\Omega(T)}$ result in 
	\begin{align*}
		\sum_{T\in\mathcal{T}\setminus\widehat{\mathcal{T}}}\Vert D^m(v_h-J_2v_h)\Vert^2_{L^2(T)}
			&\le C(\mathbb{T})^2\sum_{T\in\mathcal{T}\setminus\widehat{\mathcal{T}}}\mu_T^2(v_h|_{\Omega(T)})\\
			&\le C(\mathbb{T})^2M_3\sum_{T\in\mathcal{R}}\sum_{F\in \mathcal{F}(T)}h_F\Vert [D^mv_h]_F\times\nu_F\Vert^2_{L^2(F)}.
	\end{align*}
	Since $F\in \mathcal{F}(T)$ implies $h_F\le h_T$, this concludes the proof of \eqref{eq:uhstar-uhEstimate}.
	\end{proof}	

\section{Crouzeix-Raviart Finite Elements}\label{sec:CrouzeixRaviart}
This section establishes the conditions \eqref{eq:C1_h}--\eqref{eq:C7_ker} for Crouzeix-Raviart finite elements for $m=1$ and $n\ge2$\SPnew{, hence with a second-order problem as the Poisson Model Problem \eqref{eq:SKP_PMP} in mind}. The notation from the abstract \cref{sec:discussion} is specified for the Crouzeix-Raviart finite element method in that   $u_{CR}$ replaces $u_h$, $I_\textup{NC}$ replaces $I_h$, and $\sum_{T\in\mathcal{T}}\Vert \nabla_{\textup{NC}} \bullet\Vert^2_{L^2(T)}$ replaces $\Vert\bullet\Vert_h^2$ etc. 
\subsection{Interpolation and Conforming Companion Operator}\label{sec:CR_Interpolant}
The Crouzeix-Raviart finite element spaces (with and without boundary conditions) read
\begin{align*}
	CR^1(\mathcal{T})
		&:=\{v_{CR}\in P_1(\mathcal{T}):\ v_{CR} \text{ is continuous at }\textup{mid}(F)\text{ for all }F\in\mathcal{F}(\Omega) \},\\
	CR^1_0(\mathcal{T})
		&:=\{v_{CR}\in CR^1(\mathcal{T}):\ v_{CR}(\textup{mid}(F))=0\text{ for all }F\in\mathcal{F}(\partial\Omega) \}.	
\end{align*}
For any  admissible refinement $\widehat{\mathcal{T}}\in\mathbb{T}(\mathcal{T})$ of $\mathcal{T}\in\mathbb{T}$ and the side-oriented basis functions $\psi_{F}\in CR^1(\mathcal{T})$ with $\psi_F(\textup{mid}(E))=\delta_{EF}$ for all sides $E,F\in\mathcal{F}$, the interpolation operator $I_{\textup{NC}}:H^1_0(\Omega)+CR^1_0(\widehat{\mathcal{T}})\to CR^1_0(\mathcal{T})$ reads
\begin{align*}
I_{\textup{NC}}(f):=\sum_{F\in\mathcal{F}}\bigg(\intmean_F f\,\textup{d}s\bigg)\ \psi_F \quad\text{ for any }
f\in H^1_0(\Omega)+CR^1_0(\widehat{\mathcal{T}}).
\end{align*}
The side-oriented basis functions $\big(\widehat{\psi}_F:\,F\in\widehat{\mathcal{F}}\big)$ of $CR^1(\widehat{\mathcal{T}})$  with respect to the fine triangulation $\widehat{\mathcal{T}}$ allow the analog  definition of  the interpolation operator $\widehat{I}_{\textup{NC}}:H^1_0(\Omega)\to CR^1_0(\widehat{\mathcal{T}})$. 

The design for $\Omega\subset\mathbb{R}^2$ of the conforming companion operator $J_2: CR^1_0(\mathcal{T})\to S^2_0(\mathcal{T})\subset H^1_0(\Omega)$ from \cite[Proof of Prop.2.3]{CGS15} generalizes to any space dimension $n\ge 2$.
Let $v_{CR}|_T(z)$ denote the value of $v_{CR}\in CR^1_0(\mathcal{T})$ on $T\in\mathcal{T}$ at the vertex $z\in\mathcal{N}(T)$ and let $|\mathcal{T}(z)|\ge 1$ be the number of simplices in the nodal patch. Nodal averaging defines $J_1:CR^1_0(\mathcal{T})\to S^1_0(\mathcal{T})$, where 
\begin{align}
\big(J_1v_{CR}\big)(z)=|\mathcal{T}(z)|^{-1}\sum_{T\in \mathcal{T}(z)}v_{CR}|_{T}(z)\quad\text{ for all }z\in\mathcal{N}(\Omega)\label{def:J1}
\end{align}
is followed by linear interpolation (plus homogeneous boundary conditions). This is called an enrichment operator in \cite{BS08} and also considered in the medius analysis in \cite{Gudi10,CPS12}. 
Let $\varphi_z\in S^1(\mathcal{T})$ with $\varphi_z(a)=\delta_{az}$ for all vertices $a,z\in \mathcal{N}$ denote the $P_1$-conforming basis functions  and let $b_F:=\big(\prod_{z\in\mathcal{N}(F)}\varphi_z\big)/\int_F\big(\prod_{z\in\mathcal{N}(F)}\varphi_z\big)\,\textup{d}s\in P_n(\mathcal{T}(F))$ for any side  $F\in\mathcal{F}$ be a normalized side-bubble function.  
Then $J_2:  CR^1_0(\mathcal{T})\to S^n_0(\mathcal{T}) $  reads
\begin{align}
	J_2v_{CR}:=J_1v_{CR}+\sum_{F\in\mathcal{F}(\Omega)}\bigg(\intmean_F(v_{CR}-J_1v_{CR})\,\textup{d}s\bigg)b_F.\label{def:J2}
\end{align}

\subsection{Proof of \eqref{eq:C1_h}}\label{sec:CR_C1}
This is Theorem 3.5 in \cite{CH17} with 
$\Lambda_1=\sqrt{19/48}\le 0.629153$ for $n=2$ or $\Lambda_1=\sqrt{101/180}\le 0.749074$ for $n=3$. $\hfill{\Box}$

\subsection{Proof of \eqref{eq:C2_orthogonality}, \eqref{eq:C5_stab} }
Lemma 13  in \cite{CHel17} recalls $\Pi_0\nabla _\textup{NC}{\widehat{v}_{CR}}=\nabla _\textup{NC}I_\textup{NC}{\widehat{v}}_{CR}\in P_0(\mathcal{T};\mathbb{R}^n)$ for all $\widehat{v}_{CR}\in CR^1_0(\widehat{\mathcal{T}})$.
Since $\nabla _\textup{NC} w_{CR}\in P_0(\mathcal{T};\mathbb{R}^n)$ for all $w_{CR}\in CR^1_0(\mathcal{T})$, \eqref{eq:C2_orthogonality} follows from 
\begin{align*}
	a_h( w_{CR} , {\widehat{v}}_{CR}-   I_\textup{NC} {\widehat{v}}_{CR})=(\nabla _\textup{NC} w_{CR}, (1-\Pi_0)\nabla _\textup{NC}\widehat{v}_{CR} )_{L^2(\Omega)}=0.
\end{align*} 
The analog identity on the refined triangulation $\widehat{\mathcal{T}}$ reads $\widehat{\Pi}_0\nabla v=\nabla _\textup{NC}\widehat{I}_\textup{NC}v$  for all $v\in H^1_0(\Omega)$. This and $\nabla _\textup{NC} w_{CR}\in P_0(\mathcal{T};\mathbb{R}^n)\subset P_0(\widehat{\mathcal{T}};\mathbb{R}^n)$ imply \eqref{eq:C5_stab} for any $T\in\mathcal{T}$ with $\Lambda_5=1$.  $\hfill{\Box}$	

\subsection{Proof of \eqref{eq:C3_vanish}}
The restriction of any ${\widehat{v}_{CR}}\in CR^1_0(\widehat{\mathcal{T}})$ to some $T\in\mathcal{T}\cap\widehat{\mathcal{T}}$  satisfies 
${\widehat{v}_{CR}}|_T\in\textup{span}\{\psi_F|_T:\,F\in\mathcal{F}(T)\}$ with the side-oriented shape functions $\psi_F\in CR^1_0(\mathcal{T})$. 
The duality property $\intmean_F\psi_E\,\textup{d}s=\delta_{EF}$ for all sides $E,F\in\mathcal{F}$ implies  
$I_\textup{NC}{\widehat{v}_{CR}}|_T={\widehat{v}_{CR}}|_T$.$\hfill{\Box}$

\subsection{Proof of \eqref{eq:C4_uhstar}}
Given any $u_{CR}\in CR^1_0(\mathcal{T})$, set ${\widehat{u}_{CR}}^*:=\widehat{I}_{\textup{NC}}J_2(u_{CR})$.
The correction with normalized side-bubble functions in \eqref{def:J2} guarantees 
$\intmean_F J_2u_{CR}\,\textup{d}s=\intmean_F u_{CR}\,\textup{d}s$  for all sides $F\in\mathcal{F}.$
Hence, the definition of $I_{\textup{NC}}$  implies
$ I_\textup{NC}{\widehat{u}_{CR}}^*  =I_\textup{NC}J_2(u_{CR})=I_\textup{NC}u_{CR}=u_{CR}$. This proves \eqref{eq:C4_uhstar}. $\hfill{\Box}$

\subsection{Proof of \eqref{eq:C6_wh}}
Given any $v_{CR}\in CR^1_0(\mathcal{T})$ and $K\in\mathcal{T}$, the restriction $J_1v_{CR}|_K$ of the conforming companion $J_1v_{CR}$ is the linear interpolation of the nodal values $(J_1v_{CR})(z)$ at $z\in\mathcal{N}(K)\cap \mathcal{N}(\Omega)$ computed from the nodal values of $v_{CR}|_T(z)$ restricted to the simplex $T\in\mathcal{T}(z)\subset \mathcal{T}(\Omega(K))$.  
Therefore 	$J_2$ is associated with local contributions $J_{2,K}$ for any $K\in\mathcal{T}$ in the sense that 
	\begin{align*}
			(J_2v_{CR})|_K=J_{2,K}(v_{CR}|_{\Omega(K)})\quad \text{ for all } v_{CR}\in CR^1_0(\mathcal{T}).
	\end{align*}
Any $w_{CR}\in \big(CR^1_0(\mathcal{T})\cap H^1_0(\Omega)\big)|_{\Omega(K)}= S^1_0(\mathcal{T})|_{\Omega(K)}$ is continuous in $\Omega(K)$ and vanishes along $\partial\Omega\cap\partial(\Omega(K))$ so the values $J_1(w_{CR})(z)=w_{CR}(z)$ coincide at all vertices $z\in\mathcal{N}(\Omega(K)):=\{z\in\mathcal{N}(T):\, T\in\mathcal{T}(\Omega(K))\}$ 
and the integral means  $\intmean_F(w_{CR}-J_1w_{CR})\,\textup{d}s=0$ vanish along all sides $F\in\mathcal{F}(\Omega(K))$. Consequently, for all $K\in\mathcal{T}$, $J_{2,K}$  satisfies \eqref{eq:C6_wh}.$\hfill{\Box}$

\subsection{Proof of \eqref{eq:C7_ker}}\label{sec:CR_C7}
Any $w_{CR}\in CR^1_0(\mathcal{T})|_{\Omega(T)}$ is piecewise affine, continuous at the side midpoints and vanishes at midpoints of boundary sides $F\subset \partial\Omega\cap\partial(\Omega(T))$. 
Hence, the jump $[w_{CR}]_F$ across each side $F\in\mathcal{F}$ 
is of the form  $[w_{CR}]_F(x)=a\cdot(x-\textup{mid}(F))$ for some $a\in\mathbb{R}^n$ and any $x\in F$. 
Since $a=[\nabla _{\textup{NC}} w_{CR}]_F$ and the normal $\nu_F$ is perpendicular to $(x-\textup{mid}(F))\perp \nu_F$ at any $x\in F$,
  the jumps $[w_{CR}]_F\equiv 0$ vanish if and only if the tangential jumps of the gradients
$\Vert[\nabla  w_{CR}]_F\times \nu_F\Vert^2_{L^2(F)}=0$ vanish. 
Therefore, $\sum_{F\in\mathcal{F}(\Omega(T))}h_F\Vert[\nabla  w_{CR}]_F\times \nu_F\Vert^2_{L^2(F)}=0$ implies that $w_{CR}$ is 
continuous in $\Omega(T)$  and vanishes along each boundary side $F\in\mathcal{F}(\partial\Omega)\cap\mathcal{F}(\partial\Omega(T))$.  
This proves \eqref{eq:C7_ker}.
$\hfill{\Box}$

\subsection{Constants in $2$D}\label{sec:CR2DCase}
	In the case $\Omega\subset\mathbb{R}^2$, \cref{sec:CR_C1} shows $\Lambda_1=\sqrt{19/48}\le 0.629153$ and this section bounds the constant $\Lambda_2$ 
	in terms of the smallest angle $\omega_0$ in the set of admissible triangulations $\mathbb{T}$ 
	and $M_2 =\max_{\mathcal{T}\in\mathbb{T}}\{|\mathcal{T}(z)|: z\in\mathcal{N}\}\le 2\pi/\omega_0 $.
	The combination of \eqref{eq:UseOfC5} and the inverse estimate \cite[Lemma 4.5.3]{BS08} with constant $c_{\textup{inv},2}$ for piecewise polynomials of degree at most $2$ implies
		\begin{align}
		\Vert {\widehat{u}_{CR}}^* -u_{CR}\Vert_h
			\le c_{\textup{inv},2}  \Vert h_\mathcal{T}^{-1} ({u}_{CR}-J_2{u}_{CR} )\Vert_{L^2(\mathcal{T}\setminus\widehat{\mathcal{T}})}.\label{eq:uCR*-uCR_2D}
		\end{align}
		For each $T\in\mathcal{T}$, the definition of $J_2$ in \eqref{def:J2} and the triangle inequality lead to
		\begin{align*}
			 \Vert {u}_{CR}-J_2{u}_{CR}\Vert_{L^2(T)}
			 		\le 	\Vert {u}_{CR}-J_1{u}_{CR}\Vert_{L^2(T)}+\bigg\Vert \sum_{F\in\mathcal{F}(T)\cap\mathcal{F}(\Omega)}\bigg|\intmean_F({u}_{CR}-J_1{u}_{CR})\,\textup{d}s\bigg| b_F\bigg\Vert_{L^2(T)}.
		\end{align*}			
		Moreover, the local mass matrix for normalized bubble functions reads (with unit matrix $1_{3\times 3}\in\mathbb{R}^{ 3 \times 3}$)
		\begin{align*}
			B(T):=\Big(\int_T b_{E}b_{F}\,\textup{d}x\Big)_{E,F\in\mathcal{F}(T)}
					=\frac{|T|}{5} \big(1_{3\times 3}+(1,1,1)\otimes(1,1,1)\big)\in\mathbb{R}^{ 3 \times 3}
		\end{align*}
		and has the double eigenvalue $\lambda_{\min}=|T|/5$ and the simple eigenvalue $\lambda_{\max}=4|T|/5$.		
		The  discrete trace identity
		$\sum_{F\in\mathcal{F}(T)}\big\vert\intmean_{F} u_{CR}-J_1{u}_{CR}\,\textup{d}s\big\vert^2= {3}{|T|^{-1}}\Vert  u_{CR}-J_1{u}_{CR} \Vert^2_{L^2(T)}$ 
		holds in $2$D.			
		Consequently, 	
					\begin{align}
					  \Vert u_{CR}-J_2{u}_{CR}\Vert_{L^2(T)}\le C_{J} \Vert u_{CR}-J_1{u}_{CR}\Vert_{L^2(T)}\text{ with }C_J= 1+2\sqrt{3/5}\le 2.5492.\label{eq:J2_CR_2D} 
					\end{align}
		Theorem 4.5 and Remark 4.7 in \cite{CH17} prove that $C_{\textup{loc}}:=(16\sqrt{3}(1-\cos(\pi/M_2)))^{-1}$ and 
		$\eta_F^2:=h_F\Vert [\nabla {u}_{CR}]_F\times\nu_F\Vert_{L^2(F)}^2$ satisfy
		\begin{align}
			h_T^{-2}\Vert u_{CR}-J_1{u}_{CR}\Vert_{L^2(T)}^2
				&\le C_{\textup{loc}}\sum_{z\in \mathcal{N}(T)}\sum_{\substack{F\in\mathcal{F}\\z\in\mathcal{N}(F)}}\eta_F^2
	    		\le 2C_{\textup{loc}} \sum_{F\in\mathcal{F}(\Omega(T))}\eta_F^2.\label{eq:J1_CR_2D}
		\end{align}
        The combination of \eqref{eq:uCR*-uCR_2D}--\eqref{eq:J1_CR_2D} proves $ C(\mathbb{T})\le c_{\textup{inv},2}C_J\sqrt{2C_{\textup{loc}}}$ in \eqref{eq:defC2}. 
        Furthermore, $M_3  \le 3M_2\le 6\pi/\omega_0$, and $\Lambda_2=C(\mathbb{T})M_3 ^{1/2}$ hold in \eqref{eq:uhstar-uhEstimate}. 

\section{Morley Finite Elements} \label{sec:Morley}
This section verifies the conditions \eqref{eq:C1_h}--\eqref{eq:C7_ker} for Morley finite elements with $m=2=n$\SPnew{, hence with a fourth-order problem as the Biharmonic problem \eqref{eq:SKP_Biha} in mind}. The notation from the abstract \cref{sec:discussion} is adapted to the Morley finite element space in that $u_{M}$ replaces $u_h$, $I_{M}$ replaces $I_h$, and  $\sum_{T\in\mathcal{T}}\Vert D^2_{\textup{NC}} \bullet\Vert^2_{L^2(T)}$ replaces $\Vert\bullet\Vert_h^2$ etc.

\subsection{Interpolation and Conforming Companion Operator}\label{sec:Morley_HCT}
Given a regular triangulation $\mathcal{T}$ of $\Omega\subset\mathbb{R}^2$ with the set $\mathcal{E}$ of edges, the triangular equilibrium \cite{Mor68} also known as Morley finite element spaces (with and without boundary condition)  is 
\begin{align*}
	{M}^\prime(\mathcal{T}):=&\{v\in P_2(\mathcal{T}):\ v\text{ is continuous at }\mathcal{N}\text{ and } 
							\\ &\qquad\qquad\qquad\nabla _{\textup{NC}}v\text{ is continuous at }\textup{mid}(E)\text{ for all }E\in\mathcal{E}\},\\
	{M}(\mathcal{T}):=&\{v\in {{M}}^\prime(\mathcal{T}):\ v\text{ vanisches at }\mathcal{N}(\partial\Omega)\text{ and }
							\\ &\qquad\qquad\qquad \nabla _{\textup{NC}}v\text{ vanishes at }\textup{mid}(E)\text{ for all }E\in\mathcal{E}(\partial\Omega)\}.		
\end{align*}
The shape functions for this finite element are displayed in \cite[(6.1)]{CGH14}, the local degrees of freedom for $\phi_M\in{M}(\mathcal{T})$ on $T\in\mathcal{T}$  are the nodal values $\phi_M(z)$ for $z\in\mathcal{N}(T)$ and the normal derivatives $\partial\phi_M/\partial\nu_{E}(\textup{mid}(E))$ in the midpoints of the edges  $E\in\mathcal{E}(T)$.  
For any admissible refinement $\widehat{\mathcal{T}}\in\mathbb{T}(\mathcal{T})$ of $\mathcal{T}\in\mathbb{T}$ and the normal derivative ${\partial v}/{\partial \nu_E}:=\nabla v\cdot\nu_E$ along the edges $E\in\mathcal{E}$, the interpolation operator $I_{M}:\, H^2_0(\Omega)+{M}(\widehat{\mathcal{T}})\to{M}(\mathcal{T})$ \cite{CGal14, Gal15} for any $v\in H^2_0(\Omega)+{M}(\widehat{\mathcal{T}})$ is characterized by 
\begin{align*}
	(I_{M}v)(z)=v(z) \quad\text{for any }z\in\mathcal{N}\quad\text{ and }\quad
	\frac{\partial I_{M}v}{\partial\nu_E}(\textup{mid}(E))=\intmean_E \frac{\partial v}{\partial \nu_E}\,\textup{d}s\quad\text{for any }E\in\mathcal{E}.
\end{align*}
The analog characterization  with respect to the fine triangulation $\widehat{\mathcal{T}}$ defines  the interpolation operator $\widehat{I}_{{M}}: H^2_0(\Omega)\to {M}(\widehat{\mathcal{T}})$ to the Morley finite element space ${M}(\widehat{\mathcal{T}})$. 

A conforming finite-dimensional subspace of $H^2_0(\Omega)$ is the Hsieh-Clough-Tocher (${HCT}$) finite element  \cite[Chap. 6]{Ciarlet78}. 
For any $T\in\mathcal{T}$ let $\mathcal{K}(T):=\{T_E:\ E\in\mathcal{E}(T)\}$ denote the triangulation of $T$ into three sub-triangles $T_E:=\textup{conv}\{E,\textup{mid}(T)\}$ with edges $E\in\mathcal{E}(T)$ and common vertex $\textup{mid}(T)$. Then, 
\begin{align}
{HCT}(\mathcal{T})&:=\{v\in H^2_0(\Omega):\ v|_T\in P_3(\mathcal{K}(T))\text{ for all }T\in\mathcal{T}\}.\label{eq:HCT}
\end{align}
The local degrees of freedom for $\psi\in {HCT}(\mathcal{T})$ on $T\in\mathcal{T}$  are the nodal values of the function $\psi(z)$,  of the derivative $\nabla \psi(z)$ for $z\in\mathcal{N}(T)$ and the values of the normal derivatives $\partial\psi/\partial\nu_{E}(\textup{mid}(E))$ at the midpoints of the edges $E\in\mathcal{E}(T)$.  
\begin{lemma}
There exists a conforming companion operator $J_G: M(\mathcal{T})\to HCT(\mathcal{T})+\big(P_5(\mathcal{T})\cap H^2_0(\Omega)\big)$ such that  $J_G v_M\in HCT(\mathcal{T})+\big(P_5(\mathcal{T})\cap H^2_0(\Omega)\big)$ satisfies \ref{MorleyCompanion:Nodal}--\ref{MorleyCompanion:NormEstimate} for any $v_M\in\mathcal{M}(\mathcal{T})$. 
\begin{enumerate}[label=(\roman*), wide]\label{lem:MorleyCompanion}
\item\label{MorleyCompanion:Nodal} $J_Gv_M(z)=v_M(z)$ for any  $z\in\mathcal{N}$;
\item\label{MorleyCompanion:Derivative}  $\nabla ({J_G}v_M)(z)=\begin{cases}
					|\mathcal{T}(z)|^{-1}\sum_{T\in\mathcal{T}(z)}(\nabla v_M|_T)(z)\quad &\text{ for }z\in\mathcal{N}(\Omega),\\
					0\quad&\text{ for }z\in\mathcal{N}(\partial\Omega);
		\end{cases}$
\item \label{MorleyCompanion:Edges}$\intmean_E \partial J_G v_M/\partial\nu_E \,\textup{d}s=\intmean_E \partial v_M/\partial\nu_E \,\textup{d}s$ for any $E\in\mathcal{E}$;
\item \label{MorleyCompanion:Identity}$I_MJ_G v_M=v_M$;
\item \label{MorleyCompanion:NormEstimate}
$
 h^{-4}_{\mathcal{T}}\Vert 
v_M-{J_G}v_M\Vert_{L^2(T)}^2\lesssim \underset{{E\in\mathcal{E}(\Omega(T))}}{\sum}h_E\Vert[D^2v_M]_E\times\nu_E \Vert_{L^2(E)}^2 \lesssim \underset{{v\in H^2_0(\Omega)}}{\min} \Vert D^2_{\textup{NC}}(v_M-v)\Vert_{L^2(\Omega(T))}^2.
$
\end{enumerate}
\end{lemma}
\begin{proof}
	Proposition 2.5 of \cite{Gal15} defines a companion operator with \ref{MorleyCompanion:Nodal}--\ref{MorleyCompanion:Derivative}.  
	\SPnew{In that paper given $v_M\in M(\mathcal{T})$, the first step is the definition of some $J_1v_M\in HCT(\mathcal{T})$ 
	by averaging all the degrees of freedom. That means for each interior node $z\in\mathcal{N}(\Omega)$, the derivatives 
	$\partial^\alpha(J_1v_M)(z)$ is the average of all $\partial^\alpha v_M|_T(z)$ for $T\in\mathcal{T}(z)$ and all orders $|\alpha|\le 1$ and
	 $\partial (J_1v_M)/\partial \nu_E=\partial v_M/\partial \nu_E$ at the midpoint $\textup{mid}(E)$ for each interior edge $E\in\mathcal{E}(\Omega)$; while the degrees of 
	 freedom on the boundary $\partial\Omega$ are set to zero for $J_1v_M\in H^2_0(\Omega)$.}
	The edge-bubbles $b_{E,T}:=30(\nu_T\cdot\nu_E)\textup{dist}(z_3,E)\varphi_1^2\varphi_2^2\varphi_3\in P_5(T)$  (for 
	$T=\textup{conv}\{z_1,z_2,z_3\}=\textup{conv}\{E,z_3\}\in\mathcal{T}$ and the nodal basis function $\varphi_j\in S^1(\mathcal{T})$ 
	associated with $z_j$) \SPnew{continuously extended by zero for $T\not\in\mathcal{T}(E)$ to $b_E$ }correct the integral mean of the normal 
	derivatives along the edges to guarantee \ref{MorleyCompanion:Edges} 
	\SPnew {for $J_Gv_M:=	J_1 v_M+\sum_{E\in\mathcal{E}(\Omega)}\big(\intmean_E (v_M-J_1v_M)\,\textup{d}s\big)b_E$}	
	\cite[Prop. 2.6]{Gal15}.  
	\SPnew{Since the Morley element is continuous in the nodes \ref{MorleyCompanion:Nodal} holds.}
	The characterization of the Morley interpolation operator shows that \ref{MorleyCompanion:Nodal} 
	and \ref{MorleyCompanion:Edges} imply \ref{MorleyCompanion:Identity}.
   	Proposition 2.5 of \cite{Gal15} displays a global version of the estimate \ref{MorleyCompanion:NormEstimate} (obtained by the sum over $T\in\mathcal{T}$);
	 a closer investigation of the proof reveals that the local arguments 
	 \SPnew{ for the HCT element from \cite{Ciarlet78} (which is almost affine) } verify \ref{MorleyCompanion:NormEstimate}. 
	\SPnew{Other $C^1$-conforming elements, such as for example the Argyris element, allow for a similar construction of a conforming companion by averaging. 
	The appropriate corrections then guarantee \ref{MorleyCompanion:Nodal} and \ref{MorleyCompanion:Edges} and therefore \ref{MorleyCompanion:Identity}. The choice of $HCT$ in \cite{Gal15} is a natural one because the degrees of freedom fit conveniently  to those of the Morley finite element.}    
\end{proof}

\subsection{Proof of \eqref{eq:C1_h}}	\label{sec:C1Morley}
Theorem 3 in \cite{CGal14} asserts $\Vert v_M-I_{M}v_M\Vert_{L^2(K)}\le \kappa_{M}h_K^2\Vert D^2(v_M-I_{M}v_M)\Vert_{L^2(K)}$ for all $v\in H^2(K)$ and $K\in\mathcal{T}$ with $\kappa_M=0.257457844658$. This estimate holds on any coarse and fine triangle $K\in\mathcal{T}\cap\widehat{\mathcal{T}}$. 
The arguments in \cite{CGal14,CH17} can be generalized to prove 
$\Vert \widehat{v}_M-I_{M}\widehat{v}_M\Vert_{L^2(K)}\le \Lambda_1 h_K^2\Vert D^2(\widehat{v}_M-I_{M}\widehat{v}_M)\Vert_{L^2(K)}$ for any $K\in\mathcal{T}$ and $\widehat{v}_M\in M(\widehat{\mathcal{T}})$ with $\Lambda_1$ of \eqref{eq:C1_h}. 

The following soft analysis briefly accounts for \eqref{eq:C1_h}. 
 Let $K\in \mathcal{T}\setminus\widehat{\mathcal{T}}$ and set $\widehat{w}_M:=(\widehat{v}_M-I_M\widehat{v}_M)|_K$. It holds $\widehat{w}_M\in {M}^\prime(\widehat{\mathcal{T}}(K))$ with the fine triangulation $\widehat{\mathcal{T}}(K):=\{T\in \widehat{\mathcal{T}}:\, T\subset K\}$  for the domain $\textup{int}(K)$ rather than $\Omega$, $\widehat{w}_M(z)=0$ for any node $z\in\mathcal{N}(T)$ and $\intmean_E \partial \widehat{w}_M/\partial \nu_T\,\textup{d}s=0$ for any edge $E\in\mathcal{E}(T)$. Prop. 2.5--2.6 in \cite{Gal15} allow the definition of a conforming companion operator on the fine triangulation of a coarse triangle, 
$\widehat{J}:{M}^\prime(\widehat{\mathcal{T}}(K))\to H^2(K)$ with the properties in \cref{lem:MorleyCompanion} for $\widehat{\mathcal{T}}(K)$. Due to the missing boundary conditions in contrast to \cite{Gal15} the gradient in the new boundary nodes $z\in\widehat{\mathcal{N}}(\partial K)\setminus \mathcal{N}(\partial K)$ is computed by averaging over interior triangles $\nabla (\widehat{J}\,\widehat{v}_M(z))=|\widehat{\mathcal{T}}(z)\cap \widehat{\mathcal{T}}(K)|^{-1}\sum_{T\in\widehat{\mathcal{T}}(z)\cap \widehat{\mathcal{T}}(K)}(\nabla \widehat{v}_M|_T)(z)$.

The triangle inequality reads $\Vert \widehat{w}_M\Vert_{L^2(K)}\le \Vert \widehat{w}_M-\widehat{J}\widehat{w}_M\Vert_{L^2(K)}+\Vert \widehat{J}\widehat{w}_M\Vert_{L^2(K)}$.  \cref{lem:MorleyCompanion}.v 
 proves $\Vert  \widehat{w}_M-\widehat{J}\widehat{w}_M\Vert_{L^2(K)}\lesssim h_K^2\min_{v\in H^2_0(\Omega)}\Vert D^2_{\textup{NC}}(v_M-v)\Vert_{L^2(K)}\le  h_K^2 \Vert D^2_{\textup{NC}}v_M\Vert_{L^2(K)}$ for the first term. Since $I_M\widehat{J}\widehat{w}_M=0$, the error estimate \cite[Thm.3]{CGal14} for the Morley interpolation of $\widehat{J}\widehat{w}_M\in H^2(K)$  followed by the stability property \cite[Prop. 2.6]{Gal15} of the companion operator proves 
\begin{align*}
\Vert \widehat{J}&\widehat{w}_M\Vert_{L^2(K)}=\Vert \widehat{J}\widehat{w}_M-I_M\widehat{J}\widehat{w}_M\Vert_{L^2(K)}\le \kappa_{M}h_K^2\Vert (1-\Pi_0) D^2(\widehat{J}\widehat{w}_M)\Vert_{L^2(K)}\\
&\le \kappa_{M}h_K^2\Vert D^2(\widehat{J}\widehat{w}_M)\Vert_{L^2(K)} 
 \le \kappa_{M}h_K^2\big(\Vert D^2_\textup{NC}(\widehat{J}\widehat{w}_M-\widehat{w}_M)\Vert_{L^2(K)} +\Vert D^2_\textup{NC}\widehat{w}_M\Vert_{L^2(K)}\big)\\
& \lesssim h_K^2 \Vert D^2_\textup{NC}\widehat{w}_M\Vert_{L^2(K)}.
\end{align*}
The combination of these estimates shows $\Vert \widehat{w}_M\Vert_{L^2(K)}\lesssim  h_K^2 \Vert D^2_\textup{NC}\widehat{w}_M\Vert_{L^2(K)}$ for any $K\in\mathcal{T}$. This concludes the proof of \eqref{eq:C1_h}.$\hfill{\Box}$

\subsection{Proof of \eqref{eq:C2_orthogonality}, \eqref{eq:C5_stab}}
Since the Hessian $D^2_\textup{NC}w_M\in P_0(\mathcal{T};\mathbb{R}^{2\times 2})$  is piecewise constant for any $w_M\in{M}(\mathcal{T})\subset P_2(\mathcal{T})$, the identity  $\Pi_0 D^2_\textup{NC}=D^2_\textup{NC}I_{M}$ \cite[(3.1)]{CGal14} proves \eqref{eq:C2_orthogonality} by
$
	a_h( w_{M} , {\widehat{v}}_{M}-   I_M {\widehat{v}}_{M})
	=(D^2_\textup{NC} w_{M}, (1-\Pi_0)D^2_\textup{NC}\widehat{v}_M )_{L^2(\Omega)}=0. 
$ 

The analogue  identity on the refined triangulation $\widehat{\mathcal{T}}$ reads $\widehat{\Pi}_0D^2v=D^2_\textup{NC}\widehat{I}_\textup{NC}v$  for all $v\in H^2_0(\Omega).$ Since  $D^2_\textup{NC} w_M\in P_0(\mathcal{T};\mathbb{R}^{2\times 2})\subset P_0(\widehat{\mathcal{T}};\mathbb{R}^{2\times 2})$, it follows \eqref{eq:C5_stab} for any $T\in\mathcal{T}$ with $\Lambda_5=1$ by
		$
			\Vert D^2_\textup{NC}(\widehat{I}_{M}v-w_M)\Vert_{L^2(T)}
				= \Vert \widehat{\Pi}_0 D^2(v -w_M)\Vert_{L^2(T)}
				\le \Vert D^{2} (v - w_M)\Vert_{L^2(T)}.\hfill\Box
		$
\subsection{Proof of \eqref{eq:C3_vanish}}\label{sec:C3Morley}
Given any $T\in\mathcal{T}\cap\widehat{\mathcal{T}}$ and some $\widehat{v}_M\in M(\widehat{\mathcal{T}})$, it remains to verify that $I_M \widehat{v}_M|_T$ and $\widehat{v}_M|_T$ coincide in the degrees of freedom for the Morley finite element. Since $\partial \widehat{v}_M/\partial \nu_{E}|_{E}\in P_1(E)$ implies $\partial \widehat{v}_M/{\partial\nu_E}(\textup{mid}(E))=\intmean_E {\partial \widehat{v}_M}/{\partial \nu_E}\,\textup{d}s$ for all sides $E\in\mathcal{E}(T)$, the definition of $I_{M}$ shows indeed that the normal derivatives at the edge midpoints $\textup{mid}(E)$ for $E\in\mathcal{E}(T)= \widehat{\mathcal{E}}(T)$ and the values in the vertices $z\in \mathcal{N}(T)= \widehat{\mathcal{N}}(T)$ of $\widehat{v}_M$ coincide with those of $I_{M}\widehat{v}_M $.  $\hfill{\Box}$

\subsection{Proof of \eqref{eq:C4_uhstar}}	\label{sec:C4_Morley}
Given any $u_M\in {M}(\mathcal{T})$ and $\widehat{u}_M^*:=\widehat{I}_{{M}}{J_G}(u_M)$,  \cref{lem:MorleyCompanion}.iv  shows $u_M=I_M{J_G}(u_M)=I_M\widehat{I}_{{M}}{J_G}(u_M)$. This proves \eqref{eq:C4_uhstar}.$\hfill{\Box}$

\subsection{Proof of \eqref{eq:C6_wh}}
Given any $v_M\in M(\mathcal{T})$ and $K\in\mathcal{T}$, \cref{lem:MorleyCompanion} shows that ${J_G}v_M$  and $v_M$ have the same nodal values \ref{MorleyCompanion:Nodal} and integral means of the normal derivatives 
along the edges \ref{MorleyCompanion:Edges}.
 Only the derivatives $\nabla ({J_G}v_M)(z)$ for inner nodes $z\in\mathcal{N}(K)\cap \mathcal{N}(\Omega)$  are computed by averaging $\nabla v_M|_T(z)$ for all $T\in \mathcal{T}(z)$ and so ${J_G}$ is associated with local contributions ${J_{G,K}}$  for any $K\in\mathcal{T}$ in the sense that   
$
	{J_G}(v_M)|_K={J_{G,K}}(v_M|_{\Omega(K)}).
$

For any $w_M\in \big({M}(\mathcal{T})\cap H^2_0(\Omega)\big)|_{\Omega(K)}$, the derivative  $\nabla w_M$ is continuous in ${\Omega(K)}$ and vanishes along boundary edges $E\subset\partial\Omega$, hence  $\nabla w_M(z)=\nabla {J_G}w_M(z)$ for all $z\in\mathcal{N}(\Omega(K))$. The nodal values and the integral means of the normal derivatives of $w_M$ and $J_Gw_M$ coincide by \cref{lem:MorleyCompanion}.i and iv.   Hence, the functions $w_M\in \big(P_2(\mathcal{T})\cap H^2_0(\Omega)\big)|_{\Omega(K)}\subset HCT(\mathcal{T})|_{\Omega(K)}$ and ${J_G}w_M|_{\Omega(K)}\in HCT(\mathcal{T})|_{\Omega(K)}\subset H^2_0(\Omega)|_{\Omega(K)}$ coincide in the degrees of freedom for the $HCT$ finite element. Consequently, $w_M|_T={J_{G,T}} w_M$ for any $T\in\mathcal{T}(\Omega(K))$ proves \eqref{eq:C6_wh}.$\hfill{\Box}$
	
\subsection{Proof of \eqref{eq:C7_ker}} 
The derivative $\nabla_{\textup{NC}} v_M\in CR^1_0(\mathcal{T};\mathbb{R}^2)$  of a $v_M\in M(\mathcal{T})$ is a Crouzeix-Raviart function. 
Therefore, given any $T\in\mathcal{T}$ and $w_M\in M(\mathcal{T})|_{\Omega(T)}$, the arguments of \cref{sec:CR_C7} apply for each component of $\nabla _\textup{NC}w_M$: 
If 
	$\Vert[D^2_\textup{NC} w_M]_E\times \nu_E\Vert^2_{L^2(E)}=0$ for all $E\in\mathcal{E}(\Omega(T))$, then $\nabla w_M\in S^1_0(\mathcal{T};\mathbb{R}^2)|_{\Omega(T)}$. Consequently, $w_M\in \big(M(\mathcal{T})\cap H^2_0(\Omega)\big)|_{\Omega(T)}$. This proves \eqref{eq:C7_ker}.$\hfill{\Box}$

\subsection{\SPnew{Towards application in $3$D}}
The physical application in mind are plate problems, therefore this paper concentrates on the two-dimensional case. 
However, the Morley element is generalized to solve fourth-order elliptic equations in any space dimension in \cite{MX06}. 
Given any $n$-simplex $T\in\mathcal{T}$ with $(n-1)$-dimensional sub-simplices (faces in $3$D) $F\in F(\mathcal{T})$  
and $(n-2)$-dimensional sub-simplices (sides in $3$D) $E\in \mathcal{E}({T})$, \cite[Def. 1]{MX06} introduces the following local 
$|\mathcal{F}(T)|+|\mathcal{E}({T})|=(n+1)+ \binom{n+1}{n-1}= (n+1)(n+2)/2$ 
degrees of freedom for $v\in P_2(\mathcal{T})$  
\begin{align*}
	\intmean_{E} v\,\textup{d}s  \quad\text{ and }\quad \intmean_{F} \frac{\partial v}{\partial\nu}\,\textup{d}s \quad 
		\text{for all }E\in \mathcal{E}(T),\,F\in \mathcal{F}(T),\text{ and } v\in C^1(T). 
\end{align*}
If the integral mean over a node $z\in\mathcal{N}$ for $n=2$ is translated as point evaluation, this is a generalization of the two-dimensional definition. 
In \cite[(9)]{MX06} the dual basis of $M(\mathcal{T})$ in $n=3$ dimension is stated and used to define the standard interpolation 
$I_{M}:\, H^2_0(\Omega)+{M}(\widehat{\mathcal{T}})\to{M}(\mathcal{T})$ for any $v\in H^2_0(\Omega)+{M}(\widehat{\mathcal{T}})$ with
\begin{align*}
	\intmean_{E} v\,\textup{d}s &= \intmean_{E} I_{M}v\,\textup{d}s \quad\text{for any }E\in\mathcal{E}\quad\text{ and }\quad\\
	\intmean_F \frac{\partial I_M(v)}{\partial \nu_F}\,\textup{d}s&=\frac{\partial I_{M}v}{\partial\nu_F}(\textup{mid}(F))= \intmean_F \frac{\partial v}{\partial \nu_F}\,\textup{d}s\quad
			\text{for any }F\in\mathcal{F}.
\end{align*}
An integration by parts proves $\Pi_0 D^2=D_{NC}^2 I_M $
which leads to \eqref{eq:C2_orthogonality} and \eqref{eq:C5_stab} as in $2$D. The condition \eqref{eq:C3_vanish} holds with the same arguments as in \cref{sec:C3Morley}, while  \eqref{eq:C1_h} remains to be discussed. However, \cite[Thm. 3.5]{CH17} for the Crouzeix-Raviart case holds in any space dimension and the gradient $\nabla_{pw}v_M\in CR^1(\mathcal{T};\mathbb{R}^3)$ is a Crouzeix-Raviart function in $n$ components for any $v_m\in M(\mathcal{T})$, hence the authors are optimistic that the proof of \eqref{eq:C1_h} carries over to higher space dimension. Moreover, since \cref{sec:CR_C7} holds for all $n\in\mathbb{N}$, \eqref{eq:C7_ker} follows as above. 

To verify the conditions \eqref{eq:C4_uhstar} and \eqref{eq:C6_wh}  a $C^1$-conforming space in higher dimension has to be chosen. In \cite{Wal14} a composite $C^1$ tetrahedral element $W(\mathcal{T})$ is presented. Thereby each tetrahedron $T\in\mathcal{T}$ is subdivided into four tetrahedra $T_F:=\textup{conv}\{F,\textup{mid}(T)\}\in\mathcal{K}(T)$ with the following $45$ degrees of freedom, for any $v\in W(T)$,
\begin{enumerate}[label=(\arabic*)]
	\item $v(z)$, $\nabla v(z)$ and $D^2 v(z)$ at the four vertices $z\in\mathcal{N}(T)$,
	\item $\nabla v(\textup{mid}(F))\cdot \nu_F$ at the midpoints of the four faces $F\in\mathcal{F}(T)$,
	\item \label{item:Wmidpoint}$v(\textup{mid}(T))$ at the centroid $\textup{mid}(T)$, 
\end{enumerate}
where $v\in P^5(\mathcal{K}(T))\cap C^1(T)\cap C^4(\textup{mid}(T))$ is a piecewise $P_5$ element and the normal derivatives on the faces $\nabla v\cdot \nu_F\in P_3(F)$ are constrained to be cubic along each $F\in\mathcal{F}$.  
The interpolation operator $J_1:M(\mathcal{T})\to W(\mathcal{T})$ is defined by averaging as follows. Given any $v_M\in \mathcal{M}(\mathcal{T})$ define the degrees of freedom  for $W(\mathcal{T})$ by $J_1 D^\alpha v_M(z)=|\mathcal{T}(z)|^{-1} \sum_{T\in\mathcal{T}(z)}D^\alpha v_M|_T(z)$ { for all } $z\in\mathcal{N}(\Omega)$ and $ 0\le|\alpha|\le 1$ and zero otherwise, set $\nabla J_1 v_M(\textup{mid}(F))\cdot \nu_F= \nabla v_M(\textup{mid}(F))\cdot \nu_F${ for all }$F\in\mathcal{F}$, and $J_1 v_M(\textup{mid}(T))=v_M(\textup{mid}(T))${ for all }$T\in\mathcal{T}$.
This companion (with the  local corrections indicated below) satisfies the localisation condition \eqref{eq:C6_wh}. 
To verify \eqref{eq:C4_uhstar} for $\widehat{u}_M^*=\widehat{I}_MJ_2(u_M)$ a condition comparable to \cref{lem:MorleyCompanion}\ref{MorleyCompanion:Identity} would suffice. 
Therefore, the integral means of the function along the edges and of the normal derivatives along the faces have to be corrected without changing any of the degrees of freedom in $W(\mathcal{T})$. 
Due to the degree of freedom in the midpoint of each simplex a refined triangulation is introduced. Let $\mathcal{K}\in\mathbb{T}(\mathcal{T})$ denotes the refinement, where each tetrahedra is divided in four sub-tetrahedra with the centroid as new vertex, i.e., $\mathcal{K}:=\bigcup_{T\in\mathcal{T}}\bigcup_{T_F\in\mathcal{K}(T)}T_F$.
For any $E\in\mathcal{E}(\Omega)$ chose a function $\xi_E\in H^2_0(\Omega)$ with $\intmean_G\xi_E\textup{d}s=\delta_{GE}$ for all $G\in \mathcal{E}$, such that $\textup{supp}(\xi_E)\subset \widehat{\omega}_E :=\textup{int}\big(\bigcup_{K\in \mathcal{K}(E)} K\big)$ and $\nabla \xi_E\cdot\nu_F(\textup{mid}(F))=0$ for all sides $F\in\mathcal{F}$.
Since there exists $0<\varepsilon<\min_{F\in\mathcal{F}}h_F/2$ with  $B_\varepsilon(\textup{mid}(E))\subset \widehat{\omega}_E$, a possible choice is a mollifier $\xi_E\in C^{\infty}(\mathbb{R}^3)$ with $\textup{supp}(\xi_E)\subset B_\varepsilon(\textup{mid}(E))$ such that without loss of generality $\intmean_E \xi_E\,\textup{d}s=1$.
There are also higher-order conforming polynomials that could be chosen for this correction.  For example in \cite[Cor. 2.1]{Zhang09} a $C^1$ conforming element in $P_k(\mathcal{T})$ for $9\le k$ is introduced. For $k=10$ this element has one interior point of each edge $E\in\mathcal{E}$ as degree of freedom. The associated dual basis function $\xi_E\in P_{10}(\mathcal{K})\cap C^1(\Omega)$ normalized such that  $\intmean_E\xi_E\,\textup{d}s=1$ is an other possible choice. 
For any $v_M\in M(\mathcal{T})$ set $$\tilde{J}_1(v_M)= J_1 v_M+\sum_{E\in\mathcal{E}}\Big(\intmean_E (v_M-J_1v_M)\,\textup{d}s\Big)\xi_E\in C^1(\Omega).$$
For the correction of the integral mean of the normal derivatives along the faces choose for each $F\in\mathcal{F}(\Omega)$ a function $\zeta_F\in H^2_0(\Omega)$ with $\textup{supp}(\zeta_F)\subset\widehat{ \omega}_F:=\textup{int}\big(\bigcup_{K\in \mathcal{K}(F)} K\big)$ and $\intmean_G \nabla \zeta_F\cdot\nu_F\,\textup{d}s =\delta_{GF}$ for all $G\in\mathcal{F}$.
Inspired by \cite{Gal15} a piecewise polynomial $\zeta_F\in P_7(\mathcal{K})\cap C^1(\Omega)$  with this attributes comes to mind.
For any $v_M\in M(\mathcal{T})$ set 
$${J_2}(v_M):= \tilde{J}_1 v_M+\sum_{F\in\mathcal{F}}\Big(\intmean_F\nabla(v_M-\tilde{J}_1v_M)\cdot\nu_F\,\textup{d}s\Big)\zeta_F\in C^1(\Omega).$$ 
By construction holds $I_MJ_2 v_M=v_M$ for any Morley function $v_M\in M(\mathcal{T})$, which is the missing equality in the proof of the remaining condition \eqref{eq:C4_uhstar}.  

Other $C^1$-conforming elements, such as for example  the element in \cite{Zhang09} allow for a similar construction of the conforming companion $J_2$ by averaging (and perhaps appropriate corrections).

\section{Refined Analysis}\label{sec:RefinedAnalysis}
This section introduces the  piecewise design of companion operators in \cref{sec:PiecewiseCompanion}--\ref{sec:PiecewiseCompanion_P2}  based on a fixed subset of sides $\mathcal{F}^\prime\subset\mathcal{F}$.  
This leads in \cref{sec:RefinedAnalysisExample}  to the definition of an alternative approximation $\widehat{u}_{h}^*\in V(\widehat{\mathcal{T}})$ to the discrete solution $u_h\in V(\mathcal{T})$  
 in \eqref{eq:error<uhstarError} with  \eqref{eq:C4_uhstar} and to \eqref{eq:uhstar-uhEstimate} with $\mathcal{T}\setminus\widehat{\mathcal{T}}$ replacing $\mathcal{R}$ 
for Crouzeix-Raviart and Morley finite element methods.  A closer look reveals that  merely the jump contribution along coarse-but-not-fine sides $\mathcal{F}\setminus\widehat{\mathcal{F}}$ occur in \eqref{eq:uhstar-uhEstimate}; in fact, 
\begin{align}
	\Vert \widehat{u}_h^*-u_h\Vert_h^2\lesssim 
	\sum_{F\in\mathcal{F}\setminus\widehat{\mathcal{F}}} h_F\Vert[D^m u_h]_F\times\nu_F\Vert^2_{L^2(F)}.\label{eq:Improve_uhstar-uh}
\end{align}
The remaining conditions \eqref{eq:C1_h}--\eqref{eq:C3_vanish} sufficient for \eqref{eq:error<uhstarError}  depend only on the interpolation operators in \cref{sec:CrouzeixRaviart} and \ref{sec:Morley}, so that \eqref{eq:Improve_uhstar-uh} implies the discrete reliability \eqref{eq:dRel} with $\mathcal{T}\setminus\widehat{\mathcal{T}}$ replacing $\mathcal{R}$.

\subsection{Piecewise companion operator for piecewise affines}\label{sec:PiecewiseCompanion}
The piecewise design of a companion operator is based on a set of sides $\mathcal{F}^\prime\subseteq\mathcal{F}$ and its associated sets 
$\mathcal{T}(K,z)$ and $\mathcal{F}(K,z)$ for any simplex $K\in\mathcal{T}$ with vertex $z\in\mathcal{N}(K)$ in the sequel. Recall the set $\mathcal{T}(z):=\{T\in\mathcal{T}:\, z\in\mathcal{N}(T)\}$ of simplices with vertex $z\in\mathcal{N}$  and the set $\mathcal{F}(z):=\{F\in\mathcal{F}:\, z\in\mathcal{N}(F)\}$ of sides with vertex  $z\in\mathcal{N}$. 
\begin{definition}\label{def:A3} 
	Given $\mathcal{F}^\prime\subseteq\mathcal{F}$, a simplex $K\in\mathcal{T}$, and its vertex $z\in\mathcal{N}(K)$, let  
		\begin{align}
			\mathcal{T}(K,z):=\{ T\in\mathcal{T}:&\text{ there exist } T_1,\dots, T_J\in \mathcal{T}(z)\text{ with } T_1=K,\, T_J=T,\notag\\
									&\quad\qquad\text{ and }\partial T_j\cap\partial T_{j+1}\in \mathcal{F}^\prime \text{ for all }j=1,\dots, J-1\}\subseteq\mathcal{T}(z)\label{eq:T(K,z)Def}
		\end{align}
	denote the side-connectivity component with respect to $\mathcal{F}^\prime$ of $K$ in $\mathcal{T}(z)$ with cardinality  $|\mathcal{T}(K,z)|$. 
	Under the same premise let $$\mathcal{F}(K,z):=\mathcal{F}^\prime\cap \{F\in\mathcal{F}(z): F\in \partial T_1\cap\partial T_2 \text{ for }T_1,\,T_2\in\mathcal{T}(K,z)\}$$ 
	denote the set of interior edges in $\mathcal{T}(K,z)$. 
	Abbreviate $\mathcal{F}^\prime(\partial\Omega):=\mathcal{F}^\prime\cap\mathcal{F}(\partial\Omega)$ for the set of boundary sides in $\mathcal{F}^\prime$. (Notice $K\in\mathcal{T}(K,z)$ for any $K\in\mathcal{T}$, $z\in\mathcal{N}(K)$.)
\end{definition} 
There are two extreme examples for the choice of $\mathcal{F}^\prime$ and the applications below concern some intermediate selection in \cref{sec:RefinedAnalysisExample} illustrated in \cref{fig:refinedSideconnectivity}.
\begin{example}
\begin{enumerate}[label=(\alph*), wide]
\item The maximal set $\mathcal{F}^\prime=\mathcal{F}$ means $\mathcal{T}(K,z)=\mathcal{T}(z)$ for any $K\in\mathcal{T}$ and $z\in\mathcal{N}(K)$.
	This choice in \cref{def:LocalCompanionCR_J1}--\ref{def:LocalCompanionCR} leads to the conforming companion operator of \eqref{def:J2}.
	In \cref{def:LocalCompanionMorley_J1}--\ref{def:LocalCompanionMorley} it leads to  $J_G$  in \cref{lem:MorleyCompanion}.
\item If $\mathcal{F}^\prime\cap\mathcal{F}(z)\cap\mathcal{F}(K)=\emptyset$ for $K\in\mathcal{T}$ and  $z\in\mathcal{N}(K)$, then $\{K\}=\mathcal{T}(K,z)$ 
(the condition $\partial T_j\cap\partial T_{j+1}\in \mathcal{F}^\prime $ for $j=1,\dots,J-1$ does not arise for $J=1$); singletons are side connected. 
\end{enumerate}
\end{example}
Any choice of $\mathcal{F}^\prime\subseteq\mathcal{F}$ allows the definition of the local companion operator  below for Crouzeix-Raviart in \cref{def:LocalCompanionCR_J1}--\ref{def:LocalCompanionCR} ($k=m=1$, $n\ge 2$) and for the Morley finite element in \cref{def:LocalCompanionMorley_J1}--\ref{def:LocalCompanionMorley} ($k=m=2=n$).
\begin{definition}[Local companion $J_1$ for piecewise affines]\label{def:LocalCompanionCR_J1}
	Suppose $\mathcal{T}\in\mathbb{T}$ and the sets $\mathcal{T}(K,z)$ are as in \cref{def:A3} associated with $\mathcal{F}^\prime\subseteq\mathcal{F}$ to
	define $J_1: P_1(\mathcal{T})\to P_1(\mathcal{T})$ as follows. 
	For any $v_1\in P_1(\mathcal{T})$ and $K\in\mathcal{T}$  define  $(J_1v_1)|_K\in P_1(K)$ 
	through linear interpolation in $K$ of the nodal 
	values 
	\begin{align}
		(J_{1}v_1)|_K(z):= 
				\begin{cases}
				0 \qquad\text{if }z\in\mathcal{N}(F) \text{ for some }F\in\mathcal{F}(K)\cap \mathcal{F}^\prime(\partial\Omega),\\
				|\mathcal{T}(K,z)|^{-1}\sum_{T\in\mathcal{T}(K,z)}v_1|_T(z)\qquad\text{else}
				\end{cases} \label{eq:defJ1P1}
	\end{align}
	at the $n+1$ vertices $z\in\mathcal{N}(K)$ of $K$. 
\end{definition}	
	The values at e.g. interior vertices are computed by averaging over the side-connected $\mathcal{T}(K,z)\subseteq\mathcal{T}(z)$ of cardinality 
	$|\mathcal{T}(K,z)|$. 
	The first alternative in \eqref{eq:defJ1P1} at all vertices of a boundary side $F\in\mathcal{F}^\prime(\partial\Omega)$  enforces homogeneous 
	boundary conditions. 
	The piecewise affine $J_1v_1$  is discontinuous and violates homogeneous boundary conditions in general. 
	For $n\ge 2$ the  normalized side-bubbles 
	$$b_F:=\big(\prod_{z\in\mathcal{N}(F)}\varphi_z\big)/\intmean_F\big(\prod_{z\in\mathcal{N}(F)}\varphi_z\big)\,\textup{d}s\in S^{n}(\mathcal{T})\quad\text{ for }F\in\mathcal{F}$$
	utilize  the nodal basis-function  $\varphi_z\in S^1(\mathcal{T})=P_1(\mathcal{T})\cap C(\bar{\Omega})$  
	associated to $z\in\mathcal{N}$.
	The subsequent correction 
	assures that the operator $J_2: P_1(\mathcal{T})\to P_n(\mathcal{T})$ preserves the integral means of $v_1\in P_1(\mathcal{T})$ along all sides $F\in\mathcal{F}$. 
	 
\begin{definition}[Local companion $J_n$ for piecewise affines]\label{def:LocalCompanionCR}	
	For any $K\in\mathcal{T}$, $v_1\in P_1(\mathcal{T})$, and  $J_{1}v_1$ of \cref{def:LocalCompanionCR_J1}, set 
	\begin{align*} 
	(J_nv_1)|_K:=(J_{1}v_1)|_K+\sum_{F\in\mathcal{F}(K)}\bigg(\intmean_F(v_1-J_{1}v_1)|_K\,\textup{d}s\bigg)b_F|_K\in 
	P_n(K). 
	\end{align*} 
\end{definition}
The following properties of the companion operators from \cref{def:LocalCompanionCR_J1}--\ref{def:LocalCompanionCR} will be employed throughout this section. 
\begin{lemma}[Properties of $J_1,\,J_n$ for piecewise affines]\label{rem:jumps_J1P1andJnCR}
	\begin{enumerate}[label=(\alph*), wide] 
	\item\label{item:remJ1P1_jumps} 
		 Given any $v_1\in P_1(\mathcal{T})$, the jump $[J_1v_1]_F=0$ of $J_1v_1$ vanishes along any $F\in\mathcal{F}^\prime$.
		 In particular, the companion $J_1v_1|_{\omega_F}\in S^1(\mathcal{T}(F))$ is continuous along any 
		 $F\in\mathcal{F}^\prime\cap \mathcal{F}(\Omega)$ and vanishes along $F\in\mathcal{F}^\prime(\partial\Omega)$. 
		 
    \item\label{item:remJnCR_jumps} 
		Given any $v_{1}\in P_1(\mathcal{T})$, the companion $J_nv_{1}$ 
    		preserves the integral mean $\intmean_F(J_nv_{1})|_K\,\textup{d}s=\intmean_F v_1|_K\,\textup{d}s$ along any side $F\in\mathcal{F}(K)$ in $K\in\mathcal{T}$ and  the jump $[J_nv_1]_F=0$ vanishes
    		along $F\in\mathcal{F}^\prime$.
    \item\label{item:remJnP1_isolated}  If a simplex $K\in\mathcal{T}$ is isolated in the sense that $\mathcal{F}(K)\cap\mathcal{F}^\prime=\emptyset$, 	
			\cref{def:LocalCompanionCR_J1}--\ref{def:LocalCompanionCR} imply $v_1|_K=(J_{1}v_1)|_K=(J_{n}v_1)|_K$ for all $v_1\in P_1(\mathcal{T})$. 
	\end{enumerate}
\end{lemma}
\begin{proofof}\textit{\ref{item:remJ1P1_jumps}.}
		  For any interior side $F\in\mathcal{F}^\prime\cap \mathcal{F}(\Omega)$ with side-patch 
		  $\omega_F=\textup{int}(T_+\cup T_-)$ and $\mathcal{T}(F):=\{T_+,T_-\}$, 
		  \cref{def:A3} implies  $\mathcal{T}(T_+,z)=\mathcal{T}(T_-,z)$ for any vertex $z\in\mathcal{N}(F)$. 
		  Hence, \cref{def:LocalCompanionCR_J1} shows $(J_{1}v_{1})|_{T_+}(z)=(J_{1}v_{1})|_{T_-}(z)$ for all $z\in\mathcal{N}(F)$. 
		  Along any boundary side $F\in\mathcal{F}^\prime(\partial\Omega)$, $J_1v_1|_F=0$  vanishes by
	   	  \cref{def:LocalCompanionCR_J1}. 
		  This proves \ref{item:remJ1P1_jumps}. 
\end{proofof}
\begin{proofof}\textit{\ref{item:remJnCR_jumps}.}
 		 Since $\intmean_{E}b_F\,\textup{d}s=\delta_{EF}$ and $\textup{supp}(b_F)=\omega_F$ for any $E,F\in\mathcal{F}$, the operator $J_n$ in  \cref{def:LocalCompanionCR} preserves the integral means. 
 			 The continuity of $b_F\in C({\Omega})$ and \ref{item:remJ1P1_jumps} imply $[J_nv_1]_F=0$ for any $F\in\mathcal{F}^\prime$.
\end{proofof}
\begin{proofof}\textit{\ref{item:remJnP1_isolated}.}
		  This is elementary for $\mathcal{T}(K,z)=\{K\}$ for all $K\in\mathcal{T}$ with $\mathcal{F}(K)\cap\mathcal{F}^\prime=\emptyset$.
\end{proofof}
The following theorem provides a local a posteriori approximation error estimate for the operator $J_n$ of \cref{def:LocalCompanionCR}; recall $\mathcal{F}^\prime(\partial\Omega):=\mathcal{F}^\prime\cap\mathcal{F}(\partial\Omega)$ and $\mathcal{F}(K,z):=\mathcal{F}^\prime\cap \{F\in\mathcal{F}(z): F\in \partial T_1\cap\partial T_2 \text{ for }T_1,\,T_2\in\mathcal{T}(K,z)\}$ from \cref{def:A3}. 
\begin{theorem}[approximation error]\label{thm:P1-J2Fehler=Jumps}
	Given \ref{triang:A1}, $K\in\mathcal{T}$, and $v_1\in P_1(\mathcal{T})$, the companion  $J_nv_1$ of \cref{def:LocalCompanionCR} satisfies
		\begin{align*}
			C_n^{-1}h_K^{-1}\Vert v_1-J_n v_1\Vert^2_{L^2(K)}\le 
															\sum_{z\in\mathcal{N}(K)}\sum_{F\in\mathcal{F}(K,z)} \Vert [v_1]_F\Vert_{L^2(F)}^2+
															\sum_{F\in\mathcal{F}(K)\cap\mathcal{F}^\prime(\partial\Omega)} \Vert v_1\Vert_{L^2(F)}^2
		\end{align*}
	The constant $C_n$ soley depends on $n$ and $M_2$ from \ref{triang:A1}.
\end{theorem}
\begin{proof}
\begin{enumerate}[label=\emph{Step \arabic*}, wide]
\item\label{step:J1-J2_P1}\hspace{-.25cm}\emph{.}\hspace{.2cm} 
\cref{def:LocalCompanionCR} and the triangle inequality show
		\begin{align*}
			\Vert v_1-J_{n}v_1\Vert_{L^2(K)} 
				&\le  \Vert v_1- J_{1}v_1\Vert_{L^2(K)}+\bigg\Vert\sum_{F\in\mathcal{F}(K)}\bigg|\intmean_F(v_1-J_{1}v_1)|_K\,\textup{d}s\bigg| b_F\bigg\Vert_{L^2(K)}.
		\end{align*}
		The local mass matrix for normalized bubble functions in $K\in\mathcal{T}$ is SPD and reads 
		\begin{align*}
			B(K):=\Big(\int_K b_{E}b_{F}\,\textup{d}x\Big)_{E,F\in\mathcal{F}(K)}
					=\bigg(\frac{2^{n-3}((2n)!)^2|K|}{(3n)!n!}{(1+\delta_{EF})}\bigg)_{E,F\in\mathcal{F}(K)}\in\mathbb{R}^{(n+1)\times(n+1)}.
		\end{align*}
		It has the multiple eigenvalue $\lambda_{\min}:={2^{n-3}((2n)!)^2|K|}/{((3n)!n!)}$ and the simple eigenvalue $\lambda_{\max}:=(n+2)\lambda_{\min}$. This proves 
		\begin{align*}
			\bigg\Vert\sum_{F\in\mathcal{F}(K)}\bigg|\intmean_F(v_1-J_{1}v_1)|_K\,\textup{d}s\bigg| b_F\bigg\Vert_{L^2(K)}^2
				&\le \lambda_{\max} \sum_{F\in\mathcal{F}(K)}\bigg|\intmean_F(v_1-J_{1}v_1)|_K\,\textup{d}s\bigg|^2.
		\end{align*}		
		Lemma \ref{lem:AppendixIntmeanP1} of the appendix quantifies the constant in the discrete trace inequality and implies 
		\begin{align*}
			 \sum_{F\in\mathcal{F}(K)}\bigg|\intmean_F(v_1-J_{1}v_1)|_K\,\textup{d}s\bigg|^2
				&\le \frac{ (n+1)}{|K|} \Vert v_1-J_{1}v_1\Vert_{L^2(K)}^2. 
		\end{align*}			
		Consequently, the constant $C_J:= 1+ {(2n)!}\sqrt{\frac{2^{n-3}(n+1)(n+2)}{(3n)!n!}}$ satisfies 
		\begin{align*}
			\Vert v_1-J_{n}v_1\Vert_{L^2(K)}\le C_J\Vert v_1-J_{1}v_1\Vert_{L^2(K)}.
		\end{align*}
		(It holds  $C_J= 1+2\sqrt{3/5}\le 2.5492$ for $n=2$ and $C_J= 1+10/\sqrt{21}\le 3.1822$ for $n=3$.)
				
\item\label{step:eKDef}\hspace{-.25cm}\emph{.}\hspace{.2cm}  
		For any $z\in\mathcal{N}(K)$ set $e_K(z):=(v_1-J_{1}v_1)|_K(z)$ with the associated coefficient vector $e_K:=(e_K(z))_{z\in\mathcal{N}(K)}\in\mathbb{R}^{n+1}$. 
   	 	The local mass matrix for the $P_1$ conforming FEM is SPD and reads 
   	 	\begin{align}
   	 	M(K)=\bigg(\frac{|K|(1+\delta_{jk})}{(n+1)(n+2)}\bigg)_{j,k=1,\dots, n+1}\in\mathbb{R}^{(n+1)\times(n+1)}.\label{eq:MassMatrixP1}
   	 	\end{align}
    	The simple eigenvalue $|K|/(n+1)$ of $M(K)$ has the eigenvector $(1,\dots,1)\in \mathbb{R}^{n+1}$. The eigenvalue $|K|/((n+1)(n+2))$ has the
    	$n$-dimensional eigenspace of vectors in $\mathbb{R}^{n+1}$ perpendicular to $(1,\dots,1)$.
    	Therefore, the affine function $(v_1-J_{1}v_1)|_K\in P_1(K)$ satisfies
    	\begin{align*}
    		\Vert v_1-J_{1}v_1\Vert_{L^2(K)}^2=e_K\cdot M(K)e_K\le\frac{ |K|}{n+1} \sum_{z\in\mathcal{N}(K)}e_K(z)^2.
    	\end{align*}
    
\item\label{step:jumpInner}\hspace{-.25cm}\emph{.}\hspace{.2cm} 
	 	Given $z\in\mathcal{N}(K)$ with $(J_{1}v_1)|_K(z):=j^{-1}\sum_{T\in\mathcal{T}(K,z)}v_1|_T(z)$ for $j=|\mathcal{T}(K,z)|\le M_2$. 
		Choose an enumeration $\{T_1,\dots,T_j\}$ of $\mathcal{T}(K,z)$ such that the values $x_k:=(J_{1}v_1)|_K(z)-v_1|_{T_k}(z)\in\mathbb{R}$ for $k=1,\dots, j$ 
		are ordered in the sense that 
		$x_1\le x_2\le \dots\le x_j$. The definition of $(J_{1}v_1)|_K(z)$ guarantees that the sum $\sum_{k=1}^j x_k=0$ vanishes. In an abstract notation,  
		Lemma \ref{lem:2DConstantSide} of the appendix implies the last inequality (with the displayed constant) in 
			 	\begin{align*}
					|e_K(z)|^2& 
							   \le\max_{T\in\mathcal{T}(K,z)}\big|v_1|_T(z)-(J_{1}v_1)|_K(z)\big|^2=\max_{1\le k\le j}|x_k|^2
							   \le \frac{(j-1)(2j-1)}{6j} {\sum_{k=1}^{j-1}|x_{k+1}-x_k|^2}. 
				\end{align*}		
				Let $\mathcal{J}:=\big\{\{\alpha,\beta\}:\,T_\alpha,T_\beta\in\mathcal{T}(K,z)\text{ and } \partial T_\alpha\cap \partial T_\beta\in\mathcal{F}^\prime\}$ 
				denote the set of unordered index pairs  of all simplices in $\mathcal{T}(K,z)$ which share as side in $\mathcal{F}^\prime$.
				The choice of $\mathcal{T}(K,z)$ in \eqref{eq:T(K,z)Def} implies  that $\mathcal{J}$ is connected, in the sense that 
				for all $\alpha,\beta\in\{1,\dots,j\}$ and $\alpha\not = \beta$ there are $k\in\mathbb{N}$ pairs 
				$\{\alpha_1,\alpha_2\},\,\{\alpha_2,\alpha_3\},\,\dots,\,\{\alpha_{k},\alpha_{k+1}\}\in\mathcal{J}$ with $\alpha_1=\alpha$ and $\alpha_{k+1}=\beta$. 
				Lemma \ref{lem:minimalSum} of the appendix implies the first inequality in 
				\begin{align*}
					 			\sum_{k=1}^{j-1}|x_{k+1}-x_k|^2
								\le\hspace{-1mm}\sum_{\{\alpha,\beta\}\in\mathcal{J}} |x_{\alpha}-x_\beta|^2
								=\hspace{-1mm}\sum_{\{\alpha,\beta\}\in\mathcal{J}}\big|(v_1|_{T_{\alpha}})(z)-(v_1|_{T_\beta})(z)\big|^2				
								=\hspace{-1mm}{\sum_{F\in\mathcal{F}(K,z)}\hspace{-.5mm}|[v_1]_F(z)|^2}.
				\end{align*}
				Consequently, $|e_K(z)|^2\le {(j-1)(2j-1)}/{(6j)}\,{\sum_{F\in\mathcal{F}(K,z)}|[v_1]_F(z)|^2}$. Note, 
				$j\hspace{-.5mm}=\hspace{-.5mm}|\mathcal{T}(K,z)|\hspace{-.5mm}\le M_2$ is uniformly 
				bounded for any $K\in\mathcal{T}\in\mathbb{T}$, $z\in\mathcal{N}(K)$.
					
\item\label{step:jumpBoundary}\hspace{-.25cm}\emph{.}\hspace{.2cm} 
				If $z\in\mathcal{N}(K)\cap\mathcal{N}(\partial\Omega)$ belongs to a boundary side $F\in\mathcal{F}(K)\cap\mathcal{F}^\prime(\partial\Omega)$ and 						
				$(J_{1}v_1)|_K(z):=0$, the jump definition guarantees 
	 			$|e_K(z)| =|v_1|_K(z)|=|[v_1]_F(z)|$.
	 			
\item\label{step:inverseEstimate_P1}\hspace{-.25cm}\emph{.}\hspace{.2cm} 		
				Corollary \ref{cor:AppendixP1inNode} of the appendix provides the estimate $|F||[v_1]_F(z)|^2\le n^2 \Vert [v_1]_F\Vert_{L^2(F)}^2$ 
				for $[v_1]_F\in P_1(F)$ on any side $F\in\mathcal{F}$ with vertex $z\in\mathcal{N}(F)$.
				
\item \hspace{-.25cm}\emph{.}\hspace{.2cm} 	
				Set $M:=\max\{n,(M_2-1)(2M_2-1)/({6M_2})\}$. The combination of  \ref{step:J1-J2_P1}--\emph{5}  shows 
				\begin{align*}
				\Vert v_1-J_{n}v_1\Vert_{L^2(K)}^2
						&\le
				 \frac{ C_J^2 M n^2 |K|}{n+1}\bigg(\sum_{z\in\mathcal{N}(K)}\sum_{F\in\mathcal{F}(K,z)} \frac{\Vert [v_1]_F\Vert_{L^2(F)}^2}{|F|}
				 									+\hspace{-1mm}\sum_{F\in\mathcal{F}(K)\cap\mathcal{F}^\prime(\partial\Omega)} \hspace{-1mm} \frac{\Vert [v_1]_F\Vert_{L^2(F)}^2}{|F|}	\bigg).	
				\end{align*} 
				Let $\varrho_F=n|K|/|F|\le h_K$ be the height of the vertex $P_F$ opposite to the side $F$ in the simplex 
				$K=\textup{conv}\{F,P_F\}$. 
				This proves the theorem with $C_n:=C_J^2M\,{n}/({n+1}).$
\end{enumerate}
\end{proof}
\begin{remark}[$C_n$ for Crouzeix-Raviart]\label{rem:CR-TangentialJump}
				For any Crouzeix-Raviart function  $v_{CR}\in CR^1_0(\mathcal{T})$ the integral mean of the jump $\intmean_F [v_{CR}]_F\,\textup{d}s=0$ 
				vanishes along any side $F\in\mathcal{F}$ of diameter $h_F:=\textup{diam}(F)\le h_K$. The Poincar\'e inequality (with Payne-Weinberger constant) implies 
				\begin{align*}
					 \Vert [v_{CR}]_F\Vert_{L^2(F)}= \Big\Vert  [v_{CR}]_F-\intmean_F [v_{CR}]_F\,\textup{d}s\Big\Vert_{L^2(F)}
					 \le h_F \pi^{-1}\Vert  [\nabla_{NC} v_{CR}]_F\times \nu_F\Vert_{L^2(F)}.
				\end{align*}
				This and \cref{thm:P1-J2Fehler=Jumps} show for any $K\in\mathcal{T}$, $v_{CR}\in CR^1_0(\mathcal{T})$, the companion $(J_{n}v_{CR})|_K$ from 
				\cref{def:LocalCompanionCR}, and ${C}^\prime_{n}:= C_n/\pi^2$ that
				\begin{align}
				{{C}^{\prime}}_{n}^{-1} h_K^{-2}\Vert v_{CR}-J_{n}v_{CR}\Vert_{L^2(K)}^2
				\le & \sum_{z\in\mathcal{N}(K)}\sum_{F\in\mathcal{F}(K,z)}  h_F\Vert [\nabla_{NC}v_{CR}]_F\times\nu_F\Vert_{L^2(F)}^2
				\notag\\
					&\qquad+ \sum_{F\in\mathcal{F}(K)\cap \mathcal{F}^\prime({\partial\Omega}) } h_F \Vert [\nabla_{NC}v_{CR}]_F\times\nu_F\Vert_{L^2(F)}^2. \label{eq:J2-jump_CR}
				\end{align}
				The continuity of $v_{CR}\in CR^1_0(\mathcal{T})$ in face midpoints guarantees for each $F\in\mathcal{F}$ that  the jump 
				$[v_{CR}]_F(x)=[\nabla_{\textup{NC}} v_{CR}]_F\cdot(x-\textup{mid}(F))$ at $x\in F$. 
				The orthogonality $(z-\textup{mid}(F))\cdot \nu_F=0$  
				and 
				$|z-\textup{mid}(F)|\le h_F (n-1)/n$  for all $z\in\mathcal{N}(F)$ 
				result in 
    				\begin{align*}
    					|[v_{CR}]_F(z)|&\le |[\nabla_{\textup{NC}} v_{CR}]_F\times\nu_F|\,|z-\textup{mid}(F)|
    					\le \frac{n-1}{n} h_F |[\nabla_{\textup{NC}} v_{CR}]_F\times\nu_F| \\
    					&=\frac{(n-1)h_F}{n{\sqrt{|F|}}} \Vert[\nabla_{\textup{NC}} v_{CR}]_F\times\nu_F\Vert_{L^2(F)}.
    				\end{align*} 
    			This replaces \ref{step:inverseEstimate_P1} in the proof of \cref{thm:P1-J2Fehler=Jumps} 
				and so leads to \eqref{eq:J2-jump_CR} with
				$${C}^{\prime}_{n}:=C_J^2M\frac{ (n-1)^2 }{n^3(n+1)}<\frac{C_n}{\pi^2}.$$ 
				For $n=2$,  ${C}^{\prime}_{2}\le (1+2\sqrt{3/5})^2 \max \{2,{(M_2-1)(2M_2-1)}/{(6M_2)} \}/24$  
				(and ${C}^{\prime}_{2}\le 0.5924$ for a triangulation in 
				right isosceles triangles or more general with $M_2\le 8$). 
				
\end{remark}

\subsection{Piecewise companion operator for piecewise quadratics}\label{sec:PiecewiseCompanion_P2}
In the case of piecewise quadratic polynomials we restrict the analysis to  $n=2$, where $\mathcal{T}\in \mathbb{T}$ is a regular triangulation of $\Omega\subset\mathbb{R}^2$  into triangles and  let $\mathcal{E}$ denote the set of all edges (rather than writing $\mathcal{F}\equiv \mathcal{E}$ in $2$D).  
The local version of the HCT finite element space in \eqref{eq:HCT} without boundary conditions reads 
$${HCT}^\prime(K):={HCT}^\prime(\{K\}):=\big\{v\in H^2(K):\ v\in P_3(\mathcal{K}(K))\big\}\qquad\text{for any }K\in\mathcal{T}.$$ 

\begin{definition}[Local companion $J_1$ for piecewise quadratics]\label{def:LocalCompanionMorley_J1}
	Suppose $\mathcal{T}\in\mathbb{T}$ and $\mathcal{T}(K,z)$ associated with $\mathcal{E}^\prime\subseteq\mathcal{E}$ as in \cref{def:A3}.
	Define  $J_1:P_2(\mathcal{T})\to\prod_{K\in\mathcal{T}} HCT^\prime(K)$ as follows.
	For any $v_2\in P_2(\mathcal{T})$ and any triangle $K\in\mathcal{T}$  define $(J_{1}v_2)|_K\in HCT^\prime(K)$ through the $HCT$ interpolation of the degrees 	
	of freedom at the three midpoints $\textup{mid}(E)$ of the edges $E\in\mathcal{E}(K)$ and the three vertices $z\in\mathcal{N}(K)$ of $K$ by
	\begin{align}
		\frac{\partial (J_{1}v_2)|_K}{\partial \nu_E} (\textup{mid}(E))&= \frac{\partial v_2|_K}{\partial \nu_E} (\textup{mid}(E))\quad \text{ for any }E\in\mathcal{E}(K),\notag \\
		(J_{1}v_2)|_K(z)&=v_2|_K(z)\qquad \qquad\qquad\ \, \text{ for any }z\in\mathcal{N}(K),\notag \\
		\nabla (J_{1}v_2)|_K(z)&=\begin{cases}
											0 \ \,\qquad\text{if }z\in\mathcal{N}(E) \text{ for some }E\in\mathcal{E}(K)\cap \mathcal{E}^\prime(\partial\Omega),\\
											|\mathcal{T}(K,z)|^{-1}{\sum_{T\in\mathcal{T}(K,z)}}(\nabla v_1)|_T(z)\qquad\qquad\text{else}.
											\end{cases}\label{eq:defJ1P2_Gradient}
	\end{align}
\end{definition}	
	The function $J_{1}v_2$ from \cref{def:LocalCompanionMorley_J1} inherits  the nodal values as well as 
	the values of the normal derivatives in the edge-midpoints from $v_2\in P_2(\mathcal{T}) $.
	The values of the derivative e.g. at all interior vertices are computed by averaging over the side-connected  $\mathcal{T}(K,z)\subseteq\mathcal{T}(z)$ of 
	cardinality $|\mathcal{T}(K,z)|$. 	
	The first alternative in \eqref{eq:defJ1P2_Gradient}, $\nabla (J_{1}v_2)|_K(z)=0$ at all vertices $z\in\mathcal{N}(E)$ of an edge  
	$E\in\mathcal{E}^\prime(\partial\Omega):=\mathcal{E}^\prime\cap\mathcal{E}(\partial\Omega)$,  enforces a vanishing derivative  along an edge  
	$E\subset\partial K$ 
	with $\partial v_2/\partial\nu_E(\textup{mid}(E))=0$. 
	The composition $J_1v_2$ is piecewise $HCT$, but is  discontinuous and violates 	homogeneous boundary conditions in general. \bigskip
	
	The  normalized  edge-bubble  $b_{E,K}:=30(\nu_K\cdot\nu_E)\textup{dist}(z_3,E)\varphi_1^2\varphi_2^2\varphi_3\in P_5(\mathcal{T})$  
	is defined for $K=\textup{conv}\{z_1,z_2,z_3\}=\textup{conv}\{E,z_3\}$ with vertex $z_3$ opposite to $E$ in $K$. The nodal basis function $\varphi_j\equiv\varphi_{z_j}\in S_1(\mathcal{T})$ is associated with $z_j$. 
	The subsequent correction assures that the operator $J_2: P_1(\mathcal{T})\to P_5(\mathcal{T})+\prod_{K\in\mathcal{T}} HCT^\prime(K)$ preserves the integral means of the normal derivatives $\partial v_2/\partial\nu_E$ along all 
	edges $E\in\mathcal{E}$.
\begin{definition}[Local companion $J_2$ for piecewise quadratics]\label{def:LocalCompanionMorley}	
	For any $K\in\mathcal{T}$, $v_2\in P_2(\mathcal{T})$, and $J_{1}v_2$ as in \cref{def:LocalCompanionMorley_J1} set
	\begin{align*}
	(J_2v_2)|_K:=(J_{1}v_2)|_K+\sum_{E\in\mathcal{E}(K)}\bigg(\intmean_E\frac{\partial(v_2-J_{1}v_2)|_K}{\partial\nu_E}\,\textup{d}s\bigg)b_{E,K}\in HCT(K)+P_5(K).
	\end{align*}	 
\end{definition}
The following properties of the companion operators from \cref{def:LocalCompanionMorley_J1}--\ref{def:LocalCompanionMorley} will be applied throughout this section. 
\begin{lemma}\label{rem:jumps_J1MandJ2M}
	\begin{enumerate}[label=(\alph*), wide] 
	\item\label{item:remJ1M_jumps} For any Morley function $v_M\in M(\mathcal{T})$ the jumps  $[J_1v_M]_E$ and $[\nabla J_1v_M]_E$ vanish along any $E\in\mathcal{E}^\prime$. 
		 In particular, the companion 
		  $(J_{1}v_{M})|_{\omega_E}\in HCT^\prime(\mathcal{T}(E)):=\{v\in H^2(\omega_E): v|_T\in HCT^\prime (T) \text{ for any } T\in \mathcal{T}(E)\}$  is continuously differentiable along any 
		 $E\in\mathcal{E}^\prime\cap\mathcal{E}(\Omega)$; $J_1v_M|_E=0$ and $\nabla J_1v_M|_E=0$ 
		 vanish along any $E\in\mathcal{E}^\prime(\partial\Omega)$.  	
	
    \item\label{item:remJ2M_jumps} Given any $v_{M}\in M(\mathcal{T})$, the companion $J_2v_{M}$ is continuous at the vertices 
    				$z\in\mathcal{N}$ and at the midpoints of the edges 
    				$E\in\mathcal{E}$;  the jumps $[J_2v_{M}]_E$ and  $[\nabla J_2v_M]_E$ vanish along  $E\in\mathcal{E}^\prime$.	 
		
    \item\label{item:remJ2v2_noEdge} If a simplex $K\in\mathcal{T}$ is isolated in the sense that $\mathcal{E}(K)\cap\mathcal{E}^\prime=\emptyset$, 	
			\cref{def:LocalCompanionMorley_J1}--\ref{def:LocalCompanionMorley} imply  $v_2|_K=(J_{1}v_2)|_K=(J_{2}v_2)|_K$ for all $v_2\in 
			P_2(\mathcal{T})$. 
	\end{enumerate}
\end{lemma}
\begin{proofof}\textit{\ref{item:remJ1M_jumps}.} 
		 For any interior edge $E\in\mathcal{E}^\prime\cap \mathcal{E}(\Omega)$ with edge-patch $\omega_E=\textup{int}(T_+\cup T_-)$ and
		  $\mathcal{T}(E)=\{T_+,T_-\}$, 
		 \cref{def:A3} implies $\mathcal{T}(T_+,z)=\mathcal{T}(T_-,z)$ for any vertex $z\in\mathcal{N}(E)$. 	 
		 Any Morley function $v_{M}\in M(\mathcal{T})$ is continuous at the vertices $z\in\mathcal{N}(E)$ and  the normal derivative 
		 $\partial v_M/\partial\nu_E$ is continuous at the edge midpoint $\textup{mid}(E)$. Since 
		 the coinciding input data at the vertices $z\in\mathcal{N}(E)$ lead to
		 $\nabla_{\textup{NC}}(J_{1}v_{M})|_{T_+}(z)=\nabla_{\textup{NC}}(J_{1}v_{M})|_{T_-}(z)$ as well, the jumps $[J_1v_M]_E$  and $[\nabla J_1v_M]_E$  vanish   
		 along any
		 interior edge $E\in\mathcal{E}^\prime\cap \mathcal{E}(\Omega)$.
		 The boundary conditions of $v_M\in M(\mathcal{T})$ and 
		 \cref{def:LocalCompanionMorley_J1} 
		 directly imply $J_1v_M|_E=0$ and $\nabla J_1v_M|_E=0$  along a boundary edge $E\in\mathcal{E}^\prime(\partial\Omega)$.
		 This concludes the proof of \ref{item:remJ1M_jumps}.
\end{proofof}
\begin{proofof}\textit{\ref{item:remJ2M_jumps}.}
	The edge-bubbles $b_{E,T}\in P_5(\mathcal{T})\cap H^2(\Omega)$ satisfy $\intmean_E\partial b_{F,K}/\partial\nu_E\,\textup{d}s=\delta_{EF}$, $\textup{supp}(b_{F,K})=K$, and 
	$b_{E,K}(z)$ and $\nabla b_{E,K}(z)$ vanish at any $z\in\mathcal{N}$ for 
	$E,F\in\mathcal{E}$. Hence, $J_2$ preserves the integral means of the normal derivatives and  \ref{item:remJ2M_jumps}  follows directly from 
    \ref{item:remJ1M_jumps}. 
\end{proofof}
\begin{proofof}\textit{\ref{item:remJ2v2_noEdge}.}
		  This is elementary for $\mathcal{T}(K,z)=\{K\}$ for all $K\in\mathcal{T}$ with $\mathcal{E}(K)\cap\mathcal{E}^\prime=\emptyset$.
\end{proofof}
		The following theorem establishes a local a posteriori approximation error estimate for the operator $J_2$ of \cref{def:LocalCompanionMorley};
		recall that $h_E=|E|$ is the length of the edge $E\in\mathcal{E}$, 
		$\mathcal{E}^\prime(\partial\Omega):=\mathcal{E}^\prime\cap\mathcal{E}(\partial\Omega)$, and 
		$\mathcal{E}(K,z):=\mathcal{E}^\prime\cap \{E\in\mathcal{E}(z): E\in \partial T_1\cap\partial T_2 \text{ for }T_1,\,T_2\in\mathcal{T}(K,z)\}$ in \cref{def:A3}.
\begin{theorem}\label{thm:P2-J2Fehler=Jumps}
	Given \ref{triang:A1}, $K\in\mathcal{T}$, and $v_2\in P_2(\mathcal{T})$, the local companion $J_{2}v_2$ of \cref{def:LocalCompanionMorley} satisfies
		\begin{align*}
			h_K^{-4}\Vert v_2-J_{2} v_2\Vert^2_{L^2(K)}\lesssim&\hspace{-1mm}
															\sum_{z\in\mathcal{N}(K)}\sum_{E\in\mathcal{E}(K,z)}\hspace{-1mm}h_E^{-1} \Vert [\nabla_{\textup{NC}}v_2]_E\Vert_{L^2(E)}^2
															\hspace{-1mm}
															+ \hspace{-1.5mm}\sum_{E\in \mathcal{E}(K)\cap \mathcal{E}^\prime({\partial\Omega})}\hspace{-1mm}h_E^{-1} 
															 \Vert  \nabla_{\textup{NC}}v_2\Vert_{L^2(E)}^2.
		\end{align*}
\end{theorem}
\begin{proof}
	\begin{enumerate}[label=\emph{Step \arabic*}, wide]
		\item\label{step:MorleyJ2=J1}\hspace{-.25cm}\emph{.}\hspace{.2cm}   
				\cref{def:LocalCompanionMorley} and the triangle inequality show
				\begin{align*}
					\Vert v_2-J_{2}v_2\Vert_{L^2(K)} 
						&\le  \Vert v_2-J_{1}v_2\Vert_{L^2(K)}+\sum_{E\in\mathcal{E}(K)}\bigg|\intmean_E\partial(v_2-J_{1}v_2)|_K/\partial\nu_E\,\textup{d}s\bigg| \Vert b_{E,K}\Vert_{L^2(K)}.
				\end{align*}			
				It holds $\Vert b_{E,K}\Vert_{L^2(K)}=2\sqrt{|K|}^3/(|E|\sqrt{2310})\le h_K \sqrt{|K|/2310}$ 
				for any $E\in\mathcal{E}(K)$. Lemma \ref{lem:AppendixIntmeanP1} in the appendix quantifies the constant as displayed in the discrete trace inequality
				$$\bigg|\intmean_E\partial(v_2-J_{1}v_2)|_K/\partial\nu_E\,\textup{d}s\bigg|\le \sqrt{3/|K|}\,\Vert \nabla(v_2-J_{1}v_2)|_K\cdot\nu_E\Vert_{L^2(K)}.$$
					The combination of the above and the inverse estimate for $HCT^\prime(K)$, i.e., piecewise polynomials of degree at most $3$ 
				\cite[Lemma 4.5.3]{BS08} with constant $c_{\textup{inv},3}$ reveals  
				\begin{align*}
					\Vert v_2-J_{2}v_2\Vert_{L^2(K)}
					&\le \Vert v_2-J_{1}v_2\Vert_{L^2(K)} + 3/\sqrt{770}\, 	h_K	\,\Vert \nabla(v_2-J_{1}v_2)\Vert_{L^2(K)}
					\\
					&\le (1+3c_{\textup{inv},3}/\sqrt{770})\,\Vert v_2- J_{1}v_2\Vert_{L^2(K)}.
				\end{align*}			

		\item\hspace{-.25cm}\emph{.}\hspace{.2cm} 
			 	For each component $\alpha=1,2$ and any $z\in\mathcal{N}(K)$, let $\psi_{z,\alpha}\in {HCT}(\mathcal{T})$ 
				denote the nodal basis function with partial derivative $(\partial\psi_{z,\alpha}/\partial x_\alpha)(z)=1$ in direction $x_\alpha$,
				 which vanishes for the remaining  degrees of freedom. The Hsieh-Clough-Tocher 
				finite element is one in the sense of Ciarlet \cite{Ciarlet78} and so any $v_2|_K\in P_2(K)\subset P_3(\mathcal{K}(K))$ 
				can be represented by the HCT basis functions. 
				The definition of $J_{1}$ reveals that  $(v_2-J_{1}v_2)|_K$ vanishes at the nodes and its normal derivatives vanish at the edge midpoints.
				 Hence this difference belongs to  
				$\textup{span}\{ \psi_{z,\alpha}:\, z\in\mathcal{N}(K), \,\alpha=1,2\}$. Therefore,
					\begin{align}
						\Vert v_2-J_{1}v_2\Vert_{L^2(K)}
							= \bigg\Vert \sum_{z\in\mathcal{N}(K)}\sum_{\alpha=1,2}
											\frac{\partial(v_2-J_{1}v_2)|_K}{\partial x_\alpha}(z)\psi_{z,\alpha}\bigg\Vert_{L^2(K)}.\label{eq:J1P2_vertixvalue}
    				\end{align}
    	\item\hspace{-.25cm}\emph{.}\hspace{.2cm}  
    			The notion of an almost affine family of finite elements in \cite[Thm. 6.1.3, p.344]{Ciarlet78} concerns the scaling of the basis functions  
					\begin{align}
						\Vert h_K^{-2}\psi_{z,\alpha}\Vert_{L^2(K)}\lesssim 1\qquad \text{ for any }K\in\mathcal{T},\,
									 z\in\mathcal{N}(K), \text{ and }\alpha=1,2.\label{eq:missingConstantHCT}
					\end{align}
				The combination with a triangle inequality in \eqref{eq:J1P2_vertixvalue} shows
				$$
				\Vert h_K^{-2}(v_2-J_{1}v_2)\Vert_{L^2(K)}
					\lesssim \sum_{z\in\mathcal{N}(K)}\sum_{\alpha=1,2}\bigg|\frac{\partial(v_2-J_{1}v_2)|_K}{\partial x_\alpha}(z)\bigg|.
				$$
				The non-constructive proof of \cite[Thm. 6.1.3]{Ciarlet78} is based on compactness arguments and leaves the constant in 
				\eqref{eq:missingConstantHCT} unquantified. 
		\item\label{step:CRArgumentsForMorley} \hspace{-.25cm}\emph{.}\hspace{.2cm} 	
				For $v_2\in P_2(\mathcal{T})$ and $\alpha=1,2$ fixed,  the partial derivative $v_1:=\partial v_2/\partial x_\alpha\in P_1(\mathcal{T})$ is piecewise 
				affine.  To lower a conflict of notation, let $J_1^\prime$  denote the companion from \cref{def:LocalCompanionCR_J1} in  
				\cref{sec:PiecewiseCompanion} and let $J_1$ denote the companion from  \cref{def:LocalCompanionMorley_J1} above.  
				The  nodal values of the derivative $ \nabla (J_{1}v_2)|_K$ 
				 in \eqref{eq:defJ1P2_Gradient}  
				 coincide component-wise with 
				 the nodal values of the companion $J_1^\prime v_1$  in \eqref{eq:defJ1P1}
				 applied to  
				 $v_1$,
				$$\frac{\partial}{\partial x_\alpha} (J_{1}v_2)|_K(z)=\Big(J_{1}^\prime \frac{\partial v_2}{\partial x_\alpha}\Big)|_K(z)\qquad\text{ for any }K\in\mathcal{T} \text{ and } 
				 z\in\mathcal{N}(K).$$ 			
				The arguments in \ref{step:jumpInner}--\ref{step:inverseEstimate_P1} of the proof of \cref{thm:P1-J2Fehler=Jumps}  apply simultaneously to 
				the components $v_1=\partial v_2/\partial x_\alpha\in P_1(\mathcal{T})$ for $\alpha=1,2$ and then lead to 
				\begin{align*}
					C_2^{-1} \sum_{\alpha=1,2}\sum_{z\in\mathcal{N}(K)}\bigg|\frac{\partial(v_2-J_{1}v_2)|_K}{\partial x_\alpha}(z)\bigg|^2
						\le  &\sum_{z\in\mathcal{N}(K)}\sum_{E\in\mathcal{E}(K,z)}h_E^{-1} \Vert [\nabla v_2]_E\Vert_{L^2(E)}^2
										\\&\qquad	\qquad	
										+\sum_{E\in\mathcal{E}(K)\cap \mathcal{E}^\prime({\partial\Omega})}h_E^{-1} \Vert \nabla v_2\Vert_{L^2(E)}^2
				\end{align*}
				with $C_2=4\max \{2,{(M_2-1)(2M_2-1)}/({6M_2} )\}$.
		\item \hspace{-.25cm}\emph{.}\hspace{.2cm}  
				The combination of \ref{step:MorleyJ2=J1}--\ref{step:CRArgumentsForMorley} concludes the proof. 
	\end{enumerate}
\end{proof}
	\begin{remark}\label{rem:Morley-TangentialJump}
				The derivative $D_{NC}v_M\in CR^1_0(\mathcal{T};\mathbb{R}^2)$ of a Morley function  $v_{M}\in M(\mathcal{T})$
				is a Crouzeix-Raviart function 
				in each component. 
				Therefore, the combination of \cref{thm:P2-J2Fehler=Jumps} with the Poincar\'e argument in \cref{rem:CR-TangentialJump} implies,
				for any $K\in\mathcal{T}$, that 
				\begin{align}
				h_K^{-4}\Vert v_{M}-J_{2}v_M\Vert_{L^2(K)}^2
				\lesssim &  \sum_{z\in\mathcal{N}(K)}\sum_{E\in\mathcal{E}(K,z)}h_E \Vert [D_{NC}^2v_{M}]_E\times\nu_E\Vert_{L^2(E)}^2
				\notag\\
								&\qquad\qquad+\sum_{E\in\mathcal{E}(K)\cap \mathcal{E}^\prime({\partial\Omega})}h_E \Vert [D_{NC}^2v_{M}]_E\times\nu_E\Vert_{L^2(E)}^2. \label{eq:J2-jump_Morley}
				\end{align}
\end{remark}

\subsection{Refined Analysis for Crouzeix-Raviart and Morley FEM}\label{sec:RefinedAnalysisExample}
Throughout this section, let   $\mathcal{T}\in\mathbb{T}$ be a regular triangulation with set of all sides $\mathcal{F}$ and let $\widehat{\mathcal{T}}\in\mathbb{T}(\mathcal{T})$  be an admissible refinement with set of all sides $\widehat{\mathcal{F}}$. Then define 
\begin{align}
	\mathcal{F}^\prime:=\mathcal{F}\setminus\widehat{\mathcal{F}}\subset\mathcal{F}. \label{eq:FprimeDef}
\end{align}
\cref{fig:refinedSideconnectivity} illustrates the associated sets $\mathcal{T}(K,z)$ from \cref{def:A3} for $K\in \mathcal{T}$, $z\in \mathcal{N}(K)$.  The associated set of sides  $\mathcal{F}(K,z)\subset \mathcal{F}\setminus\widehat{\mathcal{F}}$ contains only coarse-but-not-fine sides. For a coarse and fine $K\in \mathcal{T}\cap\widehat{\mathcal{T}}$ it holds  $\mathcal{T}(K,z)=\{K\}$ as well as $\mathcal{F}(K,z)=\emptyset$ for all $z\in \mathcal{N}(K)$. 
\begin{figure}
\begin{center}
		\begin{tikzpicture}[x=15mm,y=15mm]
			  \draw (2,1) node[shape=coordinate] (C) {C};
			  \draw (2,0) node[shape=coordinate,label={[xshift=1.5mm, yshift=-4.5mm] $B$}] (B) {B};
			  \draw (1,0) node[shape=coordinate,label={[xshift=1.5mm, yshift=-4.5mm] $A$}] (A) {A};
			  \draw (0,1) node[shape=coordinate] (P1) {P1}; 
			  \draw (1,1) node[shape=coordinate] (P2) {P2};
			  \draw (0,0) node[shape=coordinate] (P3) {P3}; 
			  \draw (3,1) node[shape=coordinate] (P4) {P4}; 
			  \draw (3,0) node[shape=coordinate] (P5) {P5};
			  \draw (0,-1) node[shape=coordinate] (P6) {P6};
			  \draw (1,-1) node[shape=coordinate] (P7) {P7};
			  \draw (2,-1) node[shape=coordinate] (P8) {P8};
			  \draw (3,-1) node[shape=coordinate] (P9) {P9};
			  \filldraw[white!80!red] (A)--(P2)--(P3)--cycle;;
			  \filldraw[ne=white!20!black](A)--(B)--(C)--cycle;
			  \filldraw[ne=white!20!black](A)--(C)--(P2)--cycle;
			  \filldraw[nw=white!20!black](A)--(B)--(C)--cycle;
			  \filldraw[nw=white!20!black](B)--(P4)--(C)--cycle;
			  \filldraw[nw=white!20!black](B)--(P5)--(P4)--cycle;
			  \filldraw[nw=white!20!black](B)--(P8)--(P5)--cycle;
			  
			  \draw[line width=0.6pt]  (A)--(B)--(C)--cycle; 
			  \draw[line width=0.6pt] (A)--(P2)--(P3)--cycle; 
			  \draw[line width=0.6pt]  (A)--(P2)--(C)--cycle; 
			  \draw[line width=0.6pt]  (P1)--(P2)--(P3)--cycle;
			  \draw[line width=0.6pt]  (B)--(P4)--(C)--cycle;
			  \draw[line width=0.6pt]  (B)--(P5)--(P4)--cycle;
			  \draw[line width=0.6pt] (A)--(P3)--(P6)--cycle;
			  \draw[line width=0.6pt] (A)--(B)--(P7)--cycle;
			  \draw[line width=0.6pt]  (A)--(P6)--(P7)--cycle;
			  \draw[line width=0.6pt] (B)--(P7)--(P8)--cycle;
			  \draw[line width=0.6pt]  (B)--(P8)--(P5)--cycle;
			  \draw[line width=0.6pt]  (P5)--(P8)--(P9)--cycle;
			  \fill (A) circle (1.5pt);
			  \fill (B) circle (1.5pt);
			  
			  \draw ($(A)!.5!(C)$) node[shape=coordinate] (PAC) {PAC};    
			  \draw ($(C)!.5!(B)$) node[shape=coordinate] (PBC) {PBC};   
			  \draw ($(B)!.5!(P4)$) node[shape=coordinate] (PB4) {PB4};  
			  \draw ($(P2)!.5!(P3)$) node[shape=coordinate] (P23) {P23};
			  \draw ($(A)!.5!(P3)$) node[shape=coordinate] (PA3) {PA3};
			  \draw ($(A)!.5!(P6)$) node[shape=coordinate] (PA6) {PA6};
			  \draw ($(B)!.5!(P4)$) node[shape=coordinate] (PB4) {PB4};
			  \draw ($(P5)!.5!(P8)$) node[shape=coordinate] (P58) {P58};
			  
			  \draw[dashed, line width=0.6pt] (P1)--(P23);
			  \draw[dashed, line width=0.6pt] (A)--(P23);
			  \draw[dashed, line width=0.6pt] (P2)--(PAC);
			  \draw[dashed, line width=0.6pt] (PAC)--(B);
			  \draw[dashed, line width=0.6pt] (C)--(PB4);
			  \draw[dashed, line width=0.6pt] (P5)--(PB4);
			  \draw[dashed, line width=0.6pt] (PAC)--(PBC);
			  \draw[dashed, line width=0.6pt] (PBC)--(PB4);
			  \draw[dashed, line width=0.6pt] (P3)--(P7);
			  \draw[dashed, line width=0.6pt] (A)--(P8);
			  \draw[dashed, line width=0.6pt] (B)--(P9);
			  \draw[dashed, line width=0.6pt] (PB4)--(P58);
			  
			  \node (T1) at ($(PAC)!.5!(B)$) (T1) {$\boldsymbol{T_1}$};
			  \node (T2) at ($(P23)!.5!(A)$) (T2) {$\boldsymbol{T_2}$};
			  \draw (3.5,-2.4) node[shape=coordinate] (P) {P};
		\end{tikzpicture}
		\begin{tikzpicture}[x=10mm,y=20mm]
			  \draw (0,0) node[shape=coordinate] (C) {C};
			  \draw (2,0) node[shape=coordinate,label=below:$B$] (B) {B};
			  \draw (1,1) node[shape=coordinate,label=above:$A$] (A) {A};
			  \draw (3.3,1.3) node[shape=coordinate] (P1) {P1}; 
			  \draw (2.5,2.2) node[shape=coordinate] (P2) {P2}; 
			  \draw (0.2,2) node[shape=coordinate] (P3) {P3}; 
			  \draw (-0.5,1) node[shape=coordinate] (P4) {P4}; 
			  \draw (1.3,-0.8) node[shape=coordinate] (P5) {P5}; 
			  \draw (3.5,-0.6) node[shape=coordinate] (P6) {P6}; 
			  \draw (4.5,0.6) node[shape=coordinate] (P7) {P7};
			  \filldraw[white!80!red] (A)--(P2)--(P3)--cycle;
			  \filldraw[white!80!red] (A)--(P3)--(P4)--cycle;
			   \draw (1,1) node[shape=coordinate,label=above:$A$] (A) {A} ;
			  \filldraw[ne=white!20!black](A)--(B)--(C)--cycle;
			  \filldraw[ne=white!20!black](A)--(B)--(P1)--cycle;
			  \filldraw[ne=white!20!black](A)--(P1)--(P2)--cycle;
			  \filldraw[nw=white!20!black](A)--(B)--(C)--cycle;
			  \filldraw[nw=white!20!black](A)--(B)--(P1)--cycle;
			  \filldraw[nw=white!20!black](B)--(P7)--(P1)--cycle;
			  \filldraw[nw=white!20!black](B)--(P6)--(P7)--cycle;
			  \filldraw[nw=white!20!black](B)--(P5)--(P6)--cycle;
			  \filldraw[nw=white!20!black](B)--(P5)--(C)--cycle;
			  \draw[line width=0.6pt]  (A)--(B)--(C)--cycle; 
			  \draw[line width=0.6pt]  (A)--(B)--(P1)--cycle;
			  \draw[line width=0.6pt]  (A)--(P1)--(P2)--cycle;
			  \draw[line width=0.6pt]  (A)--(P2)--(P3)--cycle; 
			  \draw[line width=0.6pt]  (A)--(P3)--(P4)--cycle;
			  \draw[line width=0.6pt]  (A)--(P4)--(C)--cycle;
			  \draw[line width=0.6pt] (B)--(C)--(P5)--cycle;
			  \draw[line width=0.6pt]  (B)--(P5)--(P6)--cycle;
			  \draw[line width=0.6pt] (B)--(P6)--(P7)--cycle;
			  \draw[line width=0.6pt]  (B)--(P7)--(P1)--cycle;
			  \fill (A) circle (1.5pt);
			  \fill (B) circle (1.5pt);
			  
			  \draw ($(A)!.5!(B)$) node[shape=coordinate] (PAB) {PAB};    
			  \draw ($(C)!.5!(B)$) node[shape=coordinate] (PCB) {PCB};   
			  \draw ($(B)!.5!(P1)$) node[shape=coordinate] (PB1) {PB1};  
			  \draw ($(P2)!.5!(P3)$) node[shape=coordinate] (P23) {P23};
			  \draw ($(A)!.5!(P1)$) node[shape=coordinate] (PA1) {PA1};
			  \draw ($(A)!.5!(P3)$) node[shape=coordinate] (PA3) {PA3};
			  \draw ($(P5)!.5!(P6)$) node[shape=coordinate] (P56) {P56};
			  \draw ($(B)!.5!(P6)$) node[shape=coordinate] (PB6) {PB6};
			  \draw ($(B)!.5!(P7)$) node[shape=coordinate] (PB7) {PB7};
			  \draw ($(P1)!.5!(P7)$) node[shape=coordinate] (P17) {P17};
			  \draw ($(B)!.5!(PB7)$) node[shape=coordinate] (PBB7) {PBB7};
			  
			
			  \draw[dashed, line width=0.6pt] (A)--(PCB);
			  \draw[dashed, line width=0.6pt] (A)--(P23);
			  \draw[dashed, line width=0.6pt] (PA3)--(P23);
			  \draw[dashed, line width=0.6pt] (PA3)--(P4);
			  \draw[dashed, line width=0.6pt] (B)--(PA1);
			  \draw[dashed, line width=0.6pt] (PA1)--(P2);
			  \draw[dashed, line width=0.6pt] (PAB)--(PA1);
			  \draw[dashed, line width=0.6pt] (PB1)--(PA1);
			  \draw[dashed, line width=0.6pt] (PAB)--(PCB);
			  \draw[dashed, line width=0.6pt] (PCB)--(P5);
			  \draw[dashed, line width=0.6pt] (P56)--(B);
			  \draw[dashed, line width=0.6pt] (P56)--(PB6);
			  \draw[dashed, line width=0.6pt] (PB7)--(P6); 
			  \draw[dashed, line width=0.6pt] (PB7)--(P1); 
			  \draw[dashed, line width=0.6pt] (PB7)--(PB6); 
			  \draw[dashed, line width=0.6pt] (PB7)--(P17); 
			  \draw[dashed, line width=0.6pt] (PB7)--(PB1);
			  \draw[dashed, line width=0.6pt] (PB6)--(PBB7);
			  \draw[dashed, line width=0.6pt] (PB1)--(PBB7); 
			  
			  \node (T1) at ($(PCB)!.33!(A)$) (T1) {$\boldsymbol{T_1}$};
			  \node (T2) at ($(1,2)!.33!(A)$) (T2) {$\boldsymbol{T_2}$};
			
			    \draw (5.5,2.2) node[shape=coordinate, label={[xshift=4mm, yshift=-4mm]\mbox{ $\mathcal{T}$}}] (L1) {L1};
			    \draw[line width=0.8pt] ($(L1)-(0,0.075)$)--($(L1)-(0.3,0.075)$);
			    \draw (5.5,1.9) node[shape=coordinate, label={[xshift=4mm, yshift=-4mm]\mbox{ $\widehat{\mathcal{T}}$}}] (L2) {L2};
			    \draw[dashed, line width=0.8pt] ($(L2)-(0,0.075)$)--($(L2)-(0.3,0.075)$);
			    \draw (5.5,1.6) node[shape=coordinate, label={[xshift=9mm, yshift=-5mm]\mbox{ $\mathcal{T}(T_1,A)$}}] (L3) {L3};
			    \filldraw[ne=white!20!black] (L3) rectangle ($(L3)-(0.3,0.15)$);
			    \draw (5.5,1.3) node[shape=coordinate, label={[xshift=9mm, yshift=-5mm]\mbox{ $\mathcal{T}(T_1,B)$}}] (L4) {L4};
			    \filldraw[nw=white!20!black] (L4) rectangle ($(L4)-(0.3,0.15)$);
			     \draw (5.5,1) node[shape=coordinate, label={[xshift=9mm, yshift=-5mm]\mbox{ $\mathcal{T}(T_2,A)$}}] (L5) {L5};
			    \filldraw[white!80!red] (L5) rectangle ($(L5)-(0.3,0.15)$);
			    \draw[black] (L5) rectangle ($(L5)-(0.3,0.15)$);
		\end{tikzpicture}
\end{center}
\caption{Illustration of the sets $\mathcal{T}(K,z)$ given $\mathcal{F}^\prime=\mathcal{F}\setminus\widehat{\mathcal{F}}$ in \eqref{eq:FprimeDef}.}\label{fig:refinedSideconnectivity}
\end{figure}
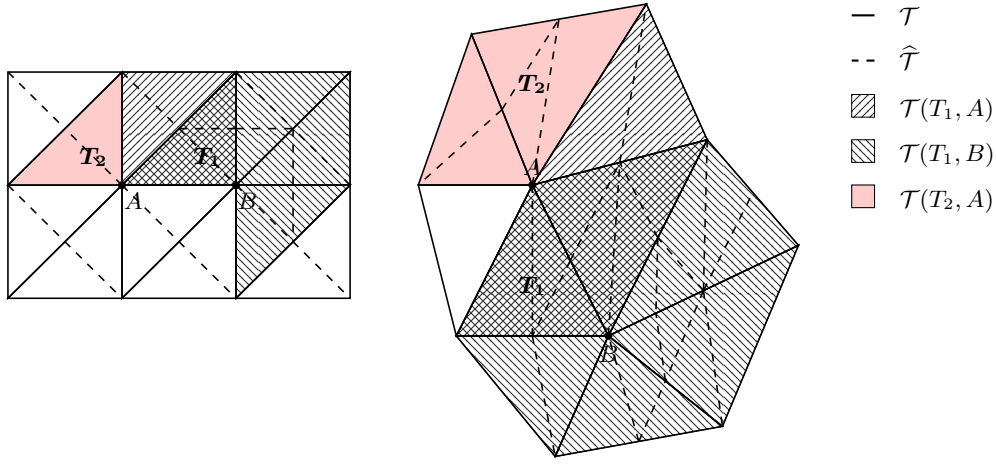
	
The choice of $\mathcal{F}^\prime$ in \eqref{eq:FprimeDef} allows the definition of an approximation $\widehat{u}_{CR}^*$ in \eqref{eq:uCR*local}  (resp. $\widehat{u}_{M}^*$ in \eqref{eq:uM*local} below) to the discrete function ${u}_{CR}\in CR^1_0({\mathcal{T}})$ (resp. ${u}_{M}\in M({\mathcal{T}})$).  
Recall  that $\widehat{I}_{\textup{NC}}$ from \cref{sec:CR_Interpolant} denotes the nonconforming interpolation operator with respect to $CR(\widehat{\mathcal{T}})$. 
\begin{lemma}[$u_{CR}^*$]\label{lem:C4newCR}
	Given any $u_{CR}\in CR^1_0(\mathcal{T})$ and  \eqref{eq:FprimeDef} in \cref{def:LocalCompanionCR}, 
	\begin{align}
		 \widehat{u}_{CR}^*:=\widehat{I}_\textup{NC}\big(J_n u_{CR}\big)\in CR^1_0(\widehat{\mathcal{T}}) \label{eq:uCR*local} 
	\end{align}
	is well-defined and satisfies \eqref{eq:C4_uhstar}.   
\end{lemma}
\begin{proof}
\begin{enumerate}[label=\emph{Step \arabic*}, wide]
	\item\hspace{-.25cm}\emph{.}\hspace{.2cm}
						\cref{rem:jumps_J1P1andJnCR}.b and  \eqref{eq:FprimeDef} 
						guarantee that $J_n(u_{CR})$ is continuous in the midpoint of any side  $F\in\widehat{\mathcal{F}}$ and  vanishes at  the 
						midpoint of boundary 
						sides  $F\in\widehat{\mathcal{F}}(\partial\Omega)$. 
						Hence 
						the nonconforming interpolation  		
						$\widehat{u}_{CR}^*=\widehat{I}_\textup{NC}\big(J_n u_{CR}\big)\in CR^1_0(\widehat{\mathcal{T}})$ is well defined and admits  homogeneous boundary conditions.
							
	\item\hspace{-.25cm}\emph{.}\hspace{.2cm} 	
						The correction with the side-bubble functions in \cref{def:LocalCompanionCR} leads to the identity  
						\begin{align}
	 						\intmean_F\widehat{u}_{CR}^*\,\textup{d}s=\intmean_F u_{CR}\,\textup{d}s\quad\text{for all  }F\in \mathcal{F}.\label{eq:vCR_vCR*}
						\end{align} 
						The integral means are traces on the neighbouring simplices $T_{\pm}$ on $F$ and those values are independent of $T_+$ or $T_-$ for an interior side. 
						This and the definition of $I_{NC}$ imply \eqref{eq:C4_uhstar}.
\end{enumerate}
\end{proof}
Recall that $\widehat{I}_M$  from \cref{sec:Morley_HCT} denotes the interpolation operator  with respect to $\mathcal{M}(\widehat{\mathcal{T}})$.
\begin{lemma}[$u_M^*$]\label{lem:C4newMorley}
	Given any $u_{M}\in M(\mathcal{T})$ and 	
	\eqref{eq:FprimeDef} in \cref{def:LocalCompanionMorley}, 
	\begin{align}
		 \widehat{u}_{M}^*:=\widehat{I}_\textup{M}\big(J_2u_{M}\big)\in M(\widehat{\mathcal{T}}) \label{eq:uM*local} 
	\end{align}
	is well-defined and satisfies \eqref{eq:C4_uhstar}. 
\end{lemma}
\begin{proof}
\begin{enumerate}[label=\emph{Step \arabic*}, wide]
	\item\hspace{-.25cm}\emph{.}\hspace{.2cm} 
					\cref{rem:jumps_J1MandJ2M}.b and  \eqref{eq:FprimeDef} guarantee the continuity of
					 $\nabla J_2(u_{M})$ at the midpoints of  $E\in\widehat{\mathcal{E}}$ and  of $J_2(u_{M})$ at the new vertices $\widehat{\mathcal{N}}$ 
					 (either $z\in\mathcal{N}\cap\widehat{\mathcal{N}}$ or $z\in E\in\mathcal{E}^\prime$). In particular, $J_2(u_M)$ vanishes at all vertices in $\widehat{\mathcal{N}}(\partial\Omega)$
					  and $\nabla J_2(u_{M})$ vanishes at the midpoints of all edges in $\widehat{\mathcal{E}}(\partial\Omega)$. 
					 Hence the Morley interpolation  $\widehat{u}_{M}^*=\widehat{I}_\textup{M}\big(J_2u_{M}\big)\in M^\prime(\widehat{\mathcal{T}})$ is well defined 
					 and admits homogeneous boundary conditions. 

	\item\hspace{-.25cm}\emph{.}\hspace{.2cm} 	
						The definition of $\widehat{I}_M$  and \cref{def:LocalCompanionMorley} show that the nodal values  $\widehat{u}_M^*(z)=u_M(z)$  
						coincide for all 
						 $z\in\mathcal{N}$  and the correction with the edge-bubble functions guarantees 
						\begin{align}
							\intmean_E \frac{\partial\widehat{u}_{M}^*}{\partial\nu_E}\,\textup{d}s=\intmean_E \frac{\partial u_{M}}{\partial\nu_E}\,\textup{d}s\quad\text{for all  }E\in \mathcal{E}.
									\label{eq:uM_uM*}
						\end{align}
					The integral means are traces on the neighbouring simplices $T_{\pm}$ on $E$ and those values are independent of $T_+$ or $T_-$ for an interior edge.
					This and the definition of $I_{M}$ imply \eqref{eq:C4_uhstar}.
\end{enumerate}
\end{proof}


The estimate \eqref{eq:Improve_uhstar-uh} follows by collecting the above results in \cref{thm:RefinedCR} resp. \ref{thm:RefinedMorley} below. 
\begin{theorem}\label{thm:RefinedCR}
	Given 
	$u_{CR}\in CR^1_0(\mathcal{T})$ and its approximation $\widehat{u}_{CR}^*\in CR^1_0(\widehat{\mathcal{T}})$ in \eqref{eq:uCR*local}, 
	\begin{align*}
		\Vert \widehat{u}_{CR}^*-u_{CR}\Vert_h^2\le c_{\textup{inv},n}^2{C}^{\prime}_nM_2 \sum_{F\in\mathcal{F}\setminus\widehat{\mathcal{F}}}h_F\Vert [\nabla_{\textup{NC}}u_{CR}]_F\times\nu_F\Vert_{L^2(F)}^2
	\end{align*}	  
	holds with $c_{\textup{inv},n}$ from the inverse estimate for piecewise polynomials up to degree $n$, ${C}^{\prime}_n$ from \cref{rem:CR-TangentialJump},  
	and $M_2$ from assumption \ref{triang:A1}. 
\end{theorem}
\begin{proof}
	Conditions \eqref{eq:C3_vanish}--\eqref{eq:C5_stab} for $\widehat{I}_{\textup{NC}}$ resp. $\widehat{u}_{CR}^*$ lead to \eqref{eq:UseOfC5} and an inverse estimate with constant $c_{\textup{inv},n}$ for  polynomial 
	functions of degree at most $n$ leads to 
	\begin{align*}
		\Vert \widehat{u}_{CR}^*-u_{CR}\Vert_h 
					&=\Vert D_\textup{NC}(\widehat{u}_{CR}^*-u_{CR})\Vert_{L^2(\mathcal{T}\setminus\widehat{\mathcal{T}})}
					\le c_{\textup{inv},n}\Vert h_\mathcal{T}^{-1}(J_{n}u_{CR}-u_{CR})\Vert_{L^2(\mathcal{T}\setminus\widehat{\mathcal{T}})}.
    \end{align*}
	\cref{thm:P1-J2Fehler=Jumps} and  \cref{rem:CR-TangentialJump}  
	conclude the proof  of \eqref{eq:Improve_uhstar-uh} for $\mathcal{F}(K,z)\subset\mathcal{F}\setminus\widehat{\mathcal{F}}$ for any $K\in\mathcal{T}$ and $z\in\mathcal{N}(K)$.
\end{proof}
\begin{corollary}\label{cor:CR_dREL_T}
	The discrete reliability \eqref{eq:dRel} holds for the Crouzeix-Raviart FEM with $\mathcal{T}\setminus\widehat{\mathcal{T}}$ replacing $\mathcal{R}$
	and the constant $\Lambda_{drel}:= (1+1/\sqrt{2})\max\big\{\sqrt{19/48},c_{\textup{inv},n}\sqrt{{C}^{\prime}_{n}M_2}\big\}$.
\end{corollary}	
\begin{proof}
	\cref{thm:RefinedCR} shows in particular \eqref{eq:uhstar-uhEstimate} with $\mathcal{T}\setminus\widehat{\mathcal{T}}$ replacing $\mathcal{R}$ and 
	constant $\Lambda_2^2=c_{\textup{inv},n}^2{C}^{\prime}_{n}M_2$.
    \cref{sec:CrouzeixRaviart} proves \eqref{eq:C1_h}--\eqref{eq:C3_vanish} for $I_{\textup{NC}}$ and	
	\cref{lem:C4newCR} proves \eqref{eq:C4_uhstar}. Hence   \cref{thm:ProofOferror<uhstarError} implies \eqref{eq:error<uhstarError}. The combination of 
	\eqref{eq:error<uhstarError}--\eqref{eq:uhstar-uhEstimate} concludes the proof of \eqref{eq:dRel}. 
\end{proof}
\begin{corollary}[Constants in $2$D]
	The constant $\Lambda_{drel}$ in \eqref{eq:dRel} is  bounded in terms of the minimal angle 
	$\omega_0$ and $M_2=\sup_{z\in\mathcal{N},\mathcal{T}\in\mathbb{T}}|\mathcal{T}(z)|\le 2\pi/\omega_0$ by
	 	\begin{align*}
			\Lambda_{drel}&:=\frac{1+1/\sqrt{2}}{12}\max\big\{\sqrt{{57}}, 
							C_J\sqrt{\max\{12M_2,{(M_2-1)(2M_2-1)}\}} \big\},\\
 	 		C_J^2&:=\frac{97}{4}\,\cot(\omega_0)	
 	 				\big(2\cot(\omega_0)-\cot(2\omega_0)\big)+24\cot(\omega_0)\sqrt{(2\cot(\omega_0)-\cot(2\omega_0))^2-3}.
	\end{align*}	
\end{corollary}	
\begin{proof}
	Given $\Lambda_1=\sqrt{19/48}$ from \cref{sec:CR_C1} in \eqref{eq:error<uhstarError}, it remains to compute the constant $\Lambda_2$ in \eqref{eq:uhstar-uhEstimate}. 
	\cref{cor:CR_dREL_T} proves $\Lambda_2= c_{inv,2}\sqrt{{C}^\prime_{2}M_2}$. The following calculation circumvent the computation of the 
	constant $c_{inv,2}$ in the inverse estimate for piecewise quadratics. 
	From the conditions \eqref{eq:C3_vanish}--\eqref{eq:C5_stab} for $\widehat{I}_{\textup{NC}}$ resp. $\widehat{u}_{CR}^*$ follows \eqref{eq:UseOfC5}. 
    For each $K\in\mathcal{T}\setminus\widehat{\mathcal{T}}$ the triangle inequality leads to
    \begin{align*}
    	\Vert \nabla(u_{CR}-J_2 u_{CR})\Vert_{L^2(K)}
				&\le  \Vert \nabla(u_{CR}- J_{1}u_{CR})\Vert_{L^2(K)}
				\\&\qquad\qquad
				+\sum_{F\in\mathcal{F}(K)}\bigg|\intmean_F(u_{CR}-J_{1}u_{CR})|_K\,\textup{d}s\bigg| \Vert \nabla b_F\Vert_{L^2(K)}.
    \end{align*}
	If $\alpha,\,\beta,\,\gamma$ denote the interior angles in $K$,
	$\Vert \nabla b_F\Vert_{L^2(K)}=\sqrt{\cot\alpha+\cot\beta+\cot\gamma}/(12\sqrt{3})$.
	A maximisation shows $\Vert \nabla b_F\Vert_{L^2(K)} \le \sqrt{2\cot(\omega_0)-\cot(2\omega_0)}/(12\sqrt{3})$.   
	The combination with Lemma \ref{lem:AppendixIntmeanP1} in the appendix for any $F\in\mathcal{F}(K)$ and $h_K^{2}\le  4|K|\cot(\omega_0)$ implies
	\begin{align*}
		\sum_{F\in\mathcal{F}(K)}\bigg|\intmean_F&(u_{CR}-J_{1}u_{CR})|_K\,\textup{d}s\bigg| \Vert \nabla b_F\Vert_{L^2(K)}\\\vspace{-2mm}
		&\le   \frac{\sqrt{\cot(\omega_0)(2\cot(\omega_0)-\cot(2\omega_0))}}{2h_K}\Vert u_{CR}-J_{1}u_{CR}\Vert_{L^2(K)}.
	\end{align*}
	On the other hand,  \cite[Lem. 4.10]{CH17} establishes the constant 
		$c_{\textup{inv},1}^2=24\cot(\omega_0)\big(2\cot(\omega_0)-\cot(2\omega_0)+\sqrt{(\cot(\omega_0)-\cot(2\omega_0))^2-3}\,\big)$
	in the inverse estimate for affine functions.
	Therefore,   $\Vert \nabla(u_{CR}-J_2 u_{CR})\Vert_{L^2(K)}\le C_{J} h_K^{-1}\, \Vert u_{CR}-J_{1}u_{CR}\Vert_{L^2(K)}$ holds with the constant $C_J$. 
	The combination of  \ref{step:eKDef}--\ref{step:jumpBoundary} in the proof of \cref{thm:P1-J2Fehler=Jumps} and \cref{rem:CR-TangentialJump} shows that $C_2:={\max\{2,{(M_2-1)(2M_2-1)}/{(6M_2)}\}}C_J^2/24$ satisfies
	\begin{align*}								 
		C_2^{-1}	\Vert \nabla(u_{CR}-J_2 u_{CR})\Vert_{L^2(K)}^2		
			\le &	 \sum_{F\in\mathcal{F}(K)\cap \mathcal{F}^\prime({\partial\Omega})}h_F \Vert [\nabla_{NC}u_{CR}]_F\times\nu_F\Vert_{L^2(F)}^2
															\\&\qquad+\sum_{z\in\mathcal{N}(K)}\sum_{F\in\mathcal{F}(K,z)}h_F
															 \Vert [\nabla_{NC}u_{CR}]_F\times\nu_F\Vert_{L^2(F)}^2.
	\end{align*} 
	The sum over all $K\in\mathcal{T}\setminus\widehat{\mathcal{T}}$ and 
	an overlap argument for $\mathcal{F}(K,z)\subset\mathcal{F}\setminus\widehat{\mathcal{F}}$ conclude the proof of \eqref{eq:uhstar-uhEstimate} with  $\lambda_2^2:= M_2C_2$. The combination of \eqref{eq:error<uhstarError}--\eqref{eq:uhstar-uhEstimate} proves \eqref{eq:dRel}.
\end{proof}
\begin{example}
Given a triangulation with a minimal angle $\omega_0=45^\circ$ and $M_2\le 8$, for instance, in a triangulation in right isosceles triangles,  $C_J=\sqrt{145/2}\le  8.5147$ and  $\Lambda_{drel}=({1+1/\sqrt{2}})\,\sqrt{{5075}/{96}}\le 12.4121$ follows,  a significant improvement over \cite[Ex. 6.3]{CH17}.
\end{example}
\begin{theorem}\label{thm:RefinedMorley}
	Given  \ref{triang:A1},
	$u_{M}\in M(\mathcal{T})$, and its approximation $\widehat{u}_{M}^*\in M(\widehat{\mathcal{T}})$ in \eqref{eq:uM*local},  
	\begin{align*}
		\Vert \widehat{u}_{M}^*-u_{M}\Vert_h^2\lesssim \sum_{E\in\mathcal{E}\setminus\widehat{\mathcal{E}}}h_E\Vert [D_{\textup{NC}}^2u_{M}]_E\times\nu_E\Vert_{L^2(E)}^2.
	\end{align*}	  
\end{theorem}
\begin{proof}
	Conditions \eqref{eq:C3_vanish}--\eqref{eq:C5_stab} for $\widehat{I}_{M}$ resp. $\widehat{u}_{M}^*$ lead to \eqref{eq:UseOfC5} and an inverse estimate for piecewise  polynomial functions shows  
	\begin{align*}
		\Vert \widehat{u}_{M}^*-u_{M}\Vert_h 
					&=\Vert D^2_\textup{NC}(\widehat{u}_{M}^*-u_{M})\Vert_{L^2(\mathcal{T}\setminus\widehat{\mathcal{T}})}
					\lesssim \Vert h_\mathcal{T}^{-2}(J_{2}u_{M}-u_{M})\Vert_{L^2(\mathcal{T}\setminus\widehat{\mathcal{T}})}.
    \end{align*}
	\cref{thm:P2-J2Fehler=Jumps} and \cref{rem:Morley-TangentialJump}  
	conclude the proof of \eqref{eq:Improve_uhstar-uh} for $\mathcal{E}(K,z)\subset\mathcal{E}\setminus\widehat{\mathcal{E}}$ for any $K\in\mathcal{T}$ and $z\in\mathcal{N}(K)$.
\end{proof}
\begin{corollary}
	The discrete reliability \eqref{eq:dRel} holds for the Morley FEM with  $\mathcal{T}\setminus\widehat{\mathcal{T}}$ replacing $\mathcal{R}$.
\end{corollary}
\begin{proof}
		\cref{thm:RefinedMorley} shows in particular \eqref{eq:uhstar-uhEstimate} with $\mathcal{T}\setminus\widehat{\mathcal{T}}$ replacing $\mathcal{R}$.
		 \cref{sec:Morley} proves \eqref{eq:C1_h}--\eqref{eq:C3_vanish} for $I_M$ and \cref{lem:C4newMorley} proves \eqref{eq:C4_uhstar}. Hence 
		 \cref{thm:ProofOferror<uhstarError} implies \eqref{eq:error<uhstarError}. 
		The combination of \eqref{eq:error<uhstarError}--\eqref{eq:uhstar-uhEstimate} concludes the proof of \eqref{eq:dRel}. 
\end{proof}		

\FloatBarrier

\section*{Acknowledgements}
This work has been supported by the Deutsche Forschungsgemeinschaft (DFG) in the Priority Program 1748 ‘Reliable simulation techniques in solid mechanics. Development of non-standard discretization methods, mechanical and mathematical analysis’ under the project CA 151/22. The second author is supported by the Berlin Mathematical School.

\bibliographystyle{alpha}
\bibliography{BibDiscReliabilityNC}

\section*{Appendix}
The appendix presents some optimal constants in fundamental inequalities. 
\subsection*{Appendix A}
The subsequent inverse estimate displays an optimal constant $(k+1)/\sqrt{b-a}$.
\begin{lemmaAppendix}\label{lem:1DinverseEstimate}
	Any polynomial $f$ of degree at most $k\in\mathbb{N}$ in a non-void bounded open interval $(a,b)$ satisfies 
		\begin{align*}
			|f(a)|\le \, \frac{k+1}{\sqrt{b-a}}\Vert f\Vert_{L^2(a,b)}.
		\end{align*}
	For any constant $C< ({k+1})/{\sqrt{b-a}}$, there exists some 
	polynomial $f$ of degree at most $k$ with $C\Vert f\Vert_{L^2(a,b)}<|f(a)|$.
\end{lemmaAppendix}
\begin{proof}
	An affine transformation of the interval $(a,b)$ onto $(-1,+1)$ shows that, without loss of generality, one may consider the particular case $a=-1$ and $b=1$. 
	The Legendre polynomials $p_m\in P_k[-1,+1]$ are defined in many ways. 
	For instance via the initialization $p_0=1$ and $p_1(x)=x$ followed by the recursion formula $$(m+1)p_{m+1}(x)=(2m+1)xp_m(x)-mp_{m-1}(x)\quad\text{ for }m=1,2,3,\dots$$ 
	Then the Legendre polynomials $p_m$ are pairwise orthogonal with $$\int_{-1}^{+1}p_m(x)p_n(x)\,\textup{d}x=\frac{2\delta_{mn}}{2m+1}$$ and normalized by 	
	$p_m(-1)=(-1)^mp_m(1)=(-1)^m$ for $m,n\in\mathbb{N}_0$. 
	The polynomial $f(x)=\sum_{j=0}^k a_jp_j$ for some coefficients $a_0,a_1,\dots,a_k\in\mathbb{R}$ satisfies 
	$$ \Vert f\Vert_{L^2(-1,1)}^2={\sum_{j=0}^k \frac{2a_j^2}{2j+1}}\qquad\text{ and }\qquad f(-1)=\sum_{j=0}^k (-1)^ja_j.$$ 
	The latter value is the scalar product in $\mathbb{R}^{k+1}$  of the vectors 
	$\big((-1)^ja_j\sqrt{2/(2j+1)}:\, j=0,1,\dots,k\big)$ and $\big(\sqrt{(2j+1)/2}:\, j=0,1,\dots,k\big)$. 
	The Cauchy inequality in $\mathbb{R}^{k+1}$ shows 
	$$|f(-1)|\le \Vert f\Vert_{L^2(-1,1)}\sqrt{\sum_{j=0}^k(j+1/2)}= \frac{k+1}{\sqrt{2}}\Vert f\Vert_{L^2(-1,1)}. $$
	Notice that the Cauchy inequality is an equality for certain coefficients and so the assertion is sharp in the sense stated in the second half of the lemma. 
\end{proof}
\subsection*{Appendix B}
This section utilizes some language of graph theory and concerns 
an undirected graph $G$ as a pair $(\{1,\dots,n\},\mathcal{E})$ of a set of vertices $\{1,\dots,n\}$ (fixed with $n$ in this section and so neglected in the notation) and a set $\mathcal{E}$ of edges $\{j,k\}$ with $j,k\in\{1,\dots,n\}$ and $j\not = k$. 
The graph $G$ (identified with $\mathcal{E}$) is connected if for all $j,k\in\{1,\dots,n\}$ and $j\not = k$ there are $m\in\{1,2,3,\dots\}$ edges $\{\alpha_1,\alpha_2\},\,\{\alpha_2,\alpha_3\},\,\dots,\,\{\alpha_{m},\alpha_{m+1}\}\in\mathcal{E}$ with $\alpha_1=j$ and $\alpha_{m+1}=k$.
The set of all connected graphs $G$ over the set $\{1,\dots,n\}$ is identified with the set $\mathcal{C}(n)$ of all sets of edges $\mathcal{E}$; so $(\{1,\dots,n\},\mathcal{E})$ is connected is abbreviated as  $\mathcal{E}\in\mathcal{C}(n)$. 

Given $x_1,\dots,x_n\in\mathbb{R}$ the goal is to minimize 
\begin{align*}
f(\mathcal{E}):=\sum_{\{j,k\}\in\mathcal{E}}(x_j-x_k)^2 \text{ over all }\mathcal{E}\in\mathcal{C}(n).
\end{align*}
Since permutations $\sigma$ of the set $\{1,\dots,n\}$ transform $\mathcal{E}\in\mathcal{C}(n)$ into 
$\sigma(\mathcal{E}):=\big\{\{\sigma(j),\sigma(k)\}:\, \{j,k\}\in\mathcal{E}\big\}\in\mathcal{C}(n)$,
without loss of generality, we may and will assume that the enumeration orders the real values $x_1\le x_2\le x_3\le\dots\le x_n$. 
			
\begin{lemmaAppendix}\label{lem:minimalSum}
	Any vector $x=(x_1,\dots,x_n)\in\mathbb{R}^n$  with $n\in\mathbb{N}$ and $x_1\le x_2\le x_3\le\dots\le x_n$  
	and the function $f(\mathcal{E}):=\sum_{\{j,k\}\in\mathcal{E}}(x_j-x_k)^2$  satisfy
	\begin{align*}
		m(x):=\min_{\mathcal{E}\in\mathcal{C}(n)}f(\mathcal{E})=\sum_{j=1}^{n-1}(x_{j+1}-x_j)^2.
	\end{align*}
\end{lemmaAppendix}
\begin{proof}
	\begin{enumerate}[wide, label=(\arabic*)]
	\item 	For any $n\in\mathbb{N}$ and $x\in\mathbb{R}^n$ the minimum $m(x):=\min_{\mathcal{E}\in\mathcal{C}(n)}f(\mathcal{E})$ (where below $f$ applies to any vector of any length) 
			is attained for some $\mathcal{E}\in\mathcal{C}(n)$ with $|\mathcal{E}|=n-1$.
	
			It is known in graph theory that loops can be avoided by certain cuts and any cut means a reduction of the target functional. 
			Therefore, we may and will assume without loss of generality,  
			that $\mathcal{E}\in\mathcal{C}(n)$ is a tree. There is only a finite number of trees for a fixed number of vertices $n$ 
			and so the minimum $m(x)$ is attained for at least one of them. 
	\item 	Given any $x_1\le x_2\le x_3\le\dots\le x_n$ it holds  
			\begin{align*}
				m(x):=\min_{\mathcal{E}\in\mathcal{C}(n)}\sum_{\{j,k\}\in\mathcal{E}}(x_j-x_k)^2=\sum_{j=1}^{n-1}(x_{j+1}-x_j)^2.
			\end{align*}
			The proof is by mathematical induction. 
			The assertion holds for $n=1$ (pathological) and $n=2$ (trivial), so suppose it holds for some $n\ge 2$ and all $x_1\le x_2\le x_3\le\dots\le x_n$. 
			Given $\widehat{x}=(x_1,\dots,x_{n+1})\in\mathbb{R}^{n+1}$ with $x_1\le x_2\le \dots\le x_n\le x_{n+1}$, let $\widehat{\mathcal{E}}\in\mathcal{C}(n+1)$ be a minimizer with 
			$f(\widehat{\mathcal{E}})=m(\widehat{x})$.
		\begin{enumerate}[wide, label=\emph{Step \arabic*}]
				\item\hspace{-.25cm}\emph{.}\hspace{.2cm} 	
						Since $\widehat{\mathcal{E}}\in\mathcal{C}(n+1)$ is connected, 
						there is a path $\{\alpha_1,\alpha_2\},\,\{\alpha_2,\alpha_3\},\,\dots,\,\{\alpha_{m},\alpha_{m+1}\}\in\widehat{\mathcal{E}}$ of length 		
						$m\in\mathbb{N}$ with $\alpha_1=n+1$ and $\alpha_{m+1}=n$. The numbers  $\alpha_1,\dots,\alpha_{m+1}$ can be chosen pairwise distinct (as loops may be excluded). 
						Then $\mathcal{E}^\prime:=\widehat{\mathcal{E}}\setminus\{\{n+1,\alpha_2\}\}$ and $\widetilde{\mathcal{E}}:=\mathcal{E}^\prime\cup\{\{n,n+1\}\}$
						lead to $\widetilde{\mathcal{E}}\in\mathcal{C}(n+1)$ 
						and $f(\widehat{\mathcal{E}})-f(\widetilde{\mathcal{E}})=(x_{n+1}-x_{\alpha_2})^2-(x_n-x_{n+1})^2\ge 0$ (since $x_{\alpha_2}\le x_n\le x_{n+1}$). 
						Consequently, there exists a minimizer $\widehat{\mathcal{E}}\in\mathcal{C}(n+1)$ with $\{n,n+1\}\in\widehat{\mathcal{E}}$. 
				\item\label{step:(k,n+1)zu(k,n)} \hspace{-.25cm}\emph{.}\hspace{.2cm}	
						For any $k\in\{1,2,\dots,n-1\}$ with $\{k,n+1\}\in \widehat{\mathcal{E}}$ consider $\mathcal{E}^\prime:=\widehat{\mathcal{E}}\setminus\{\{k,n+1\}\}$ and 
						$\widetilde{\mathcal{E}}:=\mathcal{E}^\prime\cup\{\{k,n\}\}$. 
						Then  $\widetilde{\mathcal{E}}\in\mathcal{C}(n+1)$ is connected (for $\{n,n+1\}\in\widehat{\mathcal{E}}\in\mathcal{C}(n+1)$).
						Moreover,  
						$f(\widehat{\mathcal{E}})-f(\widetilde{\mathcal{E}})=(x_{n+1}-x_{k})^2-(x_{n}-x_{k})^2\ge 0$ (since $x_k\le x_n\le x_{n+1}$).
				\item\label{step:only(n,n+1)} \hspace{-.25cm}\emph{.}\hspace{.2cm}	 	
						A finite number of changes as in \ref{step:(k,n+1)zu(k,n)} leads to a minimizer $\widehat{\mathcal{E}}\in \mathcal{C}(n+1)$  of $f$ 
						with $\{n,n+1\}\in\widehat{\mathcal{E}}$ 
						and $\{k,n+1\}\not\in\widehat{\mathcal{E}}$ for all $k=1,\dots,n-1$. 
				\item \hspace{-.25cm}\emph{.}\hspace{.2cm} 
						Given a minimizer $\widehat{\mathcal{E}}\in\mathcal{C}(n+1)$ 
						 from \ref{step:only(n,n+1)}, 
						the set $\mathcal{E}^\prime:=\widehat{\mathcal{E}}\setminus\{n+1,n\}\in\mathcal{C}(n)$ is connected and the induction hypothesis guarantees for $x:=(x_1,\dots,x_n)$ 
						that $m(x)=\sum_{j=1}^{n-1}(x_{j+1}-x_{j})^2\le f(\mathcal{E}^\prime)$ (where the abbreviation $f$ applies to $\mathcal{E}^\prime$ as well). 
						Consequently,
						$$m(\widehat{x})=f(\widehat{\mathcal{E}})=f(\mathcal{E}^\prime)+(x_{n+1}-x_n)^2\ge \sum_{j=1}^n (x_{j+1}-x_j)^2.$$  
						Since $\big\{\{1,2\},\{2,3\},\dots ,\{n,n+1\}\big\}\in\mathcal{C}(n+1)$ is in the competition with 
						$$f(\{\{1,2\},\{2,3\},\dots ,\{n,n+1\}\})=\sum_{j=1}^n (x_{j+1}-x_j)^2\le m(\widehat{x})=f(\widehat{\mathcal{E}})$$  and $\widehat{\mathcal{E}}$ is a minimizer,  
						the claim $m(\widehat{x})=\sum_{j=1}^n (x_{j+1}-x_j)^2$ follows. 
		\end{enumerate}
	\end{enumerate}
	\end{proof}
\subsection*{Appendix C}
The subsequent estimate holds with the optimal constant ${(n-1)(2n-1)}/{(6n)}$.
\begin{lemmaAppendix}\label{lem:2DConstantSide}
	Any $x\in \mathbb{R}^n$, $n\in\mathbb{N}$, with vanishing sum $x\cdot(1,\dots,1)=\sum_{j=1}^{n}x_{j}=0$ 
	satisfies
		\begin{align*}
			\max_{j=1}^n|x_j|^2\le \frac{(n-1)(2n-1)}{6n}\ {\sum_{j=1}^{n-1}(x_{j+1}-x_j)^2}. 
		\end{align*}
	For any constant $C<{(n-1)(2n-1)}/{(6n)}$, there exists some $x\in\mathbb{R}^n$ with $x\cdot(1,\dots,1)=0$ and 
		$C \,{\sum_{j=1}^{n-1}(x_{j+1}-x_j)^2}<\max_{j=1}^n|x_j|^2.$
\end{lemmaAppendix}
\begin{proof}
	The assertion holds for $n=1$ (pathological) and $n=2$ (trivial). 
	 A scaling argument for $n\ge 3$ proves that the multiplicative constant in the asserted inequality is the reciprocal of  
		\begin{align*}
			\mu(n):=\min_{x\in \mathcal{A}(n)} f(x) \text{ for } f(x):=\sum_{j=1}^{n-1}(x_{j+1}-x_j)^2
		\end{align*}
	and $\mathcal{A}(n):=\{x\in\mathbb{R}^n:\, x\perp(1,\dots, 1)\text{ and }\Vert x\Vert_\infty=1\}$.
	The arguments of Lemma \ref{lem:minimalSum} (with a change of all signs if necessary) lead to the identity 
		\begin{align*}
			\mu(n)=\min_{x\in\mathcal{B}(n)}f(x)\text{ for }\mathcal{B}(n)=\{x\in\mathbb{R}^n:\, x\perp(1,\dots,1)\text{ and }-1\le x_1\le x_2\le\dots\le x_n=1\}.
		\end{align*}
	Any $x\in\mathcal{B}(n)$ is transformed into $y=(y_1,\dots,y_{n-1})\in \mathbb{R}^{n-1}$ (recall $n\ge 3$) via 
		\begin{align}
			y_j:=x_{j+1}-x_j\text{ for all }j=1,\dots,n-1,\text{ so that } f(x)=|y|^2:=\sum_{j=1}^{n-1}y_j^2 \text{ and }y\ge 0 \label{eq:defYProof}\tag{C1}
		\end{align}
	(with $y\ge 0$ understood componentwise as $y_j\ge 0$ for $j=1,\dots,n-1$). Since 
		\begin{align}
			x_k=1-\sum_{j=k}^{n-1}y_j\text{ for  all }k=1,\dots,n \label{eq:XitoYProof} \tag{C2}
		\end{align}
	(the empty sum is zero), the condition $x\perp(1,\dots,1)$ is equivalent to $n=(1,2,\dots,n-1)\cdot y$ with the scalar product $\cdot$ in $\mathbb{R}^{n-1}$.
	The restriction $-1\le x_1\le x_2\le \dots\le x_n=1$ is equivalent to $y\ge 0$ and 	$(1,\dots,1)\cdot y\le 2$. 
	
	In conclusion,  for $x\in\mathbb{R}^{n}$ and $y\in\mathbb{R}^{n-1}$ with  \eqref{eq:defYProof}, $x\in\mathcal{B}(n)$ is equivalent to 
		\begin{align*}
			y\in\mathcal{C}(n):=\big\{0\le y\in\mathbb{R}^{n-1}:\,(1,2,\dots,n-1)\cdot y=n\text{ and }(1,\dots,1)\cdot y\le 2\big\}.
		\end{align*}
	To determine $\min_{x\in\mathcal{B}(n)}f(x)=\min_{y\in\mathcal{C}(n)}|y|^2$, suppose that $y\in\mathcal{C}(n)$ 
	and utilize a Cauchy inequality for $$n=(1,2,\dots,n-1)\cdot y\le |y|\sqrt{\sum_{j=1}^{n-1}j^2}=|y|\sqrt{\frac{(n-1)n(2n-1)}{6}}.$$ 
	Consequently, $\frac{6n}{(n-1)(2n-1)}\le |y|^2$. 
	Since $y\in\mathcal{C}(n)$ is arbitrary, this proves one inequality in the claim
		\begin{align*}
			\mu(n)= \frac{6n}{(n-1)(2n-1)}\text{ for }n\in\mathbb{N}.\tag{C3}\label{eq:mu(n)def}
		\end{align*}
	To prove the reverse inequality, let $\lambda:=n\big(\sum_{j=1}^{n-1}j^2\big)^{-1}=\frac{6}{(n-1)(2n-1)}>0$ and $y=\lambda(1,2,\dots,n-1)\ge 0$ with $y\cdot(1,2,\dots, n-1)=n$ and 
	$$(1,\dots, 1)\cdot y=\lambda\sum_{j=1}^{n-1}j=\frac{\lambda(n-1)n}{2}=\frac{3n}{2n-1}\le 2\quad\text{ for }n\ge 2.$$  Consequently, $y\in\mathcal{C}(n)$ and 
		$\mu(n)\le |y|^2=\frac{6n}{(n-1)(2n-1)}$.
	This concludes the proof of \eqref{eq:mu(n)def}. It also proves the asserted optimality of the displayed constant.
\end{proof}
\subsection*{Appendix D}
This section is devoted to some discrete trace inequality for affine functions. The first estimate in \eqref{eq:Appendix_TraceInequality} is an equality for the constant function $f\equiv 1$ in any space dimension. The affine function $f$ with $f|_{F_j}\equiv 1$ on $F_j\in\mathcal{F}(K)$ and $f(P_j)=-n/2$ at the vertex $P_j\in\mathcal{N}(K)$ opposite to $F_j$ leads to an equality in the second estimate in \eqref{eq:Appendix_TraceInequality}. The third estimate in \eqref{eq:Appendix_TraceInequality} is an equality for the affine function with $f|_{F_j}=1$ and $f(P_j)=-(n+1)$. 
\begin{lemmaAppendix}\label{lem:AppendixIntmeanP1}
Let $K\subset\mathbb{R}^n$ be a simplex of positive volume $|K|$ with the set $\mathcal{F}(K)=\{F_0,F_1$\textup{,}\dots \textup{,} $ F_n\}$ of its sides and the set $\mathcal{N}(K)=\{P_0,P_1,\dots, P_n\}$ of its vertices. 
Then  any $f\in P_1(K)$ satisfies
		\begin{align}
					\Bigg\{\sum_{k=0}^n\Big\vert\intmean_{F_k} f\,\textup{d}s\Big\vert^2,\,\frac{n}{2}\max_{j=0,\dots,n}\Big\vert\intmean_{F_j} f\,\textup{d}s\Big\vert^2,\, 
					\max_{j=0,\dots,n}\frac{|f(P_j)|^2}{(n+1)}	 \Bigg\}
					\le & \frac{n+1}{|K|}\Vert f\Vert^2_{L^2(K)}. \label{eq:Appendix_TraceInequality}\tag{D1}
		\end{align}
\end{lemmaAppendix}
\begin{proof}
	Let $P_k$ denote the vertex opposite to the side $F_k$ in $K=\textup{conv}\{P_k,F_k\}$ and set  $x_k:=f(P_k)$ for $k=0,\dots n$.
	For an affine $f\in P_1(K)$ and a 
	side $F_k\in \mathcal{F}(K)$ the integral mean  $\intmean_{F_k}f\,\textup{d}s=f(\textup{mid}(F_k))=\sum_{\substack{j=0\\j\not = k}}^{n}x_j/n$  is rewritten 
	with $s:=\sum_{k=0}^n x_k$ and $|x|^2:=\sum_{k=0}^nx_k^2$, 
	\begin{align*}
		n^2 \sum_{k=0}^n\Big\vert\intmean_{F_k} f\,\textup{d}s\Big\vert^2
			&= \sum_{k=0}^n\Big(\sum_{\substack{j=0\\j\not = k}}^{n}x_j\Big)^2
			=\sum_{k=0}^n(s-x_k)^2=|x|^2+(n-1)s^2.
	\end{align*}
	The Cauchy-Schwarz inequality implies $s^2\le (n+1)\vert x\vert^2$ and so $(n-2)s^2\le (n^2-n-2)|x|^2$, which is equivalent to 
	\begin{align*}
				|x|^2+(n-1)s^2\le \frac{n^2}{n+2}(|x|^2+s^2).
	\end{align*}
	The combination with the local mass matrix for the $P_1$ conforming FEM  in \eqref{eq:MassMatrixP1}, namely
	\begin{align*}
				\Vert f\Vert_{L^2(K)}^2 =\frac{|K|}{(n+1)(n+2)}(\vert x\vert^2+s^2),
	\end{align*}
	concludes the proof of the first inequality in \eqref{eq:Appendix_TraceInequality}. $\hfill{\Box}\hspace{-1.5mm}$
	\bigskip\\
	Without loss of generality assume $j=0$ in the remaining estimates in \eqref{eq:Appendix_TraceInequality}.
	Let $x:=(x_0,x_1,\dots,x_n)\in\mathbb{R}^{n+1}$, to deduce 
	\begin{align*}
		n^2\Big\vert\intmean_{F_0} f\,\textup{d}s\Big\vert^2=\Big(\sum_{j=1}^n x_j\Big)^2=\vert x\cdot(0, 1,\dots, 1) \vert^2.
	\end{align*}
	For any SPD matrix $A\in\mathbb{R}^{(n+1)\times(n+1)}$ and any vector $y\in\mathbb{R}^{n+1}$ let $\Vert y\Vert_A^2=y\cdot A y$ denote the associated 
	norm. 
	If $M:=M(K)$ denotes the local mass matrix for the $P_1$ conforming FEM  from \eqref{eq:MassMatrixP1}, the Cauchy-Schwarz inequality implies 
	\begin{align*}
		\vert x\cdot (0, 1,\dots, 1)\vert\le\Vert (0, 1,\dots, 1)\Vert_{M^{-1}}\,\Vert x\Vert_M
		=\Vert (0, 1,\dots, 1)\Vert_{M^{-1}}\,\Vert f\Vert_{L^2(K)}.
	\end{align*}
	{An elementary calculation with the Sherman-Morisson formula shows $\Vert (0, 1,\dots, 1)\Vert_{M^{-1}}^2=2n (n+1)/|K|$. 
	The combination of this with the previous displayed formulas concludes the proof of the second inequality in \eqref{eq:Appendix_TraceInequality}.
	$\hfill{\Box}\hspace{-1.5mm}$}
	\bigskip\\
	{In the above notation the Cauchy-Schwarz  inequality implies  
	\begin{align*}
	  |f(P_0)|=\vert x\cdot (1,0,\dots, 0) \vert \le \Vert (1, 0,\dots, 0)\Vert_{M^{-1}}\,\Vert f\Vert_{L^2(K)}.
	\end{align*}
	This and  $\Vert (1, 0,\dots, 0)\Vert_{M^{-1}}^2=(n+1)^2/|K|$ prove the third inequality in \eqref{eq:Appendix_TraceInequality}. }
\end{proof}
Since each side $F$ of a $n$-simplex is a $(n-1)$-simplex, the point estimate in Lemma \ref{lem:AppendixIntmeanP1} translates to sides;  it coincides with the optimal estimate in Lemma \ref{lem:1DinverseEstimate} for   $n=2$. 
\begin{corollaryAppendix}\label{cor:AppendixP1inNode}
	Let   $F\in\mathcal{F}(K)$ be a side of a $n$-simplex $K\subset\mathbb{R}^n$  with vertex $P\in\mathcal{N}(F)$ and positive surface measure $|F|$,
	then any affine function $f\in P_1(F)$ satisfies 
		\begin{align*}
				|f(P)|^2	&\le\frac{n^2}{|F|}\Vert f\Vert_{L^2(F)}^2.
		\end{align*}	
\end{corollaryAppendix}		

\end{document}